\documentclass[11pt,oneside,a4papper]{osajnl}
\usepackage{amsthm}
\usepackage{comment}
\journal{jocn} 

\setboolean{shortarticle}{false}
\title{\centering 
Stackelberg-Nash strategy for the null controllability of semilinear degenerate equations in non-cylindrical domains}

\author[1,3,*]{\centering \textcolor{black}{Alfredo S. Gamboa}}
\author[1]{\textcolor{black}{Juan B. Limaco}}
\author[1,2]{\textcolor{black}{Luis P. Yapu}}

\affil[1]{Universidade Federal Fluminense, Instituto de Matemática e Estatística, Niterói, Brazil}
\affil[2]{Friedrich-Alexander Universität Erlangen-Nürnberg (FAU), Chair of Dynamics, Control, Machine Learning and Numerics, Erlangen, Germany}
\affil[3]{Universidade Estadual do Rio de Janeiro, Escola Politécnica, Nova Friburgo, Brazil}
\affil[*]{\centering Contact: alfredo.soliz@iprj.uerj.br}

 \dates{}

\doi{}

\begin{abstract}
In this paper we use a Stackelberg-Nash strategy to show the local null controllability of a semilinear parabolic equation in one-dimension defined in a non-cylindrical domain where the diffusion coefficient degenerates at one point of the boundary. The linearized degenerated system is treated using a Carleman inequality for degenerated non-autonomous systems proved by the autors in \cite{GYL-Carleman-2025} and the local controllability of the semi-linear system is obtained using Liusternik's inverse function theorem. 
\end{abstract}

\definecolor{mygreen}{RGB}{44,162,67}
\definecolor{mylilas}{RGB}{186,85,211}
\setboolean{displaycopyright}{false}
\newcommand{\dpar}[2]{\frac{\partial #1}{\partial #2}}

\newcommand{\rea}{\mathbb{R}}

\newcommand{\dpd}[2]{\frac{\partial^2 #1}{\partial #2^2}}

\newcommand{\cara}{\mathbb{1}}

\newtheorem{lema}{Lemma}
\newtheorem{teo}{Theorem}
\newtheorem{propo}{Proposition}
\newtheorem{coro}{Corollary}
\newtheorem{remark}{Remark}

\newcommand{\R}{\mathbb{R}}

\setboolean{displaycopyright}{false}

\begin{document}

\maketitle

\textbf{MSC Classification (2020)}: Primary: 35K65, 93B05; Secondary: 93C10. 

\textbf{keywords}: Degenerate parabolic equations, Non-cylindrical domains, Controllability, Nonlinear systems in Control Theory, Carleman inequalities.

\section{Introduction} \label{S:Intro}
 
Let $T>0$ be a fixed constant, let $\ell(t)$ be a smooth curve depending on time $t \in [0,T]$ which determines the non-cylindrical domain $\hat Q = \{ (x',t) \ | \ 0 < x' < l(t), \ t \in [0,T] \}$ with lateral boundary $\hat \Sigma = \{ (x',t) | x'=0 \text{ or } x'=l(t), \ t \in [0,T] \}$.

For each $t \in [0,T]$, we denote by $\tau_t$ a diffeomorphism from $(0,1)$ to $\hat \Omega_t = (0,\ell(t))$ and we denote by $Q=(0,1) \times (0,T)$ the associated cylindrical domain determined by the diffeomorphism $\tau : Q \to \hat Q$ which is defined at each $t$-slice by $\tau_t$. For more details see Section \ref{sec:diffeomorphism}.

We are interested in the controllability following degenerate semilinear parabolic equation in non-cylindrical domain, 

\begin{equation}\label{eq:PDE}
	\left\{\begin{aligned}
		&u_t - \left(a(x') u_{x'}\right)_{x'} + F(u,\beta(x')u_{x'}) = \hat h 1_{\hat O} + \hat v^1 1_{\hat O_1} + \hat v^2 1_{\hat O_2} &&\text{in}&& \hat Q,\\
		&u(0,t)=u(1,t)=0&&\text{on} && (0,T), \\
		&u(\cdot,0) = u_0 &&\text{in} && \hat \Omega_0,
	\end{aligned}
	\right.
\end{equation}
where $u_{0}$ is the initial data, $\hat h$ is the control of the leader, $\hat v^i$ are controls of the followers and $1_A$ denotes the characteristic function of the set $A$. $F$ is a $C^2$-function with bounded derivatives up to order 2.

Let $\hat O_t$ and $\hat O_{i,t}$ denote open intervals included in $\hat \Omega_t$ and $\hat O = \cup_t \hat O_t$, $\hat O_{i} = \cup_t \hat O_{i,t}$ such that $\hat O_i \cap \hat O = \emptyset$.
%For $i=1,2$, we denote $\mathcal{V}_{i}=L^{2}(\hat O_{i} \times (0,T))$ and $\mathcal{V}=L^{2}(\hat O \times (0,T))$. 

The function $a:=a(\cdot)$ represents the degenerate diffusion coefficient and satisfies the following:
%We take degeneration functions $a(x')=a(\tau_t(x))$ such that the form $h(t)a(x)$ with $x \in (0,1)$. In Section \ref{sec:diffeomorphism} the diffeomorphism $\tau_t$ has a specific form which verify this condition. 
$a \in C([0,1]) \cap C^1((0,1])$ with $a(0) = 0, a > 0 \text{ in } (0,1], a' \ge 0$ and  $xa'(x)\leq K a(x), \forall x \in [0,1] \text{ and some } K \in [0,1)$. Moreover, we suppose that $a$ verifies  
$$a(x \cdot y)=a(x)a(y).$$

Therefore, the function $a$ behaves like  $x^\alpha$ with $\alpha \in (0,1)$. This condition is called  \emph{weakly degenerate} in \cite{Alabau_cannarsa_fragnelli-06}.
Under these hypotheses, the function $\frac{x}{\sqrt{a}}$ is nondecreasing and thus is bounded from above by $\frac{1}{\sqrt{a(1)}}$.

%\color{red}
On the other hand, the function $\beta(x)$ verifies the following, 
\begin{equation}\label{eq:cond_beta}
    \beta(x)^2 \leq a(x),\quad (\beta(x)^2)_x \leq 2 a(x)_x\quad and \quad \beta(x)_x \leq M.
\end{equation}

For instance one can consider $\beta$ of the form $\beta(x)=x^b$, with $b>1$, for $x$ near the degeneration point $x=0$. 
%\color{black}

Moreover, let us consider the sets $\hat O_{i,d} = \cup_t \hat O_{i,d,t}$ which represent the observation domains of the followers, where $\hat O_{i,d,t}$, $i=1,2$, are open subsets of $\hat \Omega_t$. We consider the functionals
\begin{equation}\label{eq:def_Ji}
\hat J_i(\hat h,\hat v^1,\hat v^2) = \frac{\alpha_i}{2} \int_{\hat O_{i,d}} |u-u_{i,d}|^2 |Jac (\tau)|^{-1} dxdt + \frac{\mu_i}{2} \int_{\hat O_i} |\hat v^i|^2 |Jac (\tau)|^{-1} dxdt,
\end{equation}
where $\alpha_i>0$, $\mu_i>0$ and $u_{i,d}(x,t)$ are given functions, for $i=1,2$. 
%\color{red}
We observe that one can also define $\hat J_i$ without the weight $|Jac (\tau)|^{-1}$, see Remark \ref{rmk:without_Jac} in Section \ref{sec:characterization_convexity}.
%\color{black}

For a fixed control $\hat h\in L^{2}(\hat O \times (0,T))$, the controls $\hat v^i \in L^{2}(\hat O_{i} \times (0,T))$, $i=1,2$, are a \emph{Nash equilibrium} for the functionals $\hat J_1$ and $\hat J_2$ if
\begin{equation}\label{minimousfunctional}
\left\{\begin{aligned}
&\hat J_1(\hat h,\hat v^1,\hat v^2) = \min_{\tilde v^1\in L^{2}(\hat O_{1} \times (0,T))} \hat J_1(\hat h,\tilde v^1, \hat v^2), \\
&\hat J_2(\hat h,\hat v^1,\hat v^2) = \min_{\tilde v^2\in L^{2}(\hat O_{2} \times (0,T))} \hat J_1(\hat h,\hat v^1,\tilde v^2).
\end{aligned}
\right.
\end{equation}

For a linear equation, the functionals in \eqref{eq:def_Ji} are convex and, in that case,  \eqref{minimousfunctional} is equivalent to
\begin{equation}\label{derivateoffunctional}
\left\{\begin{aligned}
&\hat J_1'(\hat h,\hat v^1,\hat v^2)(\tilde v^1,0) = 0, \quad \forall \hat v^1 \in L^{2}(\hat O_{1} \times (0,T)), \\
&\hat J_2'(\hat h,\hat v^1,\hat v^2)(0,\tilde v^2) = 0, \quad \forall \hat v^2 \in L^{2}(\hat O_{2} \times (0,T)).
\end{aligned}
\right.
\end{equation}

For our nonlinear equation we lose the convexity of the functional and the Nash equilibrium condition \eqref{minimousfunctional} is not necessarily equivalent to \eqref{derivateoffunctional}.
%For this reason, it is convenient to weaken the definition of equilibrium. 
For a fixed control $\hat h \in L^{2}(\hat O \times (0,T))$, one says that the controls $\hat v^{i}\in L^{2}(\hat O_{i} \times (0,T))$, for $i=1,2$, are a \emph{Nash quasi-equilibrium} for the functionals $\hat J_i$ if they satisfy the condition \eqref{derivateoffunctional}.

Once the Nash equilibrium has been identified for any $\hat h$, given $T>0$ we look for a control $\hat h \in L^{2}(\hat O \times (0,T))$ such that the system is \emph{null controllable} at time $T$, i.e. 
$$
u(x,T) = 0, \text{ for } x \in \hat \Omega_T.
$$

The hierarchical control approach known as the Stackelberg-Nash strategy has been widely studied in the context of partial differential equations, particularly in control problems for evolution equations. Its origin lies in game theory, where scenarios are modeled in which a leader makes strategic decisions before one or more followers, who then respond accordingly. In the context of the control of evolution equations, this strategy was first applied by J.-L. Lions in \cite{Lions-94-PMS} and \cite{Lions-94-MM}, laying the groundwork for subsequent studies that have extended its applicability to other systems \cite{HernandezSantamaria-Peralta-20, HernandezSantamaria-deTeresa-18, fragnelli-20, fragnelli-18, Djomegne-Deugou-21}.

One of the first advances in the framework of exact controllability using this strategy was presented in \cite{Araruna-EFC-Santos-15}, where the hierarchical control of a class of parabolic equations was studied, both in the linear and semilinear cases. Subsequently, in \cite{Araruna-EFC-Guerr-Santos-17}, these results were refined by imposing weaker conditions on the observation domains of the followers, allowing for a broader range of applicability of the method. In the context of wave equations, a similar approach was developed in \cite{Araruna-EFC-daSilva-18}, highlighting the versatility of this strategy beyond the parabolic regime.

In \cite{Araruna-EFC-daSilva-20}, the problem of exact control in parabolic equations with distributed and boundary controls was addressed, providing a deeper understanding of the interaction between different levels of control in systems with spatial and physical constraints. Additionally, for coupled systems, Kéré et al. \cite{frances25} introduced a bi-objective control strategy for a system of coupled parabolic equations with finite constraints on one of the states, establishing a methodology that balances multiple control objectives within a hierarchical framework. Similarly, in \cite{frances22} and \cite{frances23}, the Stackelberg-Nash strategy was studied for cascade systems of parabolic equations, providing analytical tools to understand the interaction between leaders and followers in more complex control structures.

The application of this strategy to nonlinear systems has been the subject of recent research. In \cite{frances33}, Limaco et al. applied hierarchical control to a coupled quasi-linear parabolic system with controls acting inside the domain, highlighting the additional challenges that arise in the presence of nonlinear terms. In a different context, in \cite{frances34}, the strategy was extended to nonlinear parabolic equations in unbounded domains, addressing key aspects of controllability in scenarios without finite spatial constraints. Furthermore, in \cite{frances24}, Huaman applied this strategy to the control of quasi-linear parabolic equations in dimensions 1, 2, and 3, consolidating the applicability of the method in multidimensional contexts.

More recently, in \cite{frances9}, the implementation of hierarchical control in the anisotropic heat equation with dynamic boundary conditions and drift terms has been explored, further extending the scope of the Stackelberg-Nash strategy in problems with complex boundary conditions. However, all these studies have exclusively considered non-degenerate systems. To the best of our knowledge, the only works that address the hierarchical strategy in degenerate parabolic equations are those presented in \cite{ararunano} and \cite{francescuate}, where linear and semilinear cases are studied, with the latter relying on a Carleman inequality by Alabau-Boussouira, Cannarsa, and Fragnelli \cite{Alabau_cannarsa_fragnelli-06}, introducing first-order terms in the nonlinearity. However, these works deal with problems in fixed domains that do not depend on time. Therefore, the present work is the first to consider moving domains, which generates a control problem for a non-autonomous degenerate parabolic equation ans uses the Carleman estimate demonstrated by the authors in \cite{GYL-Carleman-2025}.

Our first main result shows the existence of Nash quasi-equilibrium $(\hat v_1,\hat v_2)$ for $\hat J_1$ and $\hat J_2$, such that we obtain null controllability of our equation \eqref{eq:PDE} for small initial conditions (i.e. \emph{local controllability}). Since the equation is degenerate, the initial condition $u_0$ is taken in the weighted space $H_a^1(\hat \Omega)$,  defined in \cite{Alabau_cannarsa_fragnelli-06} for cylindrical domains; see Section \ref{Sec:null_linearized_system}.

\begin{teo}
\label{thm:local_null_controllability}
Let us assume the hypothesis considered in the setting of  \eqref{eq:PDE} and that, for $i=1,2$,
$$
\hat O_{i,d} \cap \hat O \neq \emptyset
$$
and
\begin{equation}\label{ec9}
\hat O_d := \hat O_{1,d} = \hat O_{2,d}.    
\end{equation}

Then, for any $T>0$ there exist $\varepsilon>0$ and a positive function $\rho(t)$ blowing up at $t=T$ such that, if $\rho y_{i,d} \in L^2(O_d \times (0,T))$ and $u_0 \in H_a^1(\Omega_0)$ verifies
$$
\|u_0\|_{H_a^1(\hat \Omega_0)} \leq \varepsilon, 
$$
then there exist a control $\hat h \in L^2(\hat O\times(0,T))$ and associated Nash quasi-equilibrium $(\hat v^1,\hat v^2)$ such that the solution of \eqref{eq:PDE} is null-controllable at time $T$.
\end{teo}

Our next result shows conditions under which the Nash quasi-equilibrium $(\hat v_1,\hat v_2)$ obtained in Theorem \ref{thm:local_null_controllability} is a Nash equilibrium.

\begin{propo} 
\label{thm:nash_equilibrium}
Let $u_0 \in H_a^1(\hat \Omega_0)$.
Under the hypotheses of Theorem \ref{thm:local_null_controllability}, 
%\textcolor{red}{there exists a constant $\delta_0 > 0$ such that if
%$$
%\|y^0\|_{H_a^1(\Omega)} \leq \delta_0,
%$$}
if $\mu_1,\mu_2$ are sufficiently large, then the Nash quasi-equilibrium  $(\hat v^1,\hat v^2)$ obtained in Theorem \ref{thm:local_null_controllability} is a Nash equilibrium associated to the  equation \eqref{eq:PDE} for the functionals defined in \eqref{eq:def_Ji}.
\end{propo}

\noindent{\bf Outline of the paper}.
In Section \ref{sec:diffeomorphism} we write our problem as an equation with time variable coefficients on a cylindrical domain using a diffeomorphism. 
In Section \ref{sec:characterization_convexity} we characterize Nash quasi-equilibrium and we prove Proposition \ref{thm:nash_equilibrium} implying the convexity of the functionals and the fact that the Nash quasi-equilibrium are Nash equilibrium. In Section \ref{Sec:null_linearized_system} we show the null controllability of the linearized system using a Carleman estimate for a system of non-autonomous degenerated parabolic equations proved in \cite{DemarqueLimacoViana_deg_sys2020}. The main result of this section is Theorem \ref{theorem case linear} which besides the null controllability establishes additional estimates needed for the controllability of the nonlinear system. In Section \ref{sec:control for nonlinear system} we apply Liusternik inverse function theorem to show the local null controllability of \eqref{eq:PDE}. The proof of the main result, Theorem \ref{thm:local_null_controllability}, is a consequence of Lemmas \ref{A bem definido}-\ref{Mapa sobrejetivo} where we verify the hypothesis of Liusternik's theorem and apply it to a suitable map whose local surjectivity implies the local controllability result.
In Section \ref{sec:final_remarks} we present some related problems. 
Finally, in Appendix \ref{appendix A} we sketch the well-possedness of the optimality system obtained in Section \ref{sec:characterization_convexity}.
%and in Appendix \ref{appendix B} we show a Carleman inequality for our degenerate equation with bounded and time-dependent coefficients having a specific form. 

\section{Change of variables}
\label{sec:diffeomorphism}

Let $\Omega=(0,1)$ be a nonempty bounded connected open set, with boundary $\Gamma=\partial\Omega$, $T > 0$. We denote by $Q$ the cylinder $Q = \Omega\times (0, T)$, with lateral boundary $\Sigma=\Gamma\times (0, T)$. Let us consider a family of functions $\{\tau_t\}_{0\leq t\leq T}$ , where for each $t$, $\tau_t$ is a diffeomorphism deforming $\Omega$ into an open bounded set $\Omega_t$ of $\mathbb{R}$ defined by
$$\Omega_t=\{ x'\in \rea \ | \ x'=\tau_t(x), \ \ \text{for} \ \ x\in\Omega \}.$$

For $t = 0$, we identify $\Omega_0$ with $\Omega$ so that $\tau_0$ is the identity mapping. The smooth boundary of  $\Omega_t$ is denoted by $\Gamma_t$. The non-cylindrical domain $\widehat{Q}$ and its lateral boundary $\widehat{\Sigma}$ are defined by 
 $$\widehat{Q}=\bigcup_{0\leq t\leq T}\{\Omega_t\times\{t \}\}, \ \ \ \ \ \ \ \ \ \ \ \ \ \widehat{\Sigma}=\bigcup_{0\leq t\leq T}\{ \Gamma_t\times\{t\}  \}$$
respectively.

Thus we have a natural diffeomorphism $\tau : Q\rightarrow \widehat{Q}$ defined by
$$(x,t)\in Q \rightarrow (x',t)\in \widehat{Q}, \ \ \ \ \text{where} \ \ \ x'=\tau_t(x).$$

We assume the following regularity for the maps $\tau$ and  $\tau_t$, $0\leq t\leq T$:
\begin{description}
	\item[(R1)] $\tau_t$ is a  $C^2$ diffeomorphism from $\Omega$ to $\Omega_t$,
	\item[(R2)] $\tau$ belong to  $C^1([0,T];C^0(\overline{\Omega}))$.
\end{description}

\noindent We are interested in the following 
degenerate parabolic equation defined in the non-cylindrical domain $\hat Q$,
\begin{equation}\label{eq1a}
	\begin{cases}
		u_t-({a}(x')u_{x'})_{x'}+F\left(u,\beta(x')u_{x'}\right)=\widehat{h}\cara_{_{\widehat{\mathcal{O}}}}+\widehat{v}^1\cara_{_{\widehat{\mathcal{O}}_1}}+\widehat{v}^2\cara_{_{\widehat{\mathcal{O}}_2}}, & \ \ \ \text{in} \ \ \ \widehat{Q},\\
		u=0, & \ \ \ \text{on} \ \ \ \widehat{\Sigma},\\
		u(0)=u_0(x'), & \ \ \ \text{in} \ \ \ \Omega_t,
	\end{cases}
\end{equation}
where $u_0$ is the initial data, $\widehat{h}$ is the control of the leader, $\widehat{v}^i$ are controls of the followers and $\cara_{_{A}}$
denotes the characteristic function of the set $A$.

Let $\widehat{\mathcal{O}}\subset \widehat{Q}$ denote the domain of the main control  $\widehat{f}(x',t)$ (leader).

Let $\widehat{\mathcal{O}}_1, \widehat{\mathcal{O}}_2 \subset \widehat{Q}$ denote the domains of the secondary controls  $\widehat{v}^1(x',t)$ and $\widehat{v}^2(x',t)$ (followers).

Let $\widehat{\mathcal{O}}_{1,d}, \ \widehat{\mathcal{O}}_{2,d} \subset  \widehat{Q}$ denote open sets indicating the observation domains of the followers.

Let us consider the functionals associated to each follower:
\begin{eqnarray}\label{ec2a}    \widehat{J}_1(\widehat{h},\widehat{v}^1,\widehat{v}^2)&=&\frac{\alpha_1}{2}\int_{\widehat{\mathcal{O}}_{1,d}\times(0,T)}|u-u_{1,d}|^2\ \left| Jac(\tau_t)\right|^{-1}dx' dt+\frac{\mu_1}{2}\int_{\widehat{\mathcal{O}}_{1}\times(0,T)}|\widehat{v}^1|^2\ \left| Jac(\tau_t)\right|^{-1} dx' dt\nonumber,\\	\widehat{J}_2(\widehat{h},\widehat{v}^1,\widehat{v}^2)&=&\frac{\alpha_2}{2}\int_{\widehat{\mathcal{O}}_{2,d}\times(0,T)}|u-u_{2,d}|^2\  \left| Jac(\tau_t)\right|^{-1} dx' dt+\frac{\mu_2}{2}\int_{\widehat{\mathcal{O}}_{2}\times(0,T)}|\widehat{v}^2|^2\  \left| Jac(\tau_t)\right|^{-1} dx' dt\nonumber,
\end{eqnarray}
where $\alpha_i>0$, $\mu_i>0$ are constants and $v_{i,d}=v_{i,d}(x',t)$ are given functions, with $i=1,2$.

Using the diffeomorphism  $\tau : Q\rightarrow \widehat{Q}$
%$$\Omega_t=\{ x'\in \rea  \ | \ x'=\tau_t(x), \ \ \text{para} \ \ x\in\Omega \}$$
%como $\tau_t : Q\rightarrow \widehat{Q}$ 
the domains are transformed in the following way:
$$\mathcal{O} \rightarrow \widehat{\mathcal{O}}, \qquad \mathcal{O}_1 \rightarrow \widehat{\mathcal{O}}_1, \qquad \mathcal{O}_2 \rightarrow \widehat{\mathcal{O}}_2,\qquad \mathcal{O}_{id} \rightarrow \widehat{\mathcal{O}}_{id}, \ \ i=1,2, $$
and the functions have the form
$$\widehat{h}(x',t)=h(\tau_t(x),t), \ \ \ \widehat{v}^i(x',t)=v^i(\tau_t(x),t)$$
$$\mathbb{1}_{\widehat{\mathcal{O}}}(x',t)=\mathbb{1}_{{\mathcal{O}}}(\tau_t(x),t), \ \ \ \ \mathbb{1}_{\widehat{\mathcal{O}}_i}(x',t)=\mathbb{1}_{{\mathcal{O}}_i}(\tau_t(x),t).$$

%\noindent We consider the degeneration function $a$ as having the property   ${a}(x')={a}(\ell(t)x)=a(\ell(t))\cdot a(x)$. 

The function $u$ is transformed into a function $y$ in the cylindrical domain $Q$ such that $$u(x',t)=y(x,t)=y\left(\tau_t^{-1}(x'),t\right).$$

Moreover, we use the notation $\tau_t^{-1}=\psi_t$ or $\psi(x,t)=\psi_t(x)$ and we have $u(x',t)=y(x,t)=y\left(\psi_t(x'),t\right)$.

The following formulas allow us to transform the equation to the in the cylindrical domain $Q$. 
$$\dpar{u}{x'}=\dpar{y}{x}\dpar{x}{x'}= \dpar{y}{x}\dpar{\psi}{x'}\ \ \ \ \Rightarrow \ \ \ \ \ \ \dpar{u}{x'}=u_{x'}=y_x\dpar{\psi}{x'}(\tau_t(x),t), $$
Thus,
\begin{equation*}
    \begin{split}
        \dpar{}{x'}\left( {a}(x') \dpar{u}{x'} \right) &={a}(x')\dpar{\psi}{x'}\dpar{ }{x'}\left(\dpar{y}{x} \right)+\dpar{y}{x}\dpar{ }{x'}\left({a}(x')\dpar{\psi}{x'} \right)\\
        &={a}(x')\dpar{\psi}{x'}\dpar{ }{x}\left(\dpar{y}{x} \right)\dpar{\psi}{x'}+\dpar{y}{x}\dpar{ }{x'}\left({a}(x')\dpar{\psi}{x'} \right),    
    \end{split}
\end{equation*}
or, equivalently,
\begin{equation*}
    \begin{split}
        ({a}(x')u_{x'})_{x'} &=\left( \dpar{\psi}{x'}(\tau_t(x),t) \right)^2{a}(x')y_{xx}+y_x\dpar{ }{x'}\left({a}(x')\dpar{\psi}{x'} \right) \\
        &=\left( \dpar{\psi}{x'}(\tau_t(x),t) \right)^2{a}(\tau_t(x))y_{xx}+y_x\left({a}'(\tau_t(x))\dpar{\psi}{x'}+{a}(\tau_t(x))\dpd{\psi}{x'}(\tau_t(x),t) \right).
    \end{split}
\end{equation*}
Moreover, 
$$\dpar{u}{t}=\dpar{u}{t}+\dpar{u}{x}\dpar{x}{t}=\dpar{y}{t}+\dpar{y}{x}\dpar{\psi}{t}=
y_t+\left(\dpar{\psi}{t}(\tau_t(x),t)\right)y_x.$$

Therefore, applying the diffeomorphism $\tau$, \eqref{eq:PDE} is transformed into the following equation in the cylindrical domain $Q$:
\begin{equation}\label{ec1p}
\hspace*{-1.2cm}	
    \begin{cases}
		y_t-\left( \dpar{\psi}{x'}(\tau_t(x),t) \right)^2{a}(\tau_t(x))y_{xx}+\left(\dpar{\psi}{t}(\tau_t(x),t)-{a}'(\tau_t(x))\dpar{\psi}{x'}-{a}(\tau_t(x))\dpd{\psi}{x'}(\tau_t(x),t) \right)y_x\\
		\hspace*{6cm}+F\left(y,\frac{\partial\psi}{\partial x'}\beta(\tau_t(x))y_x\right)={h}\cara_{_{{\mathcal{O}}}}+{v}^1\cara_{_{{\mathcal{O}}_1}}+{v}^2\cara_{_{{\mathcal{O}}_2}}, & \ \ \ \text{in} \ \ \ {Q},\\
		y=0, & \ \ \ \text{on} \ \ \ {\Sigma},\\
		y(0)=v_0(\tau_0(x))=y_0, & \ \ \ \text{in} \ \ \ \Omega.
	\end{cases}
\end{equation}

For our one-dimensional problem, we use the following diffeomorphism given by rescaling,
$x=\tau_t^{-1}(x')=\psi(x',t)=\frac{x'}{\ell(t)}$
sending 
$$\Omega_t=\{x'\in\mathbb{R} \ | \ 0<x' <\ell(t)\} \qquad \text{to} \qquad \Omega=\{x \in \mathbb{R} \ | \ 0< x<1 \}.$$

Then, we have
$$\dpar{\psi}{x'}(\tau_t(x),t)=\frac{1}{\ell(t)}, \ \ \ \ \ \dpd{\psi}{x'}(\tau_t(x),t)=0, \ \ \ \ \ \ \dpar{\psi}{t}(\tau_t(x),t)=-x'\frac{\ell'(t)}{\ell(t)^2}=-x\ell(t)\frac{\ell'(t)}{\ell(t)^2}=-\frac{\ell'(t)}{\ell(t)}x$$

\begin{equation}\label{ec1b}
	%\hspace*{-1.2cm}	
    \begin{cases}
		y_t-\frac{1}{\ell(t)^2}{a}(\tau_t(x))y_{xx}-\frac{1}{\ell(t)}\left(\ell'(t)x+{a}'(\tau_t(x)) \right)y_x+F\left(y,\frac{1}{\ell(t)}\beta(\ell(t)x)y_x\right)\\
        \quad={h}\cara_{_{{\mathcal{O}}}}+{v}^1\cara_{_{{\mathcal{O}}_1}}+{v}^2\cara_{_{{\mathcal{O}}_2}}, & \ \ \ \text{in} \ \ \ {Q},\\
		y=0, & \ \ \ \text{on} \ \ \ {\Sigma},\\
		y(0)=v_0(\tau_0(x))=y_0, & \ \ \ \text{in} \ \ \ \Omega.
	\end{cases}
\end{equation}

We have the identities,
$$\left({a}(x')y_x\right)_x= {a}(x')y_{xx}+{a}'(x')\dpar{x'}{x}y_x, \ \ \ \rightarrow \ \ \ \ {a}(x')y_{xx}=\left({a}(x')y_x\right)_x-{a}'(x')\ell(t)y_x,$$
Then,
$$-\frac{1}{\ell(t)^2}{a}(\tau_t(x))y_{xx}=-\frac{1}{\ell(t)^2}\left({a}(\tau_t(x))y_x\right)_x+\frac{1}{\ell(t)}{a}'(\tau_t(x))y_x.$$

The degeneration function $a$ has the property  ${a}(x')={a}(\ell(t)x)=a(\ell(t))\cdot a(x)$.
Thus,
\begin{equation}\label{eqprin}
	\hspace*{-1.2cm}	\begin{cases}
		y_t-\frac{1}{\ell(t)^2}\left({a}(\ell(t)x)y_x\right)_x-\frac{\ell'(t)}{\ell(t)}xy_x+F\left(y,\frac{1}{\ell(t)}\beta(\ell(t)x)y_x\right)={h}\cara_{_{{\mathcal{O}}}}+{v}^1\cara_{_{{\mathcal{O}}_1}}+{v}^2\cara_{_{{\mathcal{O}}_2}}, & \ \ \ \text{in} \ \ \ {Q},\\
		y=0, & \ \ \ \text{on} \ \ \ {\Sigma},\\
		y(0)=v_0(\tau_0(x))=y_0, & \ \ \ \text{in} \ \ \ \Omega,
	\end{cases}
\end{equation}
and we use the notations $b(t)=\frac{a(\ell(t))}{\ell(t)^2}, \ \ B(t)=\frac{\ell'(t)}{\ell(t)}, \ \ C(t)=\frac{\beta(\ell(t))}{\ell(t)}$ to write the latter as
\begin{equation}\label{eqprin}
	\hspace*{-1.2cm}	\begin{cases}
		y_t-b(t)\left({a}(x)y_x\right)_x-B(t)xy_x+F\left(y,C(t)\beta(x)y_x\right)={h}\cara_{_{{\mathcal{O}}}}+{v}^1\cara_{_{{\mathcal{O}}_1}}+{v}^2\cara_{_{{\mathcal{O}}_2}}, & \ \ \ \text{in} \ \ \ {Q},\\
		y=0, & \ \ \ \text{on} \ \ \ {\Sigma},\\
		y(0)=v_0(\tau_0(x))=y_0, & \ \ \ \text{en} \ \ \ \Omega.
	\end{cases}
\end{equation}

\section{CHARACTERIZATION OF NASH QUASI-EQUILIBRIUM AND EQUILIBRIUM}
\label{sec:characterization_convexity}

We use the condition \eqref{ec9} on the sets $\mathcal{O}_{i,d}$. 
%\begin{equation}\label{ec9}
%\mathcal{O}_{1,d}=\mathcal{O}_{2,d}=:\mathcal{O}_{d},
%\end{equation}
%i.e. $\mathcal{O}_{d}$ will denote both sets. 
%Assim, denotaremos esses conjuntos por $\mathcal{O}_{d}$, ver abaixo, na Seção 5, alguns comentários sobre a necessidade da hipótese.\\
In the linear case ($F=0$), the trajectory exact controllability is equivalent to the property of null controllability.
%a controllabilidad exata das trajetórias é equivalente to the property of controllabilidad nula. Vale o seguinte resultado:

\subsection{\bf Caracterization of the Nash equilibrium}
We first compute $J'_1(h,v^1,v^2)(\widetilde{v}^1,0)=\left.\frac{d}{d\lambda}J_1(h,v^1+\lambda\widetilde{v}^1,v^2)\right|_{\lambda=0}$. We have
$$J_1(h,v^1+\lambda\widetilde{v}^1,v^2)=\frac{\alpha_1}{2}\int_{\mathcal{O}_{1,d}\times(0,T)}|y^\lambda-y_{1,d}|^2\ dx dt+\frac{\mu_1}{2}\int_{\mathcal{O}_{1}\times(0,T)}|v^1+\lambda\widetilde{v}^1|^2\ dx dt,$$
where $y^\lambda$ is the solution of (in the following computations $B$ and $C$ are functions depending only on $t$),
\begin{equation}\label{eqprin2}
	\hspace*{-1.2cm}	\begin{cases}
		y^\lambda_t-b(t)\left(a(x)y^\lambda_x\right)_x-Bxy^\lambda_x+F\left(y^\lambda,C\beta(x)y_x^\lambda\right)={h}\cara_{_{{\mathcal{O}}}}+(v^1+\lambda\widetilde{v}^1)\cara_{_{{\mathcal{O}}_1}}+{v}^2\cara_{_{{\mathcal{O}}_2}}, & \ \ \ \text{in} \ \ \ {Q},\\
		y^\lambda=0, & \ \ \ \text{on} \ \ \ {\Sigma},\\
		y^\lambda(0)=y_0, & \ \ \ \text{in} \ \ \ \Omega.
	\end{cases}
\end{equation}

By continuity with respect to the initial data, we have  $y^\lambda\rightarrow y$, as $\lambda\rightarrow 0$. We denote $w^1=\lim\limits_{\lambda\rightarrow 0}\frac{y^\lambda-y}{\lambda}$.
Subtracting (\ref{eqprin}) from (\ref{eqprin2}), dividing by $\lambda$, Making $\lambda\rightarrow 0$, we get
\begin{equation}\label{eqprin3}
	\begin{cases}
		w^1_t-b(t)\left(a(x)w^1_x\right)_x-Bxw^1_x+D_1 F\left(y,C\beta(x)y_x\right)w^1 + C D_2 F\left(y,C\beta(x)y_x\right)\beta(x)w^1_x =\widetilde{v}^1\cara_{_{{\mathcal{O}}_1}}, & \ \ \ \text{in} \ \ \ {Q},\\
		w^1=0, & \ \ \ \text{on} \ \ \ {\Sigma},\\
		w^1(0)=0, & \ \ \ \text{in} \ \ \ \Omega.
	\end{cases}
\end{equation}

The adjoint system is
\begin{equation}\label{eqprin4}
	\begin{cases}
		-p^1_t-b(t)\left(a(x)p^1_x\right)_x+B \left(xp^1\right)_x + D_1 F\left(y,C\beta(x)y_x\right)p^1 - C\left(D_2 F\left(y,C\beta(x)y_x\right)\beta(x)p^1\right)_x \\
        \quad = \alpha_1(y-y_{1d})\cara_{_{{\mathcal{O}}_{1d}}}, & \ \ \ \text{in} \ \ \ {Q},\\
		p^1=0, & \ \ \ \text{on} \ \ \ {\Sigma},\\
		p^1(0)=0, & \ \ \ \text{in} \ \ \ \Omega.
	\end{cases}
\end{equation}

Multiplying the first equation of (\ref{eqprin4}) by $w^1$ and integrating in $Q$ we get

$$\int_{Q}\widetilde{v}^1\cara_{_{\mathcal{O}_1}}p^1dxdt=\int_{Q}\alpha_1(y-y_{1d})\cara_{\mathcal{O}_{1d}}w^1dxdt, \ \ \ \Rightarrow \ \ \ \ \int_{\mathcal{O}_{1}\times (0,T)}\widetilde{v}^1p^1dxdt=\int_{\mathcal{O}_{1,d}\times (0,T)}\alpha_1(y-y_{1d})w^1dxdt.$$

On the other hand,
$$J'_1(h,v^1,v^2)=\alpha_1\int_{\mathcal{O}_{1,d}\times (0,T)}(y-y_{1d})w^1dxdt+\mu_1\int_{\mathcal{O}_{1}\times (0,T)}v^1\widetilde{v}^1dxdt=0,
\ \ \ \ \ \forall \hat{v}^1\in \mathcal{H}_1$$
and comparing the last two expression we get
$$ \int_{\mathcal{O}_{1}\times (0,T)}\widetilde{v}^1p^1dxdt=-\mu_1\int_{\mathcal{O}_{1}\times (0,T)}v^1\widetilde{v}^1dxdt,
\ \ \ \ \ \forall \widetilde{v}^1\in \mathcal{H}_i,$$
from which
$$v^1=-\frac{1}{\mu_1}p^1\cara_{_{\mathcal{O}_1}}.$$

Analogously, computing $J'_2(h,v^1,v^2)$, we get $v^2=-\frac{1}{\mu_2}p^2\cara_{_{\mathcal{O}_2}}$ , where $p^i, i=1,2$, verify the
following optimality system:
\begin{equation}\label{eq:optimality_system}
	\begin{cases}
		y_t-b(t)\left(a(x)y_x\right)_x-B xy_x+F\left(y,C\beta(x)y_x\right)={h}\cara_{_{{\mathcal{O}}}}-\frac{1}{\mu_1}p^1\cara_{_{\mathcal{O}_1}}-\frac{1}{\mu_2}p^2\cara_{_{\mathcal{O}_2}}, & \ \ \ \text{in} \ \ \ {Q}\\
		-p^i_t-b(t)\left(a(x)p^i_x\right)_x+B \left(xp^i\right)_x + D_1F\left(y,C\beta(x)y_x\right)p^i - C\left(D_2 F(y,C\beta(x)y_x)\beta(x)p^i \right)_x \\
        \quad = \alpha_i(y-y_{id})\cara_{_{{\mathcal{O}}_{id}}}, & \ \ \ \text{in} \ \ \ {Q},\\
		y=0, \ \ \ \ \ p^1=0, \ \ \ \ \ p^2=0& \ \ \ \text{on} \ \ \ \Sigma,\\
		y(0)=y_0, \ \ \ \ \ p^1(T)=0, \ \ \ \ \ \ p^2(T)=0& \ \ \ \text{in} \ \ \ \Omega.
	\end{cases}
\end{equation}

The existence and uniqueness of the system \eqref{eq:optimality_system} is obtained in the following proposition:

\begin{propo}
Let $y_0 \in H^1_a(\Omega)$,  $h \in L^2(0,T,L^2(O))$ and, for $i=1,2$, $y_{i,d} \in L^2(0,T,L^2(O_{i,d}))$. Then the system \eqref{eq:optimality_system} has a unique solution $(y,p^1,p^2)$ satisfying
\begin{equation*}
\begin{split}
\|(y,p^1,p^2)\|_{L^2(0,T,L^2(\Omega)) \cap H^1(0,T,H^2_a(0,1))} \leq C &\left(\|y_0\|_{H^1_a(\Omega)} + \|h\|_{L^2(0,T,L^2(O))} + \sum_{i=1}^2 \|y_{i,d}\|_{L^2(0,T,L^2(O_{i,d}))} \right),
\end{split}    
\end{equation*}
where $C=C(\Omega)$ is a positive constant independent of $\mu_1$ and $\mu_2$.
\end{propo}

A special case of this proposition for a system two equations is given in \cite{GYL-Carleman-2025}. For completeness we sketch the proof for our system in Appendix \ref{appendix A}.

\begin{remark}\label{rmk:without_Jac}
If one defines the functionals without the Jacobian, i.e., 
\begin{equation*}%\label{eq:def_Ji}
\hat J_i(\hat h,\hat v^1,\hat v^2) = \frac{\alpha_i}{2} \int_{\hat O_{i,d}} |u-u_{i,d}|^2 dxdt + \frac{\mu_i}{2} \int_{\hat O_i} |\hat v^i|^2 dxdt,
\end{equation*}
after the coordinate transformation one gets
\begin{equation*}
J_i(h,v^1,v^2)=\frac{\alpha_i}{2}\int_{\mathcal{O}_{i,d}\times(0,T)}|y-y_{i,d}|^2 |Jac (\tau)| \ dx dt+\frac{\mu_i}{2}\int_{\mathcal{O}_{i} \times(0,T)}|v^i|^2 |Jac (\tau)| \ dx dt.   
\end{equation*}
Then, performing the same calculations performed above for the characterization of the Nash equilibrium,
$$J'_1(h,v^1,v^2)=\alpha_1\int_{\mathcal{O}_{1,d}\times (0,T)}(y-y_{1d})w^1dxdt+\mu_1\int_{\mathcal{O}_{1}\times (0,T)}v^1\widetilde{v}^1dxdt=0,
\ \ \ \ \ \forall \hat{v}^1\in \mathcal{H}_1,$$
and using the following modification of the adjoint equation,
\begin{equation*}
\begin{split}
-p^1_t-b(t)\left(a(x)p^1_x\right)_x+B \left(xp^1\right)_x + D_1 F\left(y,C\beta(x)y_x\right)p^1 - C\left(D_2 F\left(y,C\beta(x)y_x\right)\beta(x)p^1\right)_x \\
        \quad = \alpha_1(y-y_{1d})\cara_{_{{\mathcal{O}}_{1d}}} |Jac (\tau)|, & \ \ \ \text{in} \ \ \ {Q},   
\end{split}
\end{equation*}
we get analogously, for $i=1,2$,
$$v^i=-\frac{1}{\mu_i} p^i\cara_{_{\mathcal{O}_i}} |Jac (\tau)|.$$
\end{remark}

\subsection{\bf Nash quasi-equilibria and Nash equilibria}

Since the equation is not linear, the convexity of  functionals $J_i$ is not guaranteed. 
Notice that, 1.4 is equivalent to
%\begin{eqnarray*}
%	J_1(f,v^1,v^2)&=&\frac{\alpha_1}{2}\iint_{\mathcal{O}_{1,d}\times(0,T)}|z-z_{1,d}|^2\ dx dt+\frac{\mu_1}{2}\iint_{\mathcal{O}_{1}\times(0,T)}|v^1|^2\ dx dt\nonumber\\
%	& & \\
%	J_2(f,v^1,v^2)&=&\frac{\alpha_2}{2}\iint_{\mathcal{O}_{2,d}\times(0,T)}|z-z_{2,d}|^2\ dx dt+\frac{\mu_2}{2}\iint_{\mathcal{O}_{2}\times(0,T)}|v^2|^2\ dx dt\nonumber
%\end{eqnarray*}
%o que significa
%	\begin{equation*}
%	\begin{array}{ccc}
%		J_1(f;v^1,v^2)& = &\min\limits_{\hat{v}^1} J_1(f;\hat{v}^1,v^2)  \\
%		J_2(f;v^1,v^2)& = & \min\limits_{\hat{v}^2} J_2(f;v^1,\hat{v}^2)
%	\end{array}	
%\end{equation*}
% se o somente se \'e um quase equilíbrio se
%  
%$$J_1'(f,v^1,v^2)(\hat{v}_1,0)=0, \ \ \ \ \ \forall \hat{v}^1\in \mathcal{H}_1$$
%  $$J_2'(f,v^1,v^2)(0,\hat{v}^2)=0, \ \ \ \ \ \forall \hat{v}^2\in \mathcal{H}_2$$
%  
$$\alpha_i\int_{\mathcal{O}_{i,d}\times (0,T)}(y-y_{id})w^idxdt+\mu_i\int_{\mathcal{O}_{i}\times (0,T)}v^i\widetilde{v}^idxdt=0,
  \ \ \ \ \ \forall \widetilde{v}^i\in \mathcal{H}_i, \ \ \ \ \ i=1,2,$$
where $w^i$ satisfies  
    \begin{equation}\label{eqprin44}
   	\begin{cases}
   		w^i_t-b(t)\left(a(x)w^i_x\right)_x-B xw^i_x + D_1 F\left(y,C\beta(x)y_x\right)w^1 + C D_2 F(y,C\beta(x)y_x)\beta(x) w^i_x = \widetilde{v}^i\cara_{_{{\mathcal{O}}_i}}, & \ \ \ \text{in} \ \ \ {Q},\\
   		w^i=0, & \ \ \ \text{on} \ \ \ {\Sigma},\\
   		w^i(0)=0, & \ \ \ \text{in} \ \ \ \Omega.
   	\end{cases}
   \end{equation}
   
Fix $i=1$  and given $\lambda\in \rea$ and $\widetilde{v}^1, \overline{v}^1$ in $\mathcal{H}_1$, define 
$$P(\lambda) \rightarrow \langle D_1 J_1(h,v^1+\lambda\widetilde{v}^1,v^2 ),\overline{v}^1 \rangle$$
such that  
$$\langle D_1 J_1(h,v^1+\lambda\widetilde{v}^1,v^2 ),\overline{v}^1 \rangle=\alpha_1\int_{\mathcal{O}_{1,d}\times (0,T)}(y^\lambda-y_{1d})q^\lambda dxdt+\mu_1\int_{\mathcal{O}_{1}\times (0,T)}(v^1+\lambda\widetilde{v}^1)\overline{v}^1dxdt$$
where 
   \begin{equation}\label{eqprin5}
   	\hspace*{-1.2cm}	\begin{cases}
   		y^\lambda_t-b(t)\left(a(x)y^\lambda_x\right)_x-Bxy^\lambda_x+F\left(y^\lambda,C\beta(x)y^\lambda_x\right)={h}\cara_{_{{\mathcal{O}}}}+(v^1+\lambda\overline{v}^1)\cara_{_{{\mathcal{O}}_1}}+{v}^2\cara_{_{{\mathcal{O}}_2}}, & \ \ \ \text{in} \ \ \ {Q},\\
   		y^\lambda=0, & \ \ \ \text{on} \ \ \ {\Sigma},\\
   		y^\lambda(0)=y_0, & \ \ \ \text{in} \ \ \ \Omega,
   	\end{cases}
   \end{equation}
and $q^\lambda$ the derivative of the state $y^\lambda$ with respect to $v^1$ in the direction $\widetilde{v}^1$, i.e. the solution to   
   \begin{equation}\label{eqprin6}
   	\begin{cases}
   		q^\lambda_t-b(t)\left(a(x)q^\lambda_x\right)_x-Bxq^\lambda_x + D_1 F\left(y^\lambda, C\beta(x) y^\lambda_x \right)q^\lambda \\
        \quad + C D_2 F\left(y^\lambda, C\beta(x) y^\lambda_x \right)\beta(x) q^\lambda_x 
        = \widetilde{v}^1\cara_{_{{\mathcal{O}}_1}}, & \ \ \ \text{in} \ \ \ {Q},\\
   		q^\lambda=0, & \ \ \ \text{on} \ \ \ {\Sigma},\\
   		q^\lambda(0)=0, & \ \ \ \text{in} \ \ \ \Omega,
   	\end{cases}
   \end{equation}
   where we are denoting $y=\left. y^\lambda\right|_{\lambda=0}$ and $q=\left. q^\lambda\right|_{\lambda=0}$. Using the systems above we consider
    \begin{equation}\label{eq:plambda-p0)}
    \begin{split}
        P(\lambda)-P(0) = &\alpha_1\int_{\mathcal{O}_{1,d}\times (0,T)}(y^\lambda-y_{1d})q^\lambda dxdt -\alpha_1\int_{\mathcal{O}_{1,d}\times (0,T)}(y-y_{1d})q dxdt\\
        &+\lambda \mu_1 \int_{\mathcal{O}_{1}\times (0,T)}\overline{v}^1\widetilde{v}^1dxdt.    
    \end{split}
    \end{equation}
   
The adjoint system of (\ref{eqprin6}) is   
   \begin{equation}\label{eqprin7}
   	\begin{cases}
   		-\phi^\lambda_t-b(t)\left(a(x)\phi^\lambda_x\right)_x+B \left(x\phi^\lambda\right)_x + D_1 F\left(y^\lambda, C\beta(x) y^\lambda_x \right)\phi^\lambda\\
        \quad - C\left(D_2 F(y^\lambda, C \beta(x) y^\lambda_x)\beta(x)\phi^\lambda \right)_x = \alpha_1(y^\lambda-y_{1d})\cara_{_{{\mathcal{O}}_{1d}}}, & \ \text{in} \ \ {Q},\\
   		\phi^\lambda=0, & \ \text{on} \ \ {\Sigma},\\
   		\phi^\lambda(T)=0, & \ \text{in} \ \ \Omega.
   	\end{cases}
   \end{equation}
   
   Multiplying (\ref{eqprin6}) by $\phi^\lambda$ and integrating on $Q$ we get
   \begin{equation}
       \begin{split}
            \int_{Q}\left(q^\lambda_t-b(t)\left(a(x)q^\lambda_x\right)_x-Bxq^\lambda_x + D_1 F\left(y^\lambda,C\beta(x)y^\lambda_x \right)q^\lambda \right. \\
            \left. \quad + C D_2 F\left(y^\lambda, C\beta(x)y^\lambda_x \right)\beta(x) q^\lambda_x \right) \phi^\lambda dxdt
            =\int_{Q} \widetilde{v}^1\cara_{_{{\mathcal{O}}_1}}\phi^\lambda dxdt.      
       \end{split}
   \end{equation}
   
   Integrating by parts, and using (\ref{eqprin7}) we obtain
   $$\alpha_1\int_{Q}(y^\lambda-y_{1d})\cara_{_{{\mathcal{O}}_{1d}}}q^\lambda dxdt=\int_{Q} \widetilde{v}^1\cara_{_{{\mathcal{O}}_1}}\phi^\lambda dxdt,   $$
   and, for $\lambda=0$,
      $$\alpha_1\int_{Q}(y-y_{1d})\cara_{_{{\mathcal{O}}_{1d}}}q dxdt = \int_{Q} \widetilde{v}^1\cara_{_{{\mathcal{O}}_1}}\phi dxdt.$$
   
   Substituting the latter in \eqref{eq:plambda-p0)} we obtain
   \begin{equation}\label{ec212}
   	\frac{P(\lambda)-P(0)}{\lambda}=\int_{\mathcal{O}_{1}\times (0,T)}\left(\frac{\phi^\lambda-\phi}{\lambda}\right)\widetilde{v}^1 dxdt+\mu_1 \int_{\mathcal{O}_{1}\times (0,T)}\overline{v}^1\widetilde{v}^1dxdt.
   \end{equation}
   
   Notice that, using the system (\ref{eqprin7}),   
   \begin{multline}\label{ec213}
   	-(\phi^\lambda-\phi)_t-b(t)\left(a(x)(\phi^\lambda-\phi)_x\right)_x+B\left(x(\phi^\lambda-\phi)\right)_x\\
   	+\left[ D_1F\left(y^\lambda,C\sqrt{a}y^\lambda_x  \right)-D_1F\left(y,C\sqrt{a}y_x  \right)  \right]\phi^\lambda
   +D_1F\left(y,C\sqrt{a}y_x  \right)(\phi^\lambda-\phi)\\
   	-C \left(  \left( D_2F\left(y^\lambda,C\sqrt{a}y^\lambda_x  \right)-D_2F\left(y,C\sqrt{a}y_x  \right)  \right)\sqrt{a}\phi^\lambda+D_2F\left(y,C\sqrt{a}y_x  \right)\sqrt{a}(\phi^\lambda-\phi) \right)_x \\
    =\alpha_1(y^\lambda-y)\cara_{_{{\mathcal{O}}_{1d}}}
   \end{multline}

%   \begin{equation}\label{ec213}
%   	-(\phi^\lambda-\phi)_t-\frac{1}{\ell(t)^2}\left(a(x)(\phi^\lambda-\phi)_x\right)_x+\frac{\ell'(t)}{\ell(t)}\left(x(\phi^\lambda-\phi)\right)_x+(F'(y^\lambda)-F'(y))\phi^\lambda +F'(y)(\phi^\lambda -\phi)=\alpha_1(y^\lambda-y)\cara_{_{{\mathcal{O}}_{1d}}}
%   \end{equation}
   and using the system (\ref{eqprin5}) we obtain
   \begin{equation}\label{ec214}
   	(y^\lambda-y)_t-b(t)\left(a(x)(y^\lambda-y)_x\right)_x-Bx(y^\lambda-y)_x+F\left(y^\lambda,C\sqrt{a}y^\lambda_x  \right)-F\left(y,C\sqrt{a}y_x  \right)=\lambda\overline{v}^1\cara_{_{{\mathcal{O}}_1}}
   \end{equation}
   Let us define $\eta=\lim_{\lambda\to 0} \frac{1}{\lambda}(\phi^\lambda - \phi)$ and $\theta=\lim_{\lambda\to 0} \frac{1}{\lambda}(y^\lambda - y)$. Dividing (\ref{ec213}) and (\ref{ec214}) by $\lambda$ and making $\lambda \to 0$ we can deduce the following system,  
   \begin{equation}
   	\label{eq:accoplatesystem}
   	\left\{\begin{aligned}
   		&-\eta_t -b(t)\left( a(x) \eta_x \right)_x +B\left(x\eta \right)_x +D^2_{11}F\left(y,C\sqrt{a}y_x  \right)\phi \theta+C\cdot D^2_{12}F\left(y,C\sqrt{a}y_x  \right)\phi \sqrt{a}\theta_x \\
        &+D_{1}F\left(y,C\sqrt{a} \right)\eta -C\left(D^2_{21}F\left(y,C\sqrt{a}y_x  \right)\phi \sqrt{a} \theta\right)_x-C^2\left(D^2_{22}F\left(y,C\sqrt{a}y_x  \right)\phi a  \theta_x\right)_x \\ &+C\left(D_2F\left(y,C\sqrt{a} \right)\sqrt{a}\eta\right)_x=\alpha_1 \theta \cara_{_{{\mathcal{O}}_{1d}}} &&\text{in}&& Q,\\
   		&\theta_t - b(t)\left( a(x) \theta_x  \right )_x -B x\theta_x+ D_1F\left(y,C\sqrt{a}y_x  \right)\theta+C\cdot D_2F\left(y,C\sqrt{a} y_x  \right)\sqrt{a}\theta_x = \bar{v}^{1} \cara_{_{{\mathcal{O}}_1}}&&\text{in}&& Q,\\ 
   		&\eta(0,t)=\eta(1,t)=0,\ \theta(0,t)=\theta(1,t)=0 &&\text{on} && (0,T), \\
   		&\eta(\cdot,T)=0, \ \theta(\cdot,0)=0 &&\text{in} && \Omega.
   	\end{aligned}
   	\right.
   \end{equation}
   
Or, in a simpler form
      \begin{equation}
   	\label{eq:accoplatesystem1}
   	\left\{\begin{aligned}
   		&-\eta_t -b(t)\left( a(x) \eta_x \right)_x +B\left(x\eta \right)_x +D_{1}F\left(y,C\sqrt{a}y_x  \right)\eta-C\left(D_2F\left(y,C\sqrt{a}y_x  \right)\sqrt{a}\eta\right)_x\\
        &+\mathcal{N}(\eta,\theta)=\alpha_1 \theta \cara_{_{{\mathcal{O}}_{1d}}} &&\text{in}&& Q,\\
   		&\theta_t - b(t)\left( a(x) \theta_x  \right )_x -Bx\theta_x+ D_1F\left(y,C\sqrt{a}y_x  \right)\theta+C\cdot D_2F\left(y,C\sqrt{a}y_x  \right)\sqrt{a}\theta_x = \bar{v}^{1} \cara_{_{{\mathcal{O}}_1}}&&\text{in}&& Q,\\ 
   		&\eta(0,t)=\eta(1,t)=0,\ \theta(0,t)=\theta(1,t)=0 &&\text{on} && (0,T), \\
   		&\eta(\cdot,T)=0, \ \theta(\cdot,0)=0 &&\text{in} && \Omega.
   	\end{aligned}
   	\right.
   \end{equation}
where 
   \begin{multline*}
\mathcal{N}(\eta,\theta)=D^2_{11}F\left(y,C\sqrt{a}y_x  \right)\phi \theta+C\cdot D^2_{12}F\left(y,C\sqrt{a}y_x  \right)\phi \sqrt{a}\theta_x\\
-C\left(D^2_{21}F\left(y,C\sqrt{a}y_x  \right)\phi \sqrt{a} \theta\right)_x-C^2\left(D^2_{22}F\left(y,C\sqrt{a}y_x  \right)\phi a  \theta_x\right)_x.
   \end{multline*}
   
   Then, making $\lambda \to 0$ in (\ref{ec212}), we obtain the second derivative
   \begin{equation}
   	\label{eq:D2J1_w1_w2}
   	\langle D_1^2 J_1(h,v^1,v^2), (\bar{v}^{1}, \tilde{v}^{1}) \rangle = \int_{O_1 \times (0,T)} \eta \tilde{v}^{1} \ dxdt + \mu_1 \int_{O_1 \times (0,T)} \bar{v}^{1}\tilde{v}^{1} \ dxdt, \ \forall\bar{v}^{1},\Tilde{v}^{1} \in L^{2}(O_{1} \times (0,T)). 
   \end{equation}
   
   In particular, taking $\Bar{v}^{1}=\Tilde{v}^{1}$
   \begin{equation}
   	\label{eq:D2J1}
   	\langle D_1^2 J_1(h,v^1,v^2), (\bar{v}^{1}, \bar{v}^{1}) \rangle = \int_{O_1 \times (0,T)} \eta \bar{v}^{1} \ dxdt + \mu_1 \int_{O_1 \times (0,T)} |\bar{v}^{1}|^{2} \ dxdt. 
   \end{equation}
   
   Let us prove that there exists a constant $C> 0$ such that
   \begin{equation}
   	\label{eq:integral_eta_w1_bounded}
   	\left|\int_{O_1 \times (0,T)} \eta \tilde{v}^{1} \ dxdt\right|\leq C \left( 1 + 3\left( \|h\|_{L^{2}(O \times (0,T))}+\|y_{0}\| +\sum_{i=1}^{2}\|y_{i,d}\|^{2}_{L^{2}((0,T),L^2(O_{i,d}))} \right) \right) \|\bar{v}^{1}\|^{2}_{L^{2}(O_{1} \times (0,T))}.
   \end{equation}
   
   In fact, using \eqref{eq:accoplatesystem}, integration by parts and one more time \eqref{eq:accoplatesystem}, we get
  {\small \begin{eqnarray*}
   	 \int_{O_1 \times (0,T)} \eta \bar{v}^{1} \ dxdt &=& \int_{\Omega \times (0,T)} \eta\left(\theta_t - b(t)\left( a(x) \theta_x  \right )_x -Bx\theta_x+ D_1F\left(y,C\sqrt{a}y_x  \right)\theta \right.\\
     &&\left. +C\cdot D_2F\left(y,C\sqrt{a}y_x  \right)\sqrt{a}\theta_x  \right) \ dxdt \\
   	&=&\int_{\Omega \times (0,T)} \theta\left(-\eta_t -b(t)\left( a(x) \eta_x \right)_x +B\left(x\eta \right)_x +D_{1}F\left(y,C\sqrt{a}y_x  \right)\eta \right.\\
    &&\left. -C\left(D_2F\left(y,C\sqrt{a}y_x  \right)\sqrt{a}\eta\right)_x\right) \ dxdt\\
   	&=& \int_{\Omega\times(0,T)} \left( \alpha_{1}|\theta|^{2}1_{O_{1,d}}  - \mathcal{N}(\eta,\theta)\theta \right) dxdt.
   \end{eqnarray*}}

Substituting the value of $\mathcal{N}$ and integrating by parts we get
   \begin{equation*}
   \begin{split}
   	\int_{O_1 \times (0,T)} \eta \bar{v}^{1} \ dxdt &= \int_{\Omega\times(0,T)}\alpha_{1}|\theta|^{2}1_{O_{1,d}} - \int_{\Omega\times(0,T)}D^2_{11}F\left(y,C\sqrt{a}y_x  \right)\phi |\theta|^2 \\
    &-C\cdot D^2_{12}F\left(y,C\sqrt{a}y_x  \right)\phi\theta \sqrt{a}\theta_x 	+\int_{\Omega\times(0,T)}C\left(D^2_{21}F\left(y,C\sqrt{a}y_x  \right)\phi \sqrt{a} \theta\right)_x\theta \\
    &+ \int_{\Omega\times(0,T)}C^2\left(D^2_{22}F\left(y,C\sqrt{a}y_x  \right)\phi a  \theta_x\right)_x\theta\\
   	&= \int_{\Omega\times(0,T)}\alpha_{1}|\theta|^{2}1_{O_{1,d}} - \int_{\Omega\times(0,T)}D^2_{11}F\left(y,C\sqrt{a}y_x  \right)\phi |\theta|^2 \\
    &-C\cdot D^2_{12}F\left(y,C\sqrt{a}y_x  \right)\phi\theta \sqrt{a}\theta_x
   	-\int_{\Omega\times(0,T)}C\cdot D^2_{21}F\left(y,C\sqrt{a}y_x  \right)\phi\theta \sqrt{a} \theta_x\\
    &- \int_{\Omega\times(0,T)}C^2\cdot D^2_{22}F\left(y,C\sqrt{a}y_x  \right)\phi |\sqrt{a}\theta_x|^2.
    \end{split}
    \end{equation*}

We estimate each term in the last expression using that all derivatives of $F$ up to second order are bounded by a constant, say $K$. First,
$$\int_{0}^{T}\int_{\Omega}D^2_{11}F\left(y,C\sqrt{a}y_x  \right)\phi |\theta|^2\leq C\int_{0}^{T}\int_{\Omega}\phi(x,t) |\theta|^2\leq C\int_{0}^{T}\|\phi(t)\|_{L^\infty(\Omega)}\int_{\Omega} |\theta|^2, $$
   and since $H^1_a(\Omega)\hookrightarrow L^\infty(\Omega)$ continously, then $\|f\|_{L^\infty(\Omega)}\leq \|f\|_{H^1_a(\Omega)}=\|f\|_{L^2(\Omega)}+\|af_x\|_{L^2(\Omega)}$. Thus,
    $$\int_{0}^{T}\int_{\Omega}D^2_{11}F\left(y,C\sqrt{a}y_x  \right)\phi |\theta|^2 \leq K \int_{0}^{T} \|\sqrt{a}\phi_x(t)\|_{L^2(\Omega)}\|\theta(t)\|^2_{L^2(\Omega)}.$$

Analogously, for $i \neq j$, using Hölder inequality, 
\begin{equation*}
    \begin{split}    \int_{0}^{T}\int_{\Omega}C(t)D^2_{ij}F\left(y,C\sqrt{a}y_x  \right)\phi\theta \sqrt{a} \theta_x &\leq %C\int_{0}^{T}\int_{\Omega}\phi(x,t) \theta \sqrt{a} \theta_x\leq
C\int_{0}^{T}\|\phi(t)\|_{L^\infty(\Omega)}\int_{\Omega} \theta \sqrt{a} \theta_x \\
&\leq C\int_{0}^{T} \|\sqrt{a}\phi_x(t)\|_{L^2(\Omega)} \|\theta\|_{L^2(\Omega)}\|\sqrt{a}\theta_x(t)\|_{L^2(\Omega)}.    
    \end{split}
\end{equation*}

Finally,
\begin{equation*}
    \begin{split}
        \int_{0}^{T}\int_{\Omega}C^2(t)D^2_{22}F\left(y,C\sqrt{a}y_x  \right)\phi |\sqrt{a}\theta_x|^2 &\leq %C\int_{0}^{T}\int_{\Omega}\phi(x,t) |\sqrt{a}\theta_x|^2\leq
   K\int_{0}^{T}\|\phi(t)\|_{L^\infty(\Omega)}\int_{\Omega} |\sqrt{a}\theta_x|^2\\
   &\leq K \int_{0}^{T} \|\sqrt{a}\phi_x(t)\|_{L^2(\Omega)} \|\sqrt{a}\theta_x(t)\|^2_{L^2(\Omega)}.      
    \end{split}
\end{equation*}

Therefore,  
   \begin{multline*}
\left| 	\int_{O_1 \times (0,T)} \eta \bar{v}^{1} \ dxdt \right|\leq C\left( \int_{0}^{T} \|\theta(t)\|^2_{L^2(\Omega)}dt + \int_{0}^{T} \|\sqrt{a}\phi_x(t)\|_{L^2(\Omega)}\|\theta(t)\|^2_{L^2(\Omega)}dt\right.\\
\left. +\int_{0}^{T} \|\sqrt{a}\phi_x(t)\|_{L^2(\Omega)} \|\theta\|_{L^2(\Omega)}\|\sqrt{a}\theta_x(t)\|_{L^2(\Omega)}dt+\int_{0}^{T} \|\sqrt{a}\phi_x(t)\|_{L^2(\Omega)} \|\sqrt{a}\theta_x(t)\|^2_{L^2(\Omega)}dt\right).
   \end{multline*}

%   \textcolor{red}{\bf lo podrias revisar y terminarlo, no entendi muy bien lo que ustedes hicieron en la parte de su trabajo con Joao y Suerlan, creo que es muy parecido pero no entiendo mucho como hicieron despues de esto, estare avanzando algo el sistema, por lo menos que este biem escrito el sistema optimal para ubicarnos mejor como atacar el carleman}
      
   From (\ref{eqprin7}) with $\lambda = 0$ and (\ref{eq:accoplatesystem}), using the energy estimates of Appendix \ref{appendix A}, where $\phi$ verifies the same equation that $p^1$, we have
   \begin{equation*}
   \begin{split}
   	&\|(a\phi_{x})_x\|^{2}_{L^{2}(Q)} + 
   	\|\sqrt{a} \phi_{x}\|^{2}_{L^\infty(0,T,L^2(Q))} + \|\phi\|^2_{L^{\infty}(0,T,L^{2}(\Omega))} \\
    &\leq C\left( \|h\|^{2}_{L^{2}((0,T),L^2(O))}+\|y_{0}\|_{H^{1}_{a}(0,1)}+\sum_{i=1}^{2}\|y_{i,d}\|^{2}_{L^{2}((0,T),L^2(O_{i,d}))} \right),
    \end{split}
   \end{equation*}
   and, analogously, we have energy estimates for $\theta$, 
   \begin{equation*}
   \|(a \theta_{x})_x\|^{2}_{L^{2}(Q)} + 
   	\|\sqrt{a} \theta_{x}\|^{2}_{L^\infty(0,T,L^2(Q))} + \|\theta\|^{2}_{L^{\infty}(0,T,L^{2}(\Omega))} \leq C\left(\|\bar{v}^{1}\|^{2}_{L^{2}(O_{1,d}\times(0,T))}\right).
   \end{equation*}

   Thus, using the energy estimates we have
   \begin{eqnarray*}
\left|\int_{O_1 \times (0,T)} \eta \tilde{v}^{1} \ dxdt\right|&\leq& C \left( \|\theta\|_{L^\infty(0,T,L^2(\Omega))}^2 T + \|\sqrt{a}\phi_x\|_{L^\infty(0,T,L^2(\Omega))} \|\theta\|_{L^\infty(0,T,L^2(\Omega))}^2 T \right. \\
& &\left. +\|\sqrt{a}\phi_x\|_{L^\infty(0,T,L^2(\Omega))} \|\theta\|_{L^\infty(0,T,L^2(\Omega))} \|\sqrt{a}\theta_x\|_{L^\infty(0,T,L^2(\Omega))} T \right. \\
&&\left. +\|\sqrt{a}\phi_x\|_{L^\infty(0,T,L^2(\Omega))} \|\sqrt{a}\theta_x\|_{L^\infty(0,T,L^2(\Omega))}^2 T \right)  \\
   	&\leq& \tilde C \left( 1 + 3\left( \|h\|_{L^{2}( O \times (0,T))}+\|y^{0}\|_{H^{1}_{a}(\Omega)}^{2} +\sum_{i=1}^{2}\|y_{i,d}\|^{2}_{L^{2}((0,T),L^2(O_{i,d}))} \right) \right) \|\bar{v}^{1}\|^{2}_{L^{2}(O_{1} \times (0,T))}.
   \end{eqnarray*}
   %This proves \eqref{eq:integral_eta_w1_bounded}.
   
   Substituting the latter in \eqref{eq:D2J1} we get
   $$
   \langle D_1^2 J_1(h,v^1,v^2), (\bar{v}^{1}, \bar{v}^{1}) \rangle \geq (\mu_1 - C) \int_{O_1 \times (0,T)} |\tilde{v}^{1}|^2 \ dxdt.
   $$
   
   Taking $\mu_1$ big enough, this shows that $J_1$ is convex. Analogously, by a similar computation, taking $\mu_2$ big enough, we have that $J_2$ is convex. Thus, the Nash quasi-equilibrium par $(v^1,v^2)$ is in fact a Nash equilibrium.
      
%   \newpage
   
\section{Null controllability of the linearized system}
\label{Sec:null_linearized_system}

   \subsection{Hilbert spaces in the divergence case}
   Following [1], for the system in divergence form, we take the following weighted Hilbert spaces. In the \emph{weakly degenerate} case, we consider
   $$H^1_a(0,1)=\left\{ u\in L^2(0,1) \ : \ u \ \text{is absolutely continuous in} \ \ (0,1], \sqrt{a}u_x\in L^2(0,1) \ \text{and} \  u(0)=u(1)=0 \right\}$$
   and 
   $$H^2_a(0,1)=\left\{u\in H^1_a(0,1) \ : \ au_x\in H^1(0,1) \right\}.$$
   
   In both cases, we consider inner products and norms given by
   $$\langle u,v \rangle_{H^1_a(0,1)}=\langle u,v \rangle_{L^2(0,1)}+\langle \sqrt{a}u_x,\sqrt{a}v_x \rangle_{L^2(0,1)}, \ \ \|u\|^2_{H^1_a(0,1)}=\|u\|^2_{L^2(0,1)}+\|\sqrt{a} u_x\|^2_{L^2(0,1)},$$
   for all $u, v\in H^1_a(0,1)$, and
  $$\langle u,v \rangle_{H^2_a(0,1)}=\langle u,v \rangle_{H^1_a(0,1)}+\langle ({a}u_x)_x,({a}v_x)_x \rangle_{L^2(0,1)}, \ \ \|u\|^2_{H^2_a(0,1)}=\|u\|^2_{H^1_a(0,1)}+\|({a}u_x)_x\|^2_{L^2(0,1)},$$
  for all $u, v\in H^2_a(0,1)$.\\
  
%Let us recall the optimality system (\ref{eq:optimality_system}).  

%   \begin{equation}\label{eqoptimal}
%   	\begin{cases}
%   		y_t-b(t)\left(a(x)y_x\right)_x-Bxy_x+F\left(y\right)={h}\cara_{_{{\mathcal{O}}}}-\frac{1}{\mu_1}p^1\cara_{_{\mathcal{O}_1}}-\frac{1}{\mu_2}p^2\cara_{_{\mathcal{O}_2}}, & \ \ \ \text{in} \ \ \ {Q}\\
%   		-p^i_t-b(t)\left(a(x)p^i_x\right)_x+B \left(xp^i\right)_x + F'\left(y\right)p^i = \alpha_i(y-y_{id})\cara_{_{{\mathcal{O}}_{id}}}, & \ \ \ \text{in} \ \ \ {Q}\\
%   		y=0, \ \ \ \ \ p^1=0, \ \ \ \ \ p^2=0& \ \ \ \text{in} \ \ \ \sum\\
%   		y(0)=y_0, \ \ \ \ \ p^1(T)=0, \ \ \ \ \ \ p^2(T)=0& \ \ \ \text{in} \ \ \ \Omega
%   	\end{cases}
%   \end{equation}   
   
We compute formally the derivative a zero of the map $\mathcal{T}$ defined by the optimality system \eqref{eq:optimality_system}, i.e.  
$$\mathcal{T}(y,p^1,p^2)=(\mathcal{T}_0(y,p^1,p^2),\mathcal{T}_1(y,p^1,p^2),\mathcal{T}_2(y,p^1,p^2))$$
where, for $i=1,2$,   
\begin{eqnarray*}
\mathcal{T}_0(y,p^1,p^2)&=&y_t-b(t)\left(a(x)y_x\right)_x-Bxy_x+F\left(y,C\beta(x)y_x\right)+\frac{1}{\mu_1}p^1\cara_{_{\mathcal{O}_1}}+\frac{1}{\mu_2}p^2\cara_{_{\mathcal{O}_2}} \\
\mathcal{T}_i(y,p^1,p^2)&=&-p^i_t-b(t)\left(a(x)p^i_x\right)_x+B \left(xp^i\right)_x + D_1 F\left(y,C\beta(x)y_x\right)p^i \\
&&- C\left( D_2 F(y,C\beta(x)y_x)\beta(x)p^i \right)_x - \alpha_i y\cara_{_{{\mathcal{O}}_{id}}}.
\end{eqnarray*}
   
Thus, $D\mathcal{T}(0,0,0)(\overline{y},\overline{p}^1,\overline{p}^2)=\lim\limits_{
   \lambda\rightarrow 0}\frac{\mathcal{T}(\lambda(\overline{y},\overline{p}^1,\overline{p}^2))-\mathcal{T}(0,0,0)}{\lambda}$. Forgeting the
   tilde in the notations, we get the linearized system, for $i=1,2$
\begin{equation}%\label{eqoptimallineal}
   \label{eq:linearized_system}
    \begin{cases}
   			y_t-b(t)\left(a(x)y_x\right)_x-B xy_x + D_1 F\left(0,0\right)y + C D_2 F(0,0)\beta(x)y_x \\
            \quad ={h}\cara_{_{{\mathcal{O}}}}-\frac{1}{\mu_1}p^1\cara_{_{\mathcal{O}_1}}-\frac{1}{\mu_2}p^2\cara_{_{\mathcal{O}_2}}+H, & \ \ \ \text{in} \ \ \ {Q},\\
   		-p^i_t-b(t)\left(a(x)p^i_x\right)_x + B\left(xp^i\right)_x + D_1 F\left(0,0\right)p^i- C D_2 F(0,0) (\beta(x)p^i)_x \\
        \quad = \alpha_i y\cara_{_{{\mathcal{O}}_{id}}}+H_i,  \ \ i=1,2& \ \ \ \text{in} \ \ \ {Q},\\
   		y=0, \ \ \ \ \ p^1=0, \ \ \ \ \ p^2=0& \ \ \ \text{on} \ \ \ \Sigma,\\
   		y(0)=y_0, \ \ \ \ \ p^1(T)=0, \ \ \ \ \ \ p^2(T)=0& \ \ \ \text{in} \ \ \ \Omega,
   	\end{cases}
   \end{equation}
   with associated adjoint system, 
      \begin{equation}%\label{eqoptimallinealadj}
      \label{eq:adjoint_optimality_system}
   	\begin{cases}
   		-\phi_t-b(t)\left(a(x)\phi_x\right)_x+B\left(x\phi\right)_x + D_1 F\left(0,0\right)\phi - C D_2 F(0,0)(\beta(x)p^i)_x\\
        \qquad= F + \alpha_1\psi^1\cara_{_{\mathcal{O}_{1d}}}+\alpha_2\psi^2\cara_{_{\mathcal{O}_{2d}}}, & \ \ \ \text{in} \ \ \ {Q},\\
   		\psi^i_t-b(t)\left(a(x)\psi^i_x\right)_x - B x\psi^i_x + D_1 F\left(0,0\right)\psi^i + C D_2 F(0,0) \beta(x)\psi^i_x\\
        \qquad= F_i-\frac{1}{\mu_i}\phi \cara_{_{\mathcal{O}_i}},  \ \ i=1,2& \ \ \ \text{in} \ \ \ {Q},\\
   		\phi=0, \ \ \ \ \ \psi^1=0, \ \ \ \ \ \psi^2=0& \ \ \ \text{on} \ \ \ \Sigma,\\
   		\phi(T)=\phi^T, \ \ \ \ \ \psi^1(0)=0, \ \ \ \ \ \ \psi^2(0)=0& \ \ \ \text{in} \ \ \ \Omega,
   	\end{cases}
   \end{equation}
   where $\phi^T\in L^2(\Omega)$.
   
   To simplify the system, we assume the condition \eqref{ec9}, i.e., after the coordinate transformation, 
   ${\mathcal{O}_{1d}}={\mathcal{O}_{2d}}={\mathcal{O}_{d}}$. 
   Introducing the new variable $\varrho=\alpha_1\psi^1+\alpha_2\psi^2$, the adjoint system becomes
    \begin{equation}%\label{eqoptimallinealadj2}
         \label{eq:adjoint_optimality_system2}
   	\begin{cases}
   		-\phi_t-b(t)\left(a(x)\phi_x\right)_x+B\left(x\phi\right)_x+ D_1 F\left(0,0 \right)\phi - C D_2 F(0,0) \left(\beta(x)\phi \right)_x \\
        \quad = \varrho\cara_{_{\mathcal{O}_{d}}}+G_0, & \ \ \ \text{in} \ \ \ {Q}\\
   		\varrho_t-b(t)\left(a(x)\varrho_x\right)_x-B x\varrho_x + D_1 F\left(0,0 \right)\varrho + C D_2 F(0,0)\beta(x)\varrho_x \\
        \quad = -\left( \frac{\alpha_1}{\mu_1}\cara_{_{\mathcal{O}_1}} +\frac{\alpha_2}{\mu_2}\cara_{_{\mathcal{O}_2}} \right)\phi +G,  \ & \ \ \ \text{in} \ \ \ {Q}\\
   		\phi=0, \ \ \ \ \ \varrho=0  & \ \ \ \text{in} \ \ \ \sum\\
   		\phi(T)=\phi^T, \ \ \ \ \ \varrho(0)=0, & \ \ \ \text{in} \ \ \ \Omega
   	\end{cases}
   \end{equation}
   where $\phi^T\in L^2(\Omega)$ and $G=\alpha_1 F_1+\alpha_2 F_2$.

%--------- BEGIN COMMENT ---------
\begin{comment}   
\begin{teo}\label{teo34noauto}
	Assume a (WD) or (SD) at $x_0\in \{0,1\}$ and $b\in W^{1,\infty}(0, T)$ a strictly positive function. If $f\in L^2(0,T; L^2(0,1))$ and $u_0\in L^2(0,1)$, then there exists a unique solution $u\in C([0,T];L^2(0,1))$ of
	(P), and the following inequality holds
	\begin{equation}\label{ec36noauto}
		\sup_{t\in[0,T]}\|u(t)\|^2_{L^2(0,1)}\leq C_T \left(  \|u_0\|^2_{L^2(0,1)}+\|f\|^2_{L^2(0,T;L^2(0,1))} \right)
	\end{equation}
	for a positive constant $C_T$.\\
	In addition, if \ $f\in W^{1,1}(0,T;L^2(0,1))\cap L^2(0,T; H^2_a(0,1))$ and $u_0\in H^2_a(0,1)$ then $u\in C([0,T];H^2_a(0,1))\cap C^1([0,T];L^2(0,1))\subseteq L^2(0,T;H^2_a(0,1))\cap C([0,T];H^1_a(0,1))\cap H^1(0,T;L^2(0,1)) $ and we have
	\begin{eqnarray}
		\sup_{t\in[0,T]}\|u(t)\|^2_{H^1_a(0,1)}\ + \ \int_{0}^{T}\left(  \|u_t\|^2_{L^2(0,1)}+\|(au_x)_x(t)\|^2_{L^2(0,1)} \right)dt\leq C\left( \|u_0\|^2_{H^1_a(0,1)}+\|f\|^2_{L^2(0,T;L^2(0,1))}  \right)
	\end{eqnarray}
\end{teo}
este resultado hay q adecuarle para el sistema de ecuaciones que se tiene del sistema optimal linealizado 
\end{comment}
%------------------ END COMMENT -------------------------
   
\subsection{Carleman estimates}
%%%%%% Do artigo HIERARQUICO 2 %%%%%%%
%Since $O_d \cap O \neq \emptyset$, then it was proved in \cite{Fur_Ima-96} that there is a function $\sigma \in C^2([0,1])$ such that, in an open set $O' = (\alpha',\beta') \subset O \cap O_d$, 
%$$
%\sigma(x) > 0 \quad \text{in} \quad (0,1), \quad \sigma(0)=\sigma(1) = 0, \quad \sigma_x(x) \neq 0, \quad \text{in} \quad [0,1]-O_o,
%$$
%where $O_o \subset\subset O'$ is an open subset.

Following \cite{DemarqueLimacoViana_deg_eq2018, DemarqueLimacoViana_deg_sys2020}, let $O'=(\alpha',\beta') \subset\subset O$ and let $\Psi : [0,1] \to \R$ be a $C^2$ function such that
$$
\Psi(x) = 
\begin{cases}
\int_0^x \frac{s}{a(s)} ds, \quad x \in [0,\alpha'), \\
-\int_{\beta'}^x \frac{s}{a(s)} ds, \quad x \in [\beta',1].
\end{cases}
$$
For $\lambda \geq \lambda_0$ define the functions
$$
\theta(t) = \frac{1}{(t(T-t))^4}, \qquad \eta(x) = e^{\lambda(|\Psi|_\infty + \Psi)}, \qquad \sigma(x,t) = \theta(t) \eta(x),
$$
$$
\varphi(x,t) = \theta(t) (e^{\lambda(|\Psi|_\infty + \Psi)}-e^{3\lambda|\Psi|_\infty}).
$$

We introduce the notation
$$
I(\varsigma) = \int_Q e^{2s\varphi}((s\lambda)\sigma b(t)^2 a \varsigma_x^2 + (s\lambda)^2 \sigma^2 b(t)^2\varsigma^2)dxdt.
$$

The following proposition was proved for a cylindrical domain in \cite{DemarqueLimacoViana_deg_sys2020} for an adjoint system of two retrograde equations. After making the substitution $t \mapsto T-t$ to our second equation in \eqref{eq:adjoint_optimality_system2}, we are in the case treated in \cite{DemarqueLimacoViana_deg_sys2020} and the conclusion remains true after coming back to the original variable.

For a non-cylindrical domain we have an analogous result \cite{GYL-Carleman-2025}. In fact, first we consider the linear equation:
   \begin{equation}\label{adjunto2_texto_art}
	\begin{cases}
		v_t+b(t)\left(a(x)v_x\right)_x =h(x,t), & \ \ \ \text{in} \ \ \ {Q},\\
		v=0, & \ \ \ \text{on} \ \ \ {\Sigma}, \\
        v(T) = v_T, &\ \ \ \text{in} \ \ \ (0,1),
	\end{cases}
\end{equation}
where $b$ is a continuous function in $[0,T]$ and we denote
$m := \min\limits_{t \in [0,T]} b(t)$ and $M :=\max\limits_{t \in [0,T]} b(t).$
Moreover, we suppose that $\frac{b'(t)}{b(t)} \leq B$, for some constant $B>0$.

\begin{propo}(\cite{GYL-Carleman-2025})
\label{prop_carleman_1}
	There exist $C>0$ and $\lambda_0,  s_0>0$ such that every solution $v$ of (\ref{adjunto2_texto_art}) satisfies, for all
	$s\geq s_0$ and $\lambda\geq \lambda_0$,
	\begin{equation}\label{ecjv34}
		\int_{0}^{T}\int_{0}^{1}e^{2s\varphi}\left(  (s\lambda)\sigma a b^2 v_x^2+(s\lambda)^2\sigma^2b^2 v^2 \right)\leq C\left(	\int_{0}^{T}\int_{0}^{1} e^{2s\varphi} |h|^2+(\lambda s)^3\int_{0}^{T}\int_{\omega}e^{2s\varphi} \sigma^3  v^2  \right)
	\end{equation}
\end{propo}
The proof of this proposition is the main result of the authors in \cite{GYL-Carleman-2025}. Also from \cite{GYL-Carleman-2025} we need the following Carleman estimate for the system \eqref{eq:adjoint_optimality_system2} which allows to get the controllability with control acting only in one of the equations. 

\begin{propo}(\cite{GYL-Carleman-2025}) %\cite{DemarqueLimacoViana_deg_sys2020}
There exist positive constants $C$, $\lambda_0$ and $s_0$ such that, for any $s \geq s_0$, $\lambda \geq \lambda_0$ and any $\phi^T \in L^2(Q)$, the corresponding solution $(\phi,\varrho)$ of \eqref{eq:adjoint_optimality_system2} satisfies
$$
I(\phi) + I(\varrho) \leq C \left( \int_Q e^{2s\varphi} s^4 \lambda^4 \sigma^4 (|G_0|^2+|\tilde{G}|^2)dxdt + \int_{O\times(0,T)} e^{2s\varphi} s^8 \lambda^8 \sigma^8 \phi^2dxdt \right).
$$
\end{propo}
%\begin{proof}
%We follow the computations of Proposition 3 of \cite{DemarqueLimacoViana_deg_sys2020}. In general we consider an adjoint system of the form:
%\begin{equation} \label{eq:adj_system_gen}
%\begin{cases}
%    -y_t - b(t)(a(x)y_x)_x - d_1(x,t)\sqrt{a}y_x + b_{11}y + b_{21}z = F_1,\qquad \text{in } (0,T)\times(0,1) \\
%    -z_t - b(t)(a(x)z_x)_x - d_2(x,t)\sqrt{a}z_x + b_{12}y + b_{22}z = F_1,\qquad \text{in } (0,T)\times(0,1) \\
%    y(t,0)=y(0,1)=z(t,0)=z(t,1)=0, \qquad \text{in } (0,T) \\
%    y(T,x)=y_T(x), \quad z(T,x)=z_T(x), \qquad \text{in } (0,1),  
%\end{cases}    
%\end{equation}
%with bounded coefficients... 
%\end{proof}

In order to get the global null controllability of the linearized system, we need a Carleman inequality with weights which do not vanish at $t=0$. Consider the function $m \in C^\infty([0,T])$ satisfying $m(0)>0$,
\begin{equation}
\label{eq:def_m}
m(t) \geq t^4(T-t)^4, \quad t \in (0,T/2], \qquad\qquad m(t) = t^4(T-t)^4, \quad t \in [T/2,T],    
\end{equation}
and define
$$
\tau(t) = \frac{1}{m(t)}, \qquad \zeta(x,t) = \tau(t) \eta(x), \qquad A(x,t)  = \tau(t) (e^{\lambda(|\Psi|_\infty + \Psi)}-e^{3\lambda|\Psi|_\infty}).
$$

Define
$$
\Gamma_0(\phi,\varrho) = \int_Q e^{2s A}((s\lambda) \zeta b^2 a (|\phi_x|^2 + |\varrho_x|^2) + (s\lambda)^2 {\zeta}^2 b^2 (|\phi|^2 + |\varrho|^2))dxdt,
$$
$$
\Gamma(\phi,\varrho,\tilde \varrho) = \int_Q e^{2s A}((s\lambda) \zeta b^2 a (|\phi_x|^2+|\varrho_x|^2+|\tilde \varrho_x|^2) + (s\lambda)^2 {\zeta}^2 b^2 (|\phi|^2 + |\varrho|^2 + |\tilde \varrho|^2))dxdt.
$$

\begin{propo} (\cite{GYL-Carleman-2025}) 
There exist positive constants $C$, $\lambda_0$ and $s_0$ such that, for any $s \geq s_0$, $\lambda \geq \lambda_0$ and any $\phi^T\in L^2(Q)$, the corresponding solution $(\phi,\varrho)$ of \eqref{eq:adjoint_optimality_system2} satisfies
$$
\Gamma_0(\phi,\varrho) \leq C \left( \int_Q e^{2s A} s^4 \lambda^4 \zeta^4 (|G_0|^2+|\tilde{G}|^2)dxdt + \int_{O\times(0,T)} e^{2s A} s^8 \lambda^8 \zeta^8 |\phi|^2 dxdt\right).
$$
\end{propo}
%\begin{proof}
%    We follow the computations of Proposition 4 of \cite{DemarqueLimacoViana_deg_sys2020}.

%\	In the same way as the previous proposition, we can work in a general for, using the following system \eqref{eq:adj_system_gen}.

%\end{proof}

Since $\varrho = \alpha_1 \psi^1 + \alpha_2 \psi^2$ and $\tilde{G}=\alpha_{1}F_{1}+\alpha_{2}F_{2}$ we get directly the following Carleman inequality for the original adjoint system \eqref{eq:adjoint_optimality_system}.

\begin{propo}
\label{prop:carleman}
There exist positive constants $C$, $\lambda_0$ and $s_0$ such that, for any $s \geq s_0$, $\lambda \geq \lambda_0$ and any $\phi^T \in L^2(Q)$, the corresponding solution $(\phi,\psi^1,\psi^2)$ of \eqref{eq:adjoint_optimality_system} satisfies
\begin{equation}
\label{Carleman for eq:adjoint_optimality_system}
\Gamma(\phi,\psi^1,\psi^2) \leq C \left( \int_Q e^{2s A} s^4 \lambda^4 \zeta^4 (|F|^2+|F_1|^2+|F_2|^2)dxdt + \int_{O\times(0,T)} e^{2s A} s^8 \lambda^8 \zeta^8 |\phi|^2 dxdt\right).
\end{equation}
\end{propo}

As a corollary we get the observability inequality.
\begin{coro} %\cite{DemarqueLimacoViana_deg_sys2020}
\label{cor:observability}
There exist positive constants $C$, $\lambda_0$ and $s_0$ such that, for any $s \geq s_0$, $\lambda \geq \lambda_0$ and any $\phi^T\in L^2(Q)$, the corresponding solution $(\phi,\varrho)$ of \eqref{eq:adjoint_optimality_system2} with $F_0 = F_1 = 0$, satisfies
\begin{equation}\label{Observability}
\|\phi(0)\|^2_{L^2(0,1)} + \|\varrho(T)\|^2_{L^2(0,1)} \leq C \int_{O\times(0,T)} e^{2s A} s^8 \lambda^8 \zeta^8 |\phi|^2dxdt.
\end{equation}
\end{coro}
\begin{proof}
    See Corollary 1 of \cite{DemarqueLimacoViana_deg_sys2020}. 
\end{proof}
\color{black}

In the following sections we will need weights which depend only on $t$. Let us define
\begin{equation*}
\begin{cases}
A^*(t) = \displaystyle\max_{x \in (0,1)} A(x,t), \qquad \hat A(t) = \displaystyle\min_{x \in (0,1)} A(x,t), \\
\zeta^*(t) = \displaystyle\max_{x \in (0,1)} \zeta(x,t), \qquad \hat \zeta(t) = \displaystyle\min_{x \in (0,1)} \zeta(x,t),
\end{cases}
\end{equation*}
and we observe that $A^*(t) < 0$, $\hat A(t) < 0$ and that $\zeta^*(t) / \hat \zeta(t)$ does not depend on $t$ and is equal to some constant $\zeta_0 \in \R$. Moreover, if $\lambda$ is sufficiently large we can suppose
\begin{equation}
\label{eq:comp_pesos}
3 A^*(t) < 2 \hat A(t) < 0.
\end{equation}

Let us define 
$$
\hat \Gamma(\phi,\varrho,\tilde \varrho) = \int_Q e^{2s \hat A}[(s\lambda) \hat \zeta b^2 a (|\phi_x|^2+|\varrho_x|^2+|\tilde \varrho_x|^2) + (s\lambda)^2 {\hat \zeta}^2 b^2 (|\phi|^2 + |\varrho|^2 + |\tilde \varrho|^2)]dxdt.
$$

Thus, Proposition \ref{prop:carleman} and Corollary \ref{cor:observability}  imply directly the following corollary where the weights depend only on $t$.

\begin{coro} %\cite{DemarqueLimacoViana_deg_sys2020}
\label{cor:carleman_pesos_t}
There exist positive constants $C$, $\lambda_0$ and $s_0$ such that, for any $s \geq s_0$, $\lambda \geq \lambda_0$ and any $\phi^T \in L^2(Q)$, the corresponding solution $(\phi,\psi^1,\psi^2)$ of \eqref{eq:adjoint_optimality_system} satisfies
\begin{equation}\label{pesos_t_Carleman for eq:adjoint_optimality_system}
\begin{split}
\|\phi(0)\|^2_{L^2(0,1)} +
\hat \Gamma(\phi,\psi^1,\psi^2) \leq & C \left( \int_Q e^{2s A^*} (\zeta^*)^4 (|F|^2+|F_1|^2+|F_2|^2)dxdt + \int_{O\times(0,T)} e^{2s A^*} (\zeta^*)^8 |\phi|^2 dxdt\right).
\end{split}
\end{equation}
\end{coro}

\subsection{Global null controllability of the linearized system}

Let us define the weights: 
\begin{equation}\label{eq:weights_rhos}
   \left\{ \begin{array}{l}
 \rho_0 = e^{-sA^*}  (\zeta^*)^{-2}, \qquad   \rho_1 = e^{-sA^*}  (\zeta^*)^{-4},\\  \rho_2 = e^{-3sA^*/2}  \hat \zeta^{-1},  \qquad \hat{\rho} = e^{-sA^*} (\zeta^*)^{-3},      
    \end{array}\right.
\end{equation}
satisfying
\begin{equation}\label{eq:compara_rhos}
 \rho_{1}\leq C \hat{\rho}\leq C\rho_{0}\leq C\rho_{2}, \qquad \hat{\rho}^{2}=\rho_{1}\rho_{0} \qquad \text{and} \qquad \rho_2 \leq C \rho_1^2.   
\end{equation}

In particular, Corollary \ref{cor:carleman_pesos_t} and \eqref{eq:comp_pesos} imply 
\begin{equation}
\label{eq:carleman_simples}
\begin{split}
\|\phi(0)\|^2_{L^2(0,1)} +
\int_Q \rho_2^{-2} (|\phi|^2 + |\psi^{1}|^2 + |\psi^{2}|^2)dxdt
%\hat \Gamma(\phi,\psi^1,\psi^2) 
\leq & \ C \left( \int_Q \rho_0^{-2} (|F|^2+|F_1|^2+|F_2|^2)dxdt \right. \\
& \left. + \int_{O\times(0,T)} \rho_1^{-2} |\phi|^2 dxdt\right).
\end{split}
\end{equation}

In the following theorem we show the global null controllability of the linearized system \eqref{eq:linearized_system}. In particular, since the weight $\rho_0$ blows up at $t=T$, \eqref{estimate for solution} shows that $y(\cdot,T)=0$ in $[0,1]$. The same vanishing conclusion holds for $p^i$, $i=1,2$ (which are used to define the follower's controls $v^i$) and the leader control $h$.
%Furthermore, estimates \eqref{des Proposition 5} and \eqref{des Proposition 6} give additional estimates verified by the derivatives of the states $(y,p^1,p^2)$ that are needed for the local null controllability of the nonlinear system in Section \ref{control for nonlinear system}.

\begin{teo}\label{theorem case linear}
If $y_{0}\in L^{2}(0,1)$, $\rho_2 H$, $\rho_2 H_1$, $\rho_2 H_2 \in L^2(Q)$, then there exists a control $h\in L^{2}(O\times (0,T))$ with associated states $y, p^{1}, p^{2}\in C^{0}([0,T];L^{2}(0,1))\cap L^{2}(0,T;H^{1}_{a})$  from \eqref{eq:linearized_system} such that
\begin{equation}\label{estimate for solution}
\int_Q \rho_0^2 |y|^2 dxdt+ \int_Q \rho_0^2 (|p^1|^2 +  |p^2|^2)dxdt + \int_{O \times (0,T)} \rho_1^2 |h|^2dxdt \leq
C \kappa_{0}(H,H_{1},H_{1},y_{0}),   
\end{equation}
where $\kappa_{0}(H,H_{1},H_{1},y_{0})= |\rho_2 H|^2_{L^2(Q)} + |\rho_2 H_1|^2_{L^2(Q)} + |\rho_2 H_2|^2_{L^2(Q)} + |y_0|^2_{L_2(0,1)}$. In particular, $y(x,T)=0$, for all $x\in [0,1]$. 

%Furthermore, 
%\begin{equation}\label{des Proposition 5}
%    \begin{array}{c}
%\displaystyle\sup_{[0,T]}(\hat{\rho}^{2}\|y\|^{2}_{L^{2}(0,1)}) + \displaystyle\sup_{[0,T]}(\hat{\rho}^{2}\|p^{1}\|^{2}_{L^{2}(0,1)}) + \displaystyle\sup_{[0,T]}(\hat{\rho}^{2}\|p^{2}\|^{2}_{L^{2}(0,1)})\vspace{0.1cm}\\
%+\displaystyle\int_{Q}\hat{\rho}^{2} a(x)(|y_{x}|^{2} + |p^{1}_{x}|^{2}+|p^{2}_{x}|^{2})dxdt\ \leq C \kappa_{0}(H,H_{1},H_{2},y_{0})  
%    \end{array}
%\end{equation}
%and, if $y_{0}\in H^{1}_{a}(0,1)$ 
%\begin{equation}\label{des Proposition 6}
%    \begin{array}{c}
%\displaystyle\sup_{[0,T]}(\rho_{1}^{2}\|\sqrt{a}y_{x} \|^{2}_{L^{2}(0,1)})+ \displaystyle\sup_{[0,T]}(\rho_{1}^{2}\|\sqrt{a}p^{1}_{x} \|^{2}_{L^{2}(0,1)}) + \displaystyle\sup_{[0,T]}(\rho_{1}^{2}\|\sqrt{a}p^{2}_{x} \|^{2}_{L^{2}(0,1)})\\
%+ \displaystyle\int_{Q}\rho_{1}^{2}(|y_{t}|^{2}+|p^{1}_{t}|^{2}+|p^{2}_{t}|^{2}+|(a(x)y_{x})_{x}|^{2} + |(a(x)p^{1}_{x})_{x}|^{2}+ |(a(x)p^{2}_{x})_{x}|^{2})dxdt\\
%\leq C \kappa_{1}(H,H_{1},H_{2},y_{0}),  
%    \end{array}
%\end{equation}
%where $\kappa_{1}(H,H_{1},H_{1},y_{0})= |\rho_2 H|^2_{L^2(Q)} + |\rho_2 H_1|^2_{L^2(Q)} + |\rho_2 H_2|^2_{L^2(Q)} + \|y_0\|^2_{H^{1}_{a}}$. 
\end{teo}

%\textcolor{red}{Los coeficientes del sistema de optimalidad de orden uno y cero no son iguales para ambas ecuaciones, como en el caso de Miguel, para definir un único operador $L$ y $L^*$. Como salvar?}

\begin{proof}
Let us denote by $$L\varphi =  \varphi_{t} - \frac{1}{l(t)^2}(a(x)\varphi_{x})_{x} - \frac{l'(t)}{l(t)}x \varphi_x + c(x,t)\varphi + d(x,t)\sqrt{a} \varphi_x$$ and $$L^{\ast}\varphi=-\varphi_{t} - \frac{1}{l(t)^2} (a(x)\varphi_{x})_{x} + \frac{l'(t)}{l(t)} (x \varphi)_x + c(x,t)\varphi - (d(x,t)\sqrt{a} \varphi)_x.$$
%(\textcolor{green}{Temos que ver sobre a função regular que substitui a função caracteristica. Se vamos defini-la no inicio ou somente aqui, como no trabalho do Danny}).
Then, we define 
\begin{equation*}
\begin{array}{l}
    \mathcal{P}_{0} = \{ (\phi,\psi^{1},\psi^{2})\in C^{2}(\overline{Q})^{3}; \phi(0,t)=\phi(1,t) = 0\ \text{a.e in}\ (0,T),\ \psi^{i}(0,t)=\psi^{i}(1,t)=0\, \text{a.e in}\, (0,T),\\ \hspace{1.2cm}\psi^{i}(\cdot,0)=0\ \text{in}\ \Omega, \ i=\lbrace 1, 2\rbrace\}
\end{array}
\end{equation*}
and the application $b:\mathcal{P}_{0}\times \mathcal{P}_{0}\rightarrow \mathbb{R}$ given by
\begin{equation}
\label{eq:def_b}
    \begin{array}{l}
      \hspace{-1.5cm}b((\tilde{\phi},\tilde{\psi^{1}},\tilde{\psi^{2}}),(\phi,\psi^{1},\psi^{2})) \\ = \displaystyle\int_{Q}\rho_{0}^{-2}(L^{\ast}\tilde{\phi}-\alpha_{1}\tilde{\psi^{1}}1_{O_{d}}-\alpha_{2}\tilde{\psi^{2}}1_{O_{d}})(L^{\ast}\phi-\alpha_{1}\psi^{1}1_{O_{d}}-\alpha_{2}\psi^{2}1_{O_{d}})\ dx\ dt \vspace{0.1cm}\\
      + \displaystyle\sum_{i=1}^{2}\displaystyle\int_{Q}\rho_{0}^{-2}(L\tilde{\psi^{i}} + \frac{1}{\mu_{i}}\tilde{\phi}1_{O_{i}})(L\psi^{i} + \frac{1}{\mu_{i}}\phi1_{O_{i}})\ dx\ dt\vspace{0.1cm}\\
      +\displaystyle\int_{O\times (0,T)}\rho_{1}^{-2}\tilde{\phi}\phi\ dx\ dt,\,\, \forall (\phi,\psi^{1},\psi^{2}),(\tilde{\phi},\tilde{\psi^{1}},\tilde{\psi^{2}})\in \mathcal{P}_{0},
    \end{array}
\end{equation}
which is bilinear on $\mathcal{P}_{0}$ and defines an inner product. Indeed, taking $(\tilde{\phi},\tilde{\psi^{1}},\tilde{\psi^{2}})=(\phi,\psi^{1},\psi^{2})$ in \eqref{eq:def_b}, we have that, by \eqref{eq:carleman_simples}  $b(\cdot,\cdot)$ is positive definite. The other properties are straightforwardly verified.

Let us consider the space $\mathcal{P}$ the completion of $\mathcal{P}_{0}$ for the norm associated to $b(\cdot,\cdot)$ (which we denote by $\|.\|_{\mathcal{P}}$). Then, $b(\cdot,\cdot)$ is symmetric, continuous and coercive bilinear form on $\mathcal{P}$.

Now, let us define the functional linear $\ell :\mathcal{P}\rightarrow\mathbb{R}$ as \begin{equation*}
    \langle\ell, (\phi,\psi^{1},\psi^{2})\rangle = \displaystyle\int_{0}^{1}y_{0}\phi(0)dx + \displaystyle\int_{Q}(H\phi + H_{1}\psi^{1} + H_{2}\psi^{2})dx dt.
\end{equation*}

Note that $\ell$ is a bounded linear form on $\mathcal{P}$. Indeed, applying the classical Cauchy-Schwartz inequality 
%$|u\cdot v|\leq |u||v|$, for $u,v\in 
in $\mathbb{R}^{4}$ and using \eqref{eq:carleman_simples}, we get
\begin{equation}\label{l limitado}
    \begin{array}{l}
       |\langle\ell, (\phi,\psi^{1},\psi^{2})\rangle|\leq  |y_{0}|_{L^{2}(0,1)}|\phi(0)|_{L^{2}(0,1)} + |\rho_{2}H|_{L^{2}(Q)}|\rho_{2}^{-1}\phi|_{L^{2}(Q)} \\
       \hspace{3cm} + |\rho_{2}H_{1}|_{L^{2}(Q)}|\rho_{2}^{-1}\psi^{1}|_{L^{2}(Q)} + |\rho_{2}H_{2}|_{L^{2}(Q)}|\rho_{2}^{-1}\psi^{2}|_{L^{2}(Q)}\\
       \leq C(|y_{0}|^{2}_{L^{2}(0,1)} + |\rho_{2}H|^{2}_{L^{2}(Q)} + |\rho_{2}H_{1}|^{2}_{L^{2}(Q)}  + |\rho_{2}H_{2}|^{2}_{L^{2}(Q)})^{1/2}\left(b((\phi,\psi^{1},\psi^{2}),(\phi,\psi^{1},\psi^{2}))\right)^{1/2}\\
     \leq C(|y_{0}|^{2}_{L^{2}(0,1)} + |\rho_{2}H|^{2}_{L^{2}(Q)} + |\rho_{2}H_{1}|^{2}_{L^{2}(Q)}  + |\rho_{2}H_{2}|^{2}_{L^{2}(Q)})^{1/2}\|(\phi,\psi^{1},\psi^{2})\|_{\mathcal{P}},
    \end{array}
\end{equation}
for all $(\phi,\psi^{1},\psi^{2})\in\mathcal{P}$. Consequently, in view of Lax-Milgram's theorem, there is only one $(\hat{\phi},\hat{\psi^{1}},\hat{\psi^{2}})\in\mathcal{P}$ satisfying
\begin{equation}\label{eq: por Lax-M.}    b((\hat{\phi},\hat{\psi^{1}},\hat{\psi^{2}}),(\phi,\psi^{1},\psi^{2})) = \langle\ell, (\phi,\psi^{1},\psi^{2})\rangle,\,\, \forall (\phi,\psi^{1},\psi^{2})\in\mathcal{P}.
\end{equation}

Let us set
\begin{equation}\label{definição de y, pi, h}
    \left\{\begin{array}{lll}
      y = \rho_{0}^{-2}(L^{\ast}\hat{\phi}-\alpha_{1}\hat{\psi^{1}}1_{O_{d}}-\alpha_{2}\hat{\psi^{2}}1_{O_{d}})&\text{in} & Q,\\
      p^{i} = \rho_{0}^{-2}(L\hat{\psi^{i}} + \dfrac{1}{\mu_{i}}\hat{\phi}1_{O_{i}}),\, i=\{1,2\} &\text{in} & Q,\\
      h = -\rho_{1}^{-2}\hat{\phi}1_{O} &\text{in} & Q.
\end{array}\right.
\end{equation}
Then, replacing \eqref{definição de y, pi, h} in \eqref{eq: por Lax-M.} we have
\begin{equation*}
    \begin{array}{l}
         \displaystyle\int_{Q}y B\,dx\,dt + \displaystyle\int_{Q}(p^{1}B_{1} + p^{2}B_{2})\,dx\,dt\\
         = \displaystyle\int_{0}^{1}y_{0}\phi(0)dx + \displaystyle\int_{O\times(0,T)} h \phi\,dx\,dt+\displaystyle\int_{Q}(H\phi + H_{1}\psi^{1} + H_{2}\psi^{2})dx\ dt,
    \end{array}
\end{equation*}
where $(\phi,\psi^{1},\psi^{2})$ is a solution of the system
\begin{equation*}
    \left\{\begin{array}{lll}
       L^{\ast}\phi = B +\alpha_{1}{\psi^{1}}1_{O_{d}} + \alpha_{2}{\psi^{2}}1_{O_{d}}  &\text{in}& Q,  \\
       L\psi^{i} = B_{i} - \dfrac{1}{\mu_{i}}{\phi}1_{O_{i}}  \  &\text{in}& Q, \\
       \phi(0,t)=\phi(1,t)=0 & \text{on} & (0,T), \\ \psi^i(0,t)=\psi^i(1,t)=0 & \text{on}& (0,T), \\
       \phi(\cdot,T)=0,\, \psi^{i}(\cdot,0)=0 &\text{in}& \Omega,
   \end{array}\right.
\end{equation*}
for $i=\lbrace 1, 2\rbrace$. Therefore, $(y,p^{1},p^{2})$ is a solution by transposition of \eqref{eq:linearized_system}. Also, as $(\hat{\phi},\hat{\psi^{1}},\hat{\psi^{2}})\in\mathcal{P}$ and 
$H, H_{1}, H_{2}\in L^{2}(Q)$,  
using the well-posedness result of Appendix \ref{appendix A} applied to a linear equation we obtain
%Proposition \ref{Regularidade para linear system}, we obtain
$$y,p^{1},p^{2}\in C^{0}([0,T];L^{2}(0,1))\cap L^{2}(0,T;H^{1}_{a}).$$

%\textcolor{red}{Esta faltando una proposition tipo la 3.1 del artículo con Joao y Suerlan. Esa proposicion tiene el label "Regularidade para linear system",por eso sale el ??. Corresponderia al Teorema \ref{teo34noauto} de este pdf.}

Moreover, from \eqref{l limitado} and \eqref{eq: por Lax-M.}
\begin{equation*}
\left(b((\hat{\phi},\hat{\psi^{1}},\hat{\psi^{2}}),(\hat{\phi},\hat{\psi^{1}},\hat{\psi^{2}}))\right)^{1/2}  \leq C(|y_{0}|^{2}_{L^{2}(0,1)} + |\rho_{2}H|^{2}_{L^{2}(Q)} + |\rho_{2}H_{1}|^{2}_{L^{2}(Q)}  + |\rho_{2}H_{2}|^{2}_{L^{2}(Q)})^{1/2}
\end{equation*}
that is,
\begin{equation*}
\begin{array}{l}
\displaystyle\int_{Q}\rho_{0}^{2}|y|^{2}\ dxdt       + \displaystyle\sum_{i=1}^{2}\displaystyle\int_{Q}\rho_{0}^{2}|p^{i}|^{2}\ dxdt +\displaystyle\int_{O\times (0,T)}\rho_{1}^{2}|h|^{2} dx dt\\
\leq C(|y_{0}|^{2}_{L^{2}(0,1)} + |\rho_{2}H|^{2}_{L^{2}(Q)} + |\rho_{2}H_{1}|^{2}_{L^{2}(Q)}  + |\rho_{2}H_{2}|^{2}_{L^{2}(Q)}),
\end{array}
\end{equation*}
proving \eqref{estimate for solution}.

\end{proof}

%\newpage
In order to get the local null controllability of he nonlinear system we will need the following additional estimates.

\begin{propo}
\label{addicional_estimates_case_linear}
Under the hypothesis of Theorem \ref{theorem case linear}, we have, furthermore, that the control $h\in L^{2}(O\times (0,T))$ and the associated states $y, p^{1}, p^{2}\in C^{0}([0,T];L^{2}(0,1))\cap L^{2}(0,T;H^{1}_{a})$, solution of \eqref{eq:linearized_system}, satisfies the additional estimates 
\begin{equation}\label{des Proposition 5}
    \begin{array}{c}
\displaystyle\sup_{[0,T]}(\hat{\rho}^{2}\|y\|^{2}_{L^{2}(0,1)}) + \displaystyle\sup_{[0,T]}(\hat{\rho}^{2}\|p^{1}\|^{2}_{L^{2}(0,1)}) + \displaystyle\sup_{[0,T]}(\hat{\rho}^{2}\|p^{2}\|^{2}_{L^{2}(0,1)})\vspace{0.1cm}\\
+\displaystyle\int_{Q}\hat{\rho}^{2} a(x)(|y_{x}|^{2} + |p^{1}_{x}|^{2}+|p^{2}_{x}|^{2})dxdt\ \leq C \kappa_{0}(H,H_{1},H_{2},y_{0})  
    \end{array}
\end{equation}
and, if $y_{0}\in H^{1}_{a}(0,1)$ 
\begin{equation}\label{des Proposition 6}
    \begin{array}{c}
\displaystyle\sup_{[0,T]}(\rho_{1}^{2}\|\sqrt{a}y_{x} \|^{2}_{L^{2}(0,1)})+ \displaystyle\sup_{[0,T]}(\rho_{1}^{2}\|\sqrt{a}p^{1}_{x} \|^{2}_{L^{2}(0,1)}) + \displaystyle\sup_{[0,T]}(\rho_{1}^{2}\|\sqrt{a}p^{2}_{x} \|^{2}_{L^{2}(0,1)})\\
+ \displaystyle\int_{Q}\rho_{1}^{2}(|y_{t}|^{2}+|p^{1}_{t}|^{2}+|p^{2}_{t}|^{2}+|(a(x)y_{x})_{x}|^{2} + |(a(x)p^{1}_{x})_{x}|^{2}+ |(a(x)p^{2}_{x})_{x}|^{2})dxdt\\
\leq C \kappa_{1}(H,H_{1},H_{2},y_{0}),  
    \end{array}
\end{equation}
where $\kappa_{1}(H,H_{1},H_{1},y_{0})= |\rho_2 H|^2_{L^2(Q)} + |\rho_2 H_1|^2_{L^2(Q)} + |\rho_2 H_2|^2_{L^2(Q)} + \|y_0\|^2_{H^{1}_{a}}$. 
\end{propo}

\begin{proof}
%Let us denote by $L\varphi =  \varphi_{t} - (a(x)\varphi_{x})_{x} + c(x,t)\phi$ and $L^{\ast}\varphi=-\varphi_{t} - (a(x)\varphi_{x})_{x} + c(x,t)\phi$.
%(\textcolor{green}{Temos que ver sobre a função regular que substitui a função caracteristica. Se vamos defini-la no inicio ou somente aqui, como no trabalho do Danny}).
We proceed following the steps of \cite{DemarqueLimacoViana_deg_eq2018}.
However, as in \cite{Joao_Juan_Suerlan-25}, we notice that the system \eqref{eq:linearized_system} has its first equation acting forward in time and its other two backward in time. Therefore, to establish estimates \eqref{des Proposition 5} and \eqref{des Proposition 6} we need to change variables in the form $\tilde{p}^{i}(x,t)=p^{i}(x,T-t)$ and $ \tilde{H}_{i}(x,t)=H_{i}(x,T-t)$, with $i=\lbrace 1,2\rbrace$, so that we obtain a system with new variables $(\tilde {y},\tilde{p}^{i},\tilde{H}_{i})$ in which all its equations are forward in time. In this way, the new system have initial conditions $(\tilde{y}(\cdot,0),\tilde{p}^{1}(\cdot,0),\tilde{p}^{2}(\cdot, 0 ))=(y^{0},0,0)$ and the estimates with weights for the solution, as well as regularity results, will be equivalent to those of the system \eqref{eq:linearized_system}. Keeping this in mind, for simplicity, we will maintain the notations of \eqref{eq:linearized_system}, for $i=1,2$,
\begin{equation}
\label{eq:linearized_system new}
\left\{\begin{aligned}
&y_t - \frac{1}{\ell(t)^2}\left(a(x) y_x\right)_x - \frac{\ell'(t)}{\ell(t)}x y_x 
+ c(x,t)y + d(x,t)\sqrt{a}y_x = h 1_{O} -\frac{1}{\mu_1} p^1 1_{O_1} - \frac{1}{\mu_2} p^2 1_{O_2} + H & \text{in } Q, \\
&p_t^i - \frac{1}{\ell(t)^2} \left(a(x) p_x^i \right)_x + \frac{\ell'(t)}{\ell(t)}(x p^i)_x + c(x,t)p^i -d(x,t)\sqrt{a} p^i_x = \alpha_i y 1_{O_{i,d}} + H_i  & \text{in } Q,\\
%p_t^2 - \left(l(0) a(x) p_x^2 \right)_x  = \alpha_2 y 1_{O_{2,d}} + H_2,  & \text{in } Q, \\
&p^i(0,t) = p^i(1,t) = 0 & \text{on}\ (0,T), \\
&p^i(\cdot,0) = 0 & \text{in } \Omega, \\
&y(0,t)=y(1,t)=0 & \text{on}\ (0,T), \\
&y(\cdot,0) = y^0 & \text{in } \Omega,
\end{aligned}
\right.
\end{equation}
Or, more generally, 
\begin{equation}
\label{eq:linearized_system new}
\left\{\begin{aligned}
&y_t - b(t)\left(a(x) y_x\right)_x  
+ c_1(x,t)y + d_1(x,t)\sqrt{a}y_x = h 1_{O} -\frac{1}{\mu_1} p^1 1_{O_1} - \frac{1}{\mu_2} p^2 1_{O_2} + H & \text{in } Q, \\
&p_t^i - b(t) \left(a(x) p_x^i \right)_x + c_2(x,t)p^i + d_2(x,t)\sqrt{a} p^i_x = \alpha_i y 1_{O_{i,d}} + H_i  & \text{in } Q,\\
%p_t^2 - \left(l(0) a(x) p_x^2 \right)_x  = \alpha_2 y 1_{O_{2,d}} + H_2,  & \text{in } Q, \\
&p^i(0,t) = p^i(1,t) = 0 & \text{on}\ (0,T), \\
&p^i(\cdot,0) = 0 & \text{in } \Omega, \\
&y(0,t)=y(1,t)=0 & \text{on}\ (0,T), \\
&y(\cdot,0) = y^0 & \text{in } \Omega,
\end{aligned}
\right.
\end{equation}
where the coefficients $b(t)=\frac{1}{\ell(t)^2}$,  $c_i$, $d_i$, $i=1,2$ are bounded.

Let us multiply $\eqref{eq:linearized_system new}_{1}$ by $\hat{\rho}^{2} y$, $\eqref{eq:linearized_system new}_{2}$ by $\hat{\rho}^{2} p^{i}$ and integrate over $[0,1]$. Hence, using that $\hat{\rho}^{2} = \rho_{0}\rho_{1}$, and $\rho_{1}\leq C\rho_{2}$, we compute
{
\begin{equation}\label{second estimate}
    \begin{array}{l}
     \dfrac{1}{2}\dfrac{d}{dt}\displaystyle\int_{0}^{1}\hat{\rho}^{2}|y|^{2}dx +\sum_{i=1}^{2}\dfrac{1}{2}\dfrac{d}{dt}\displaystyle\int_{0}^{1}\hat{\rho}^{2} | p^{i}|^{2}dx + b(t) \displaystyle\int_{0}^{1}\hat{\rho}^{2}a(x)|y_{x}|^{2} dx + b(t)  \sum_{i=1}^{2}\displaystyle\int_{0}^{1}\hat{\rho}^{2}a(x)|p_{x}^{i}|^{2}dx\hspace{0.1cm}\\
     + \displaystyle\int_0^1 \hat{\rho} d_1(x,t) \sqrt{a(x)}y_x y dx +\sum_{i=1}^{2} \int_0^1 \hat{\rho} d_2(x,t) \sqrt{a(x)}p^i_x p^i dx  \\
          \leq C\left(\displaystyle\int_{0}^{1}\rho_{0}^{2}|y|^{2}dx + \sum_{i=1}^{2}\displaystyle\int_{0}^{1}\rho_{0}^{2}|p^{i}|^{2}dx + \displaystyle\int_{O}\rho_{1}^{2}|h|^{2}dx + \displaystyle\int_{0}^{1}\rho_{2}^{2}|H|^{2}dx+
          \sum_{i=1}^{2}\displaystyle\int_{0}^{1}\rho_{2}^{2}|H_{i}|^{2}dx\right)\hspace{0.1cm}\\
          + \ \mathcal{M},
          %+ |\mathcal{N}|,
    \end{array}
\end{equation}}
where $\mathcal{M} = \displaystyle\int_{0}^{1}\hat{\rho}(\hat{\rho})_{t}|y|^{2}dx +\sum_{i=1}^{2}\displaystyle\int_{0}^{1}\hat{\rho}(\hat{\rho})_{t}|p^{i}|^{2}dx,$
and $(\cdot)_{t}=\frac{d}{dt}(\cdot)$. To facilitate notation, we will omit the sum sign. Recall that
$A^*(t) = C_1 \tau(t)$, and $\zeta^*(t) = C_2 \tau(t)$, then we have that $A^*_{t}=\bar C (\zeta^*)_{t}$ and consequently
\begin{equation*}
    \begin{split}
      \hat{\rho}(\hat{\rho})_{t} &= e^{-sA^*}(\zeta^*)^{-3}\left(-se^{-sA^*} A^*_{t}(\zeta^*)^{-3} -3 e^{-sA^*}(\zeta^*)^{-4}(\zeta^*)_{t}\right)    \\
       &= -e^{-2sA^*}(\zeta^*)^{-4}(\zeta^*)_{t}\left(s(\zeta^*)^{-2}\bar{C} + 3(\zeta^*)^{-3} \right)  \\
        &=-\rho_{0}^{2}(\zeta^*)_{t}\left(s(\zeta^*)^{-2}\bar{C} + 3(\zeta^*)^{-3} \right).
    \end{split}
\end{equation*}
Thus, for any $t\in [0,T)$,
\begin{equation*}
    \begin{array}{l}
        |\hat{\rho}(\hat{\rho})_{t}|\leq C\rho_{0}^{2}\tau^{2}|s(\zeta^*)^{-2}\bar{C} + 3(\zeta^*)^{-3}|  
         \leq C\rho_{0}^{2}|s\bar{C}+3(\zeta^*)^{-1}| \leq C\rho_{0}^{2},
    \end{array}
\end{equation*}
and we obtain
\begin{equation*}
%\label{estimate of M}
    \mathcal{M}\leq C \displaystyle\int_{0}^{1}{\rho_{0}^{2}}(|y|^{2} +  |p^{i}|^{2})dx.
\end{equation*}

Therefore, using Young's inequality, \eqref{second estimate} becomes, for a small $\epsilon>0$, 
\begin{equation*}
    \begin{array}{l}
       \dfrac{1}{2}\dfrac{d}{dt}\displaystyle\int_{0}^{1}\hat{\rho}^{2}(|y|^{2}+| p^{i}|^{2})dx + b(t) \displaystyle\int_{0}^{1}\hat{\rho}^{2} a(x)(|y_{x}|^{2}+|p^{i}_{x}|^{2})dx\vspace{0.1cm}\\
          \leq C\left(\displaystyle\int_{0}^{1}\rho_{0}^{2}(|y|^{2} + |p^{i}|^{2})dx + \displaystyle\int_{O}\rho_{1}^{2}|h|^{2}dx + \displaystyle\int_{0}^{1}\rho_{2}^{2}(|H|^{2}+|H_{i}|^{2})dx\right)  \\
          + \epsilon \displaystyle\int_{0}^{1}\hat{\rho}^{2} a(x)(|y_{x}|^{2}+|p^{i}_{x}|^{2})dx + C_\epsilon \left( \int_0^1 |y|^2 + |p^i|^2 \right).
    \end{array}
\end{equation*}

Thus, since $b(t)$ is bounded and $\rho_0$ is bounded by bellow, there is constant $D>0$ such that
\begin{equation*}
    \begin{array}{l}
       \dfrac{1}{2}\dfrac{d}{dt}\displaystyle\int_{0}^{1}\hat{\rho}^{2}(|y|^{2}+| p^{i}|^{2})dx + \displaystyle\int_{0}^{1}\hat{\rho}^{2} a(x)(|y_{x}|^{2}+|p^{i}_{x}|^{2})dx\vspace{0.1cm}\\
          \leq D\left(\displaystyle\int_{0}^{1}\rho_{0}^{2}(|y|^{2} + |p^{i}|^{2})dx + \displaystyle\int_{O}\rho_{1}^{2}|h|^{2}dx + \displaystyle\int_{0}^{1}\rho_{2}^{2}(|H|^{2}+|H_{i}|^{2})dx\right)
    \end{array}
\end{equation*}
and, integrating in time, we conclude \eqref{des Proposition 5}.

%--------------------------------------
%\textcolor{red}{Falta verificar bien desde aqui}

Now, to prove \eqref{des Proposition 6}, multiply $\eqref{eq:linearized_system new}_{1}$ by $\rho_{1}^{2} y_{t}$ and $\eqref{eq:linearized_system new}_{2}$ by $\rho_{1}^{2} p^{i}_{t}$ and integrate over $[0,1]$. Thus, we get
%\textcolor{red}{(Precisou usar $y_{xt}=y_{tx}$, perguntar se vale)}

\begin{equation}\label{third estimate}
    \begin{array}{l}
\displaystyle\int_{0}^{1}\rho_{1}^{2}(|y_{t}|^{2}+|p^{i}_{t}|^{2})dx + \dfrac{1}{2} b(t)
 \displaystyle\int_{0}^{1}\rho_{1}^{2} a(x) \dfrac{d}{dt}(|y_{x}|^{2}+|p^{i}_{x}|^{2})dx  \vspace{0.1cm}\\
 + \displaystyle\int_0^1 \hat{\rho} c_1(x,t) y y_t dx +\sum_{i=1}^{2} \int_0^1 \hat{\rho} c_2(x,t) p^i p^i_t dx
 + \int_0^1 \hat{\rho} d_1(x,t) \sqrt{a}y_x y_t dx +\sum_{i=1}^{2} \int_0^1 \hat{\rho} d_2(x,t) \sqrt{a}p^i_x p^i_t dx  \\
         \leq C\left(\displaystyle\int_{0}^{1}\rho_{0}^{2}(|y|^{2} + |p^{i}|^{2})dx + \displaystyle\int_{O}\rho_{1}^{2}|h|^{2}dx + \displaystyle\int_{0}^{1}\rho_{2}^{2}(|H|^{2}+|H_{i}|^{2})dx\right) \vspace{0.1cm}\\
        + \ \dfrac{1}{4}\displaystyle\int_{0}^{1}\rho_{1}^{2}(|y_{t}|^{2}+|p^{i}_{t}|^{2})dx.
    \end{array}
\end{equation}
Thus, using that $\hat\rho = \rho_0 \rho_1$, Young's inequality and the boundedness of $c_i$ and $d_i$, we get
\begin{equation}\label{third estimate}
    \begin{array}{l}
\displaystyle\int_{0}^{1}\rho_{1}^{2}(|y_{t}|^{2}+|p^{i}_{t}|^{2})dx + \dfrac{1}{2} b(t) \dfrac{d}{dt}\displaystyle\int_{0}^{1}\rho_{1}^{2} a(x)(|y_{x}|^{2}+|p^{i}_{x}|^{2})dx  \vspace{0.1cm}\\
         \leq D\left(\displaystyle\int_{0}^{1}\rho_{0}^{2}(|y|^{2} + |p^{i}|^{2})dx + \displaystyle\int_{O}\rho_{1}^{2}|h|^{2}dx + \displaystyle\int_{0}^{1}\rho_{2}^{2}(|H|^{2}+|H_{i}|^{2})dx\right) \vspace{0.1cm}\\
        + \ \dfrac{1}{2}\displaystyle\int_{0}^{1}\rho_{1}^{2}(|y_{t}|^{2}+|p^{i}_{t}|^{2})dx + |\widetilde{\mathcal{M}}|,
    \end{array}
\end{equation}
where $\widetilde{\mathcal{M}}= \dfrac{1}{2}\displaystyle\int_{0}^{1}(\rho_{1}^{2})_{t}\, a(x)(|y_{x}|^{2}+|p^{i}_{x}|^{2})dx$.
Since $|\zeta_{t}|\leq C\zeta^{2}$,  
we have that $|(\rho^{2}_{1})_{t}|\leq C\hat{\rho}^{2}$. 
%and $|(\rho_{1}^{2})_{x}|^{2}\rho_{1}^{-2}\leq C\hat{\rho}^{2}$. 
Hence, $$|\widetilde{\mathcal{M}}| \leq C\displaystyle\int_{0}^{1}\hat{\rho}^{2}\, a(x)(|y_{x}|^{2}+|p^{i}_{x}|^{2})dx.$$

Thus, \eqref{third estimate} gives 
\begin{equation*}
    \begin{array}{l}
\dfrac{1}{2}\displaystyle\int_{0}^{1}\rho_{1}^{2}(|y_{t}|^{2}+|p^{i}_{t}|^{2})dx + \dfrac{1}{2} b(t)  \dfrac{d}{dt}\displaystyle\int_{0}^{1}\rho_{1}^{2} a(x)(|y_{x}|^{2}+|p^{i}_{x}|^{2})dx  \vspace{0.1cm}\\
         \leq D\left(\displaystyle\int_{0}^{1}\rho_{0}^{2}(|y|^{2} + |p^{i}|^{2})dx + \displaystyle\int_{O}\rho_{1}^{2}|h|^{2}dx + \displaystyle\int_{0}^{1}\rho_{2}^{2}(|H|^{2}+|H_{i}|^{2})dx\right) \vspace{0.1cm}\\
        \, + \, C\displaystyle\int_{0}^{1}\hat{\rho}^{2}\, a(x)(|y_{x}|^{2}+|p^{i}_{x}|^{2})dx.
    \end{array}
\end{equation*}
 Integrating the previous inequality from $0$ to $t$ and using the boundedness of $b(t)$ and the first estimate, \eqref{des Proposition 5}, we arrive at
 \begin{equation}\label{primeira da des proposição 6}
     \begin{array}{c}
\displaystyle\sup_{[0,T]}(\rho_{1}^{2}\|\sqrt{a}y_{x} \|^{2}_{L^{2}(0,1)})+ \displaystyle\sup_{[0,T]}(\rho_{1}^{2}\|\sqrt{a}p^{i}_{x} \|^{2}_{L^{2}(0,1)}) + \displaystyle\int_{Q}\rho_{1}^{2}(|y_{t}|^{2}+|p^{i}_{t}|^{2})dxdt\\
\leq C \kappa_{1}(H,H_{1},H_{2},y_{0}).
     \end{array}
 \end{equation}

%--------------------------------------
%\textcolor{red}{Falta adaptar desde aqui}

Finally, to conclude \eqref{des Proposition 6}, it remains to estimate $\int_{Q}\rho_{1}^{2}(|(a(x)y_{x})_{x}|^{2} + |(a(x)p^{i}_{x})_{x}|^{2})dxdt$. To do this, it is enough to multiply $\eqref{eq:linearized_system new}_{1}$ by $-\rho_{1}^{2}(a(x)y_{x})_{x}$ and $\eqref{eq:linearized_system new}_{2}$ by $-\rho_{1}^{2}(a(x)p^{i}_{x})_{x}$, integrate over $[0,1]$ as before. We get
\begin{equation*}
\begin{split}
\frac{1}{2}\int_0^1 \rho_1^2 a(x) \frac{d}{dt} (|y_x|^2 + |p_x^i|^2) dx + b(t) \int_0^1 \rho_1^2 \left( |(a y_x)_x|^2 + |(a p^i_x)_x|^2 \right) \\
- \int_0^1 \rho_1^2 \left( (a y_x)_x c_1 y + (a p^i_x))_x c_2 p^i \right) - \int_0^1 \rho_1^2 \left( (a y_x)_x d_1 \sqrt{a} y_x + (a p^i_x)_x d_2 \sqrt{a} p^i_ x \right) \\
\leq C \left( \int_0^1 \rho_1^2 \left( (|h| + |p^i| + |H)|(a y_x)_x| + (|y| + |H_i|)|(a p^i_x)_x| \right) \right)
\end{split}    
\end{equation*}
Thus, for a small $\epsilon>0$, 
\begin{equation*}
    \begin{array}{l}
         \dfrac{1}{2}\dfrac{d}{dt}\displaystyle\int_{0}^{1}\rho_{1}^{2} \,a(x)(|y_{x}|^{2}+|p^{i}_{x}|^{2})dx + b(t) \displaystyle\int_{0}^{1}\rho_{1}^{2}(|(a(x)y_{x})_{x}|^{2} + |(a(x)p^{i}_{x})_{x}|^{2} )dx\vspace{0.1cm}\\
         \leq C_\epsilon\left(\displaystyle\int_{0}^{1}\rho_{1}^{2}(|y|^{2} + |p^{i}|^{2})dx + \displaystyle\int_{O}\rho_{1}^{2}|h|^{2}dx + \displaystyle\int_{0}^{1}\rho_{1}^{2}(|H|^{2}+|H_{i}|^{2})dx\right) \vspace{0.1cm}
         \\
         + \displaystyle \epsilon \int_{0}^{1}\rho_{1}^{2}\left(|(a y_x)_x|^{2} + |(a p^i_x)_x|^{2}\right)dx
         + \displaystyle \epsilon \int_{0}^{1}\rho_{1}^{2}\left(|(a y_x)_x|^{2} + |(a p^i_x)_x|^{2}\right)dx \\
         + C_\epsilon \int_{0}^{1}\rho_{1}^{2}\left( |c_1 y|^2 + |c_2 p^i|^2 \right) + |\mathcal{N}|,
    \end{array}
\end{equation*}
where $|\mathcal{N}| = \dfrac{1}{2}\displaystyle\int_{0}^{1}(\rho_{1}^{2})_{t}\, a(x)(|y_{x}|^{2}+|p^{i}_{x}|^{2})dx$.
Since $|\zeta_{t}|\leq C\zeta^{2}$,  
we have that $|(\rho^{2}_{1})_{t}|\leq C\hat{\rho}^{2}$. 
%and $|(\rho_{1}^{2})_{x}|^{2}\rho_{1}^{-2}\leq C\hat{\rho}^{2}$. 
Hence, $$|\mathcal{N}| \leq C\displaystyle\int_{0}^{1}\hat{\rho}^{2}\, a(x)(|y_{x}|^{2}+|p^{i}_{x}|^{2})dx.$$ 

Thus, since $b(t)$,  $c_i$ and $d_i$ are bounded, and $\rho_1 \leq C\rho_0 \leq C\rho_2$, we get, for some $D>0$,
\begin{equation*}
    \begin{array}{l}
         \dfrac{1}{2}\dfrac{d}{dt}\displaystyle\int_{0}^{1}\rho_{1}^{2} \,a(x)(|y_{x}|^{2}+|p^{i}_{x}|^{2})dx + \dfrac{1}{2}\displaystyle\int_{0}^{1}\rho_{1}^{2}(|(a(x)y_{x})_{x}|^{2} + |(a(x)p^{i}_{x})_{x}|^{2} )dx\vspace{0.1cm}\\
         \leq D\left(\displaystyle\int_{0}^{1}\rho_{0}^{2}(|y|^{2} + |p^{i}|^{2})dx + \displaystyle\int_{O}\rho_{1}^{2}|h|^{2}dx + \displaystyle\int_{0}^{1}\rho_{2}^{2}(|H|^{2}+|H_{i}|^{2})dx\right) \vspace{0.1cm}\\
         +\, D\displaystyle\int_{0}^{1}\hat{\rho}^{2}\, a(x)(|y_{x}|^{2}+|p^{i}_{x}|^{2})dx 
         %+ C\displaystyle\int_{0}^{1}\rho_{1}^{2}(|y_{t}|^{2} + |p^{i}_{t}|^{2})dx.
    \end{array}
\end{equation*}

Therefore, integrating over time and using estimates \eqref{des Proposition 5} and \eqref{primeira da des proposição 6}  we obtain
\begin{equation}\label{segunda da des proposição 6}
    \begin{array}{l}
\displaystyle\int_{Q}\rho_{1}^{2}(|(a(x)y_{x})_{x}|^{2} + |(a(x)p^{i}_{x})_{x}|^{2})dxdt  \leq C \kappa_{1}(H,H_{1},H_{2},y_{0}).     \end{array}
\end{equation}
From \eqref{primeira da des proposição 6} and \eqref{segunda da des proposição 6} we infer \eqref{des Proposition 6}.
\end{proof}
   
\section{Local null controllability of the nonlinear system}
\label{sec:control for nonlinear system}

All along this section we use the weights defined in \eqref{eq:weights_rhos} and Liusternik's (right) inverse function theorem (see \cite{Alekseev}) to obtain our results. Here $B_{r}(0)$ and $B_{\delta}(\zeta_{0})$ are open balls of radius $r$ and $\delta$, centered at $0$ and $\zeta_0$, respectively.
\begin{teo}[Liusternik]\label{Liusternik}
Let $ \mathcal{Y}$ and $ \mathcal{Z}$ be Banach spaces and let $\mathcal{A}:B_{r}(0)\subset  \mathcal{Y}\rightarrow  \mathcal{Z}$ be a $\mathcal{C}^{1}$ mapping. Let as assume that $\mathcal{A}^{\prime}(0)$ is onto and let us set $\mathcal{A}(0)=\zeta_{0}$. Then, there exist $\delta >0$, a mapping $W: B_{\delta}(\zeta_{0})\subset  \mathcal{Z}\rightarrow  \mathcal{Y}$ and a constant $K>0$ such that
\begin{equation*}
    W(z)\in B_{r}(0),\,\, \mathcal{A}(W(z))=z\,\, \text{and}\,\, \Vert W(z)\Vert_{ \mathcal{Y}}\leq K\Vert z-\zeta_{0}\Vert_{ \mathcal{Z}}\, \, \forall\, z\in B_{\delta}(\zeta_{0}).
\end{equation*}
In particular, $W$ is a local inverse-to-the-right of $\mathcal{A}$.
\end{teo}

 In order to do so, we define a map $\mathcal{A} : \mathcal{Y} \to \mathcal{Z}$ between suitable Banach spaces $\mathcal{Y}$ and $\mathcal{Z}$ whose definition and properties came from the controllability result of the linearized system and the additional estimates shown in Theorem \ref{theorem case linear} and Proposition \ref{addicional_estimates_case_linear}  in order to verify that the map $\mathcal{A}$ is well defined and verifies Liusternik's theorem hypothesis.

From the linearized system \eqref{eq:linearized_system}, we denote
\begin{equation*}
    \begin{split}
        H=&y_t - b(t) (a(x) y_x)_x - B(t) c(x)\sqrt{a}y_x + D_1 F(0,0)y + C(t)D_2 F(0,0) \beta(x) y_x \\
        &- h 1_{O} + \frac{1}{\mu_1} p^1 1_{O_1} + \frac{1}{\mu_2} p^2 1_{O_2},      
    \end{split}
\end{equation*}
\begin{equation*}
    \begin{split}
        H_{i} =& -p_t^i - b(t) \left(a(x) p_x^i \right)_x + B(t) (c(x)\sqrt{a}p^i)_x + D_1 F(0,0) p^i 
        - C(t)\left(D_2 F(0,0) \beta(x) p^i \right)_x \\
        &- \alpha_i y 1_{O_{i,d}}, \quad i=1,2.      
    \end{split}
\end{equation*}

Let us define the space
\begin{equation}
\label{eq:espaceY}
\begin{array}{c}
\mathcal{Y} = \{  (y,p^1,p^2,h)\in [L^{2}(\Omega\times (0,T))]^{3}\times L^{2}( O\times (0,T)) \ : \ y(\cdot,t), p^{1}(\cdot,t), p^{2}(\cdot,t)\\ \text{are absolutely continuous in}\ [0, 1],\ \text{a.e. in}\ [0, T], \ \rho_{1}h\in L^{2}( O \times (0,T)), \\
% \text{for}\ H=y_t - (l(0) a(x) y_x)_x - h 1_{O} + \frac{1}{\mu_1} p^1 1_{O_1} + \frac{1}{\mu_2} p^2 1_{O_2}\ \text{and}\ \\H_{i}= -p_t^i - \left(l(0) a(x) p_x^i \right)_x  - \alpha_i y 1_{O_{i,d}}, \ i=1,2, 
 \rho_{0}y, \rho_{0}p^{i}, \rho_2 H \in L^2(Q),
 \rho_2H_{i} \in L^2(Q), \\
 y(1, t) \equiv p^{1}(1,t) \equiv  p^{2}(1,t)\equiv y(0,t) \equiv p^{1}(0,t) \equiv p^{2}(0,t)\equiv 0 \ \text{a.e in}\ [0, T],\\
 y(\cdot,0) \in H_a^1(\Omega) \}.
\end{array}
\end{equation}

Thus, $\mathcal{Y}$ is a Hilbert space for the norm $\Vert .\Vert_{\mathcal{Y}}$, where
\begin{equation*}
    \begin{array}{lll}
          \Vert (y,p^1,p^2,h)\Vert^{2}_{\mathcal{Y}} &=& \Vert \rho_{0}y\Vert^{2}_{L^{2}(Q)} + \Vert \rho_{0}p^{i}\Vert^{2}_{L^{2}(Q)} + \Vert\rho_{1}h\Vert^{2}_{L^{2}(O\times (0,T))} + \Vert\rho_{2} H\Vert^{2}_{L^{2}(Q)}\\
          &\quad &+\, \Vert\rho_{2} H_{i}\Vert^{2}_{L^{2}(Q)} + \Vert y(\cdot,0)\Vert^{2}_{H^{1}_{a}(\Omega)}.
    \end{array}
\end{equation*}

Due to Proposition \ref{addicional_estimates_case_linear}, for any $(y,p^1,p^2,h)\in \mathcal{Y}$ we have:
%\textcolor{red}{colocar as desigualdades obtidas $\leq \|(...)\|_{\mathcal{Y}}$}
\begin{equation}\label{eq:estimativas_total}
    \begin{array}{c}
\displaystyle\sup_{[0,T]}(\hat{\rho}^{2}\|y\|^{2}_{L^{2}(0,1)}) + \displaystyle\sup_{[0,T]}(\hat{\rho}^{2}\|p^{1}\|^{2}_{L^{2}(0,1)}) + \displaystyle\sup_{[0,T]}(\hat{\rho}^{2}\|p^{2}\|^{2}_{L^{2}(0,1)})\vspace{0.1cm}
+\displaystyle\int_{Q}\hat{\rho}^{2} a(x)(|y_{x}|^{2} + |p^{1}_{x}|^{2}+|p^{2}_{x}|^{2})dxdt \\
%\ \leq C \Vert(y,p^1,p^2,h)\Vert^{2}_{\mathcal{Y}},  
%    \end{array}
%\end{equation*}
%and 
%\begin{equation*}
%    \begin{array}{c}
+ \displaystyle\sup_{[0,T]}(\rho_{1}^{2}\|\sqrt{a}y_{x} \|^{2}_{L^{2}(0,1)})+ \displaystyle\sup_{[0,T]}(\rho_{1}^{2}\|\sqrt{a}p^{1}_{x} \|^{2}_{L^{2}(0,1)}) + \displaystyle\sup_{[0,T]}(\rho_{1}^{2}\|\sqrt{a}p^{2}_{x} \|^{2}_{L^{2}(0,1)})\\
+ \displaystyle\int_{Q}\rho_{1}^{2}(|y_{t}|^{2}+|p^{1}_{t}|^{2}+|p^{2}_{t}|^{2}+|(a(x)y_{x})_{x}|^{2} + |(a(x)p^{1}_{x})_{x}|^{2}+ |(a(x)p^{2}_{x})_{x}|^{2})dxdt
\leq C \Vert(y,p^1,p^2,h)\Vert^{2}_{\mathcal{Y}}. 
    \end{array}
\end{equation}

Now, let us introduce the Banach space $\mathcal{Z} = \mathcal{F} \times \mathcal{F} \times \mathcal{F} \times H_a^1(\Omega)$ such that 
$$\mathcal{F}=\{ z \in L^2(Q) \ : \ \rho_2 z \in L^2(Q) \}.$$

Finally, consider the map $\mathcal{A} : \mathcal{Y} \to \mathcal{Z}$ such that
$(y,p^1,p^2,h) \mapsto \mathcal{A}(y,p^1,p^2,h) = (\mathcal{A}_0, \mathcal{A}_1, \mathcal{A}_2, \mathcal{A}_3)$ where the components $\mathcal{A}_i$, $i=0,...,3$, are given by
\begin{equation}\label{aplicação A}
\left\{\begin{array}{ll}
\mathcal{A}_0(y,p^1,p^2,h) =& y_t - b(t) \left( a(x) y_x \right)_x - B(t) c(x)\sqrt{a}y_x + F(y,C(t)\beta(x)y_x) \\
&- h 1_{O}  + \frac{1}{\mu_1} p^1 1_{O_1} + \frac{1}{\mu_2} p^2 1_{O_2},\\ %\vspace{0.1cm}\\
\mathcal{A}_i(y,p^1,p^2,h) =& -p_t^i - b(t) \left( a(x) p_x^i \right)_x + B(t) \left( c(x)\sqrt{a}p^i \right)_x + D_1 F(y,C(t)\beta(x)y_x) p^i \\
&-C(t)\left( D_2 F(y,C(t)\beta(x) y_x)\beta(x) p^i \right)_x- \alpha_i (y-y_{i,d}) 1_{O_{i,d}},\,\,\text{for i=1, 2}, \\
\mathcal{A}_3(y,p^1,p^2,h) =& y(.,0).
\end{array}\right.
\end{equation}

Therefore, we will prove that we can apply this Theorem \ref{Liusternik} to the mapping $\mathcal{A}$ in \eqref{aplicação A}, through the following three lemmas:
\begin{lema}\label{A bem definido}
    Let $\mathcal{A}: \mathcal{Y}\rightarrow  \mathcal{Z}$ be given by \eqref{aplicação A}. Then, $\mathcal{A}$ is well defined and continuous. 
\end{lema}

%-------------------------------- BEGIN COMMENT ----------------
%\begin{comment}

\begin{proof}
	We want to show that $\mathcal{A}(y,p^{1},p^{2},h)$ belongs to $ \mathcal{Z}$, for every $(y,p^{1},p^{2},h)\in  \mathcal{Y}$. We will therefore show that each $\mathcal{A}_{i}(y,p^{1},p^{2},h)$, with $i=\lbrace 0,1, 2, 3\rbrace$, defined in \eqref{aplicação A} belongs to its respective space. Notice that, 
	\begin{equation*}
		\begin{array}{lll}
			\|\mathcal{A}_{0}(y,p^{1},p^{2},h)\|^{2}_{\mathcal{F}}&\leq & 2\displaystyle\int_{Q}\rho^{2}_{2}|H|^{2}dxdt \\
            &&+ 2\displaystyle\int_{Q}\rho^{2}_{2}\left|F(y,C(t)\beta(x)y_x) -D_1 F(0,0)y - C(t)D_2 F(0,0)\beta(x)y_x \right|^{2}dxdt \\
			&\quad & = 2 I_{1} + 2 I_{2}.
		\end{array}
	\end{equation*}

It is immediate by definition of the space $\mathcal{Y}$ that $I_{1}\leq C \|(y,p^{1},p^{2},h)\|^{2}_{\mathcal{Y}}$. Furthermore, 
using the mean value theorem, for some $\tilde \theta = \theta(x,t) \in (0,1)$, the boundedness of $C(t)$ and the properties of the weigths \eqref{eq:compara_rhos}, we have that  
\begin{equation*}
\begin{split}
I_2 &\leq \int_Q \rho_2^2 |\nabla F(\tilde\theta y, \tilde\theta C(t)\beta(x)y_x) - \nabla F(0,0)|^2 (|y|^2 + C(t)^2 |\beta(x)y_x|^2) dxdt \\
&\leq \int_Q \rho_2^2 \tilde\theta^2 (|y|^2 + C(t)^2 |\beta(x)y_x|^2) (|y|^2 + C(t)^2 |\beta(x)y_x|^2) \\
& \leq \int_Q \rho_1^4 (|y|^2 + |\beta(x)y_x|^2) (|y|^2 + |\beta(x)y_x|^2) \\
& \leq \int_0^T \sup_{x \in \Omega} \{ \hat \rho^2 |y|^2 + \rho_1^2 |\sqrt{a}y_x|^2 \} \int_\Omega (\hat\rho^2 |y|^2 + \rho_1^2 |\sqrt{a}y_x|^2)dx dt,
\end{split}    
\end{equation*}
where in the last inequality we used that $\beta(x)^2 \leq a(x)$. Therefore, using the continous immersion $H^1_a(\Omega) \subset L^\infty(\Omega)$ and the weight comparison properties \eqref{eq:compara_rhos}, we get 
\begin{equation*}
\begin{split}
I_2 &\leq \int_0^T \left( \| \hat\rho^2 y_x(t)\|^2 + \|\rho_1^2 (a y_x)_x(t) \|^2 \right) \int_\Omega (\rho_0^2 |y|^2 + \hat\rho^2 |\sqrt{a}y_x|^2)dx dt \\
&\leq \sup_{t \in [0,T]} \left( \hat\rho^2 \|y(t)\|^2 + \rho_1^2 \|(a y_x)_x(t) \|^2 \right) \int_Q  (\rho_0^2 |y|^2 + \hat\rho^2 |\sqrt{a}y_x|^2)dx dt \\
&\leq \Vert(y,p^1,p^2,h)\Vert^{2}_{\mathcal{Y}} \Vert(y,p^1,p^2,h)\Vert^{2}_{\mathcal{Y}}.
\end{split}    
\end{equation*}

%\color{red}
%\begin{equation*}
%    \begin{array}{l}
%I_{2}=\displaystyle\int_{Q}\rho^{2}_{2}\left|\left[l\left(\int_{0}^{1}ydx\right) - l(0)\right](a(x)y_{x})_{x}\right|^{2}dxdt \leq C\displaystyle\int_{Q}\rho_{2}^{2}\left(\displaystyle\int_{0}^{1}ydx\right)^{2}|(a(x)y_{x})_{x}|^{2}dxdt
% \leq C\|(y,p^{1},p^{2},h)\|^{4}_{\mathcal{Y}}.
%\end{array}
%\end{equation*}

	Thus, $\mathcal{A}_{0}(y,p^{1},p^{2},h)\in \mathcal{F}.$
	
	\noindent Now let us analyze $\mathcal{A}_{i}(y,p^{1},p^{2},h)$, with $i=\lbrace 1,2\rbrace$. Notice the following:
\begin{equation*}
\begin{array}{lll}
\Vert\mathcal{A}_{i}(y,p^{1},p^{2},h)\Vert^{2}_{\mathcal{F}}
&\leq & 2\displaystyle\int_{Q} \rho^{2}_{2}|H_i|^{2}dxdt \\
&\quad & + \ 2\displaystyle\int_{Q} \rho^{2}_{2}\left| \left( D_1F\left(y,C(t)\beta(x)y_x \right) - D_1 F(0,0) \right) p^1 \right.\\
&&\quad\quad\quad\quad\quad \left.
-C(t)\left( \left( D_2F\left(y,C(t)\beta(x)y_x \right)-D_2 F(0,0) \right) \sqrt{a}p^1 \right)_x \right|^{2}dxdt \\
&\quad & + \
2\displaystyle\int_{Q}\rho^{2}_{2}\left| \alpha_i y_{i,d} 1_{O_{i,d}}\right|^{2}dxdt + \\
& = & 2I_{3} + 2I_{4}+2I_5.
\end{array}
\end{equation*}

By definition of the space $\mathcal{Y}$ we have that $I_{3}\leq C \|(y,p^{1},p^{2},h)\|^{2}_{\mathcal{Y}}$. 

In the following estimates $M$ is a positive constant whose value can change from line to line. Using the mean value theorem and computing the derivative with respect to to $x$,
\begin{equation*}
\begin{split}
I_4 &\leq M \int_Q \rho_2^2 |\left( D_1F\left(y,C(t)\beta(x) y_x \right) - D_1 F(0,0) \right)|^2 |p^i|^2 dxdt \\
&+ M \int_Q \rho_2^2 \left| \left( \left( D_2F\left(y,C(t)\beta(x) y_x \right)-D_2 F(0,0) \right) \beta(x)p^i \right)_x \right|^2 \\
&\leq M\int_Q \rho_2^2 \tilde\theta^2 (|y|^2 + C(t)^2 |\sqrt{a}y_x|^2) |p^i|^2 dxdt \\
%&\iint_Q \rho_2^2 |D_{12}F(y,C(t)\beta(x) y_x)|^2|y_x|^2|y|^2|p^i|^2 + |D_{22}F(y,C(t)\beta(x) y_x)|^2 C(t)^2 |\sqrt{a}y_x|^2 |p^i|^2 + ... \\
& + M \int_Q \rho_2^2 \left( |D_{12}F(y,C(t)\beta(x) y_x)|^2|y_x|^2|\beta(x)|^2|p^i|^2 \right. \\
&\left. \qquad\qquad\quad + |D_{22}F(y,C(t)\beta(x) y_x)|^2 C(t)^2 | (\beta(x)y_x)_x|^2 |\beta(x)|^2 |p^i|^2 \right) \\
& + M \int_Q \rho_2^2 \left| D_2F(y,C(t)\beta(x)y_x)-D_2F(0,0)\right|^2 |(\beta(x)p^i)_x|^2.  %\leq \iint_Q \rho_2^2 \tilde\theta^2 (|y|^2 + C(t)^2 |\beta(x) y_x|^2) |(\beta(x)p^i)_x|^2 } \\
\end{split}
\end{equation*}

Now we use the
hypotheses on $\beta(x)$ in \eqref{eq:cond_beta}
%fact that $\beta(x)^2 \leq a(x)$, $(\beta(x)^2)_x \leq 2 a(x)_x$ and $\beta(x)_x \leq M$ 
to estimate the expressions
\begin{equation}\label{eq:est_beta1}
    |(\beta(x)y_x)_x|^2 |\beta(x)|^2 = |\beta (\beta)_x y_x + \beta^2 y_{xx}|^2 \leq |a_x y_x + a y_{xx}|^2 = |(a(x)y_x)_x|^2,        
\end{equation}
and
\begin{equation}\label{eq:est_beta2}
|(\beta(x)p^i)_x|^2 \leq M|p^i|^2 + |\sqrt{a} p^i_x|^2.   
\end{equation}

Thus, using \eqref{eq:cond_beta}, \eqref{eq:est_beta1}, \eqref{eq:est_beta2}, that $F$ has bounded second derivatives, that $C$ is bounded and the comparison of the weights in \eqref{eq:compara_rhos} we get
\begin{equation*}
\begin{split}
I_4 &\leq M \int_Q \rho_2^2 \tilde\theta^2 (|y|^2 + C(t)^2 |\sqrt{a}y_x|^2) |p^i|^2 dxdt \\
%&\iint_Q \rho_2^2 |D_{12}F(y,C(t)\beta(x) y_x)|^2|y_x|^2|y|^2|p^i|^2 + |D_{22}F(y,C(t)\beta(x) y_x)|^2 C(t)^2 |\sqrt{a}y_x|^2 |p^i|^2 + ... \\
&\quad + M \int_Q \rho_2^2 (|D_{12}F(y,C(t)\beta(x) y_x)|^2|\sqrt{a}y_x|^2 |p^i|^2 + |D_{22}F(y,C(t)\beta(x) y_x)|^2 C(t)^2 |(a(x)y_x)_x|^2 |p^i|^2 \\
&\quad + M \int_Q \rho_2^2 \tilde\theta^2 (|y|^2 + C(t)^2 |\sqrt{a} y_x|^2) (|p^i|^2 + |\sqrt{a} p^i_x|^2)  \\
&\leq M \int_Q \rho_1^2 (|y|^2 + |\sqrt{a}y_x|^2) \rho_1^2 |p^i|^2 dxdt \\
%&\iint_Q \rho_2^2 |D_{12}F(y,C(t)\beta(x) y_x)|^2|y_x|^2|y|^2|p^i|^2 + |D_{22}F(y,C(t)\beta(x) y_x)|^2 C(t)^2 |\sqrt{a}y_x|^2 |p^i|^2 + ... \\
%&\quad 
&\quad + M \int_Q \rho_1^2 |\sqrt{a}y_x|^2 \rho_1^2 |p^i|^2 + \rho_1^2 |(a(x)y_x)_x|^2 \rho_1^2 |p^i|^2 
%&\quad 
+ M \int_Q \rho_1^2 (|y|^2 +  |\sqrt{a} y_x|^2) \rho_1^2 (|p^i|^2 + |\sqrt{a} p^i_x|^2). 
\end{split}
\end{equation*}

Using one more time \eqref{eq:compara_rhos} and the additional estimates in Proposition \ref{addicional_estimates_case_linear}, 
\begin{equation*}
\begin{split}
I_4 & \leq 
%\iint_Q \rho_1^4 (|y|^2 + |\sqrt{a}y_x|^2) (|y|^2 + |\sqrt{a}y_x|^2) \\
%& \leq 
M \int_Q \sup_{t \in [0,T]} \{ \hat \rho^2 |y|^2 + \rho_1^2 |\sqrt{a}y_x|^2 \} \ \hat\rho^2 |p^i|^2 + \left( \sup_{t \in [0,T]} \{ \rho_1^2 |\sqrt{a}y_x|^2 \} + \sup_{t \in [0,T]} \{ \rho_1^2 |(a(x)y_x)_x|^2 \} \right) \hat\rho^2  |p^i|^2)dx dt\\
& + M \int_Q \sup_{t \in [0,T]} \{ \hat \rho^2 |y|^2 + \rho_1^2 |\sqrt{a}y_x|^2 \} \ (\hat \rho^2 |y|^2 + \rho_1^2 |\sqrt{a}y_x|^2) dxdt \\
& \leq M \|(y,p^{1},p^{2},h)\|^{4}_{\mathcal{Y}}.
\end{split}    
\end{equation*}

%We obtain that $I_{4}\leq C \|(y,p^{1},p^{2},h)\|^{4}_{\mathcal{Y}}$. 

%...

%\noindent By definition of the space $\mathcal{Y}$ we have that $I_{3}\leq C \|(y,p^{1},p^{2},h)\|^{2}_{\mathcal{Y}}$. \\

%Since $F$ is of class $C^2$  we obtain $I_{4}\leq C \|(y,p^{1},p^{2},h)\|^{4}_{\mathcal{Y}}$. 
%		\begin{equation*}
%		\begin{array}{l}			I_{4}=\displaystyle\int_{Q}\rho^{2}_{2}\left|F'(y)p^i -c(x,t)p^i \right|^{2}dxdt  \\
%			\quad \leq C\displaystyle\int_{Q}\rho^{2}_{2}\left|F'(y) \right|^{2}\left|p^i \right|^{2}dxdt+\displaystyle\int_{Q}\rho^{2}_{2}\left|-c(x,t)p^i \right|^{2}dxdt\\
%			\quad \leq \ \ \ \ \ C\displaystyle\int_{Q}\rho^{2}_{2} |p^i|^2dxdt+\|c\|_{L^\infty(Q)}\displaystyle\int_{Q}\rho^{2}_{2}\left|p^i \right|^{2}dxdt\\
%			\quad \leq \ \ C\displaystyle\int_{Q}\rho^{2}_{2}\left|p^i \right|^{2}dxdt\\
%			\quad \leq C\|(y,p^{1},p^{2},h)\|^2_{\mathcal{Y}}.
%		\end{array}
%	\end{equation*}
    
Finally, by the hypothesis $\rho^2 y_{i,d} \in L^2(Q)$, we have that $I_5$ is bounded. Thus $\mathcal{A}$ is well-defined.

\end{proof}

%\end{comment}
%--------------------- END COMMENT -----------------

\begin{lema}\label{DA continuo}
    The mapping $\mathcal{A}: \mathcal{Y}\longrightarrow  \mathcal{Z}$ is continuously differentiable.
\end{lema}

%---------------------------- BEGIN COMMENT ----------------
%\begin{comment}

\begin{proof}
	First we prove that $\mathcal{A}$ is Gateaux differentiable at any $(y,p^1,p^2,h) \in \mathcal{Y}$ and let us compute the
	$\textit{G-derivative}$ ${\mathcal{A}}^{\prime}(y, p^{1}, p{^2}, h)$.
	Consider the linear mapping $D \mathcal{A}: \mathcal{Y} \to \mathcal{Z}$ given by
	$$
	D\mathcal{A}(y,p^1,p^2,h) = (D\mathcal{A}_0,D\mathcal{A}_1,D\mathcal{A}_2,D\mathcal{A}_3),
	$$
	where, for $i=1,2$ and  $(\bar y, \bar p^1,\bar p^2,\bar h) \in \mathcal{Y}$,
	\begin{equation}\label{eq:der_map_A}
		\left\{\begin{split}
			D\mathcal{A}_0(\bar y, \bar p^1,\bar p^2,\bar h) = &  \, \bar y_t - b(t)(a(x) \bar y_x)_x -B\sqrt{a} \bar y_x+ D_1 F(y,C\beta(x)y_x)\bar y \\
            &+ C D_2 F(y,C\beta(x)y_x)\beta(x)\bar y_x - \bar h \cara_{\mathcal{O}} + \frac{1}{\mu_1} \bar p^1 \cara_{\mathcal{O}_{1}} + \frac{1}{\mu_2} \bar p^2 \cara_{\mathcal{O}_{2}}, \\
		D\mathcal{A}_i(\bar y, \bar p^1,\bar p^2,\bar h) =& - \bar       p_t^i - (a(x) \bar p_x^i)_x +\left(B\sqrt{a} \bar            y\right)_x 
        %\textcolor{red}{+ F''(y)\bar y p^i + F'(y)\bar p^i} 
        \\
        &+D_{11}^2 F(y,C\beta(x)y_x)\bar y p^i + D_{12}^2(y,C\beta(x)y_x)C\beta(x)\bar y_x p^i \\
        &-(D_{21}^2 F(y,C\beta(x)y_x)\bar y p^i)_x - (D_{22}^2 F(y,C\beta(x)y_x)C \beta(x) \bar y_x p^i)_x \\
        &+ D_1 F(y,C\beta(x)y_x)\bar p^i - \left(D_2 F(y,C\beta(x)y_x) \bar p^i \right)_x - \alpha_i \bar y \cara_{\mathcal{O}_{id}}, \\
		D\mathcal{A}_3(\bar y, \bar p^1,\bar p^2,\bar h) =& \, \bar      y(0).
		\end{split}\right.    
	\end{equation}
	
	We have to show that, for $i=0,1,2,3$, 
	$$
	\frac{1}{\lambda}\left[ \mathcal{A}_i ((y,p^1,p^2,h)+\lambda(\bar y, \bar p^1,\bar p^2,\bar h)) - \mathcal{A}_i (y,p^1,p^2,h) \right] \to D\mathcal{A}_i(\bar y, \bar p^1,\bar p^2,\bar h),
	$$
	strongly in the corresponding factor of $\mathcal{Z}$ as $\lambda \to 0$.
	
	Indeed, we have
	\begin{equation*}
		\begin{split}
			&\left\| \frac{1}{\lambda}\left[ \mathcal{A}_0 ((y,p^1,p^2,h)+\lambda(\bar y, \bar p^1,\bar p^2,\bar h)) - \mathcal{A}_0(y,p^1,p^2,h) \right] - D\mathcal{A}_0(\bar y, \bar p^1,\bar p^2,\bar h) \right\|^{2}_{L^2(\rho_2^2,Q)} \\
			&=\left\|
			 \frac{1}{\lambda}(y_t+\lambda \bar y_t - b(t)\left( a(x) (y_x+\lambda \bar y_x) \right)_x- B\sqrt{a}(y_x+\lambda \bar y_x) + F(y+\lambda \bar y,C\beta(x)(y_x+\lambda \bar y_x))  \right. \\
			 & \ \ \ \ \ \ \ \ \ \  - (h+\lambda \bar h) \cara_{\mathcal{O}} + \frac{1}{\mu_1}  (p^1+\lambda \bar p^1) \cara_{\mathcal{O}_1} + \frac{1}{\mu_2} (p^2+\lambda \bar p^2) \cara_{\mathcal{O}_2}\\
			& \ \ \ \ \ \ \ \ \ \ \left. -y_t + b(t)\left( a(x) y_x \right)_x+B\sqrt{a}y_x - F(y,C\beta(x)y_x) + h 1_{O}  - \frac{1}{\mu_1} p^1 \cara_{\mathcal{O}_1} - \frac{1}{\mu_2} p^2 \cara_{\mathcal{O}_2})\right.\\
			& \ \ \ \ \ \ \ \ \ \ \left. -\bar y_t + b(t)(a(x) \bar y_x)_x +B\sqrt{a} \bar y_x - \nabla F(y,C\beta(x)y_x)(\bar y,C\beta(x)\bar y_x)\right. \\
            %- F'(y)\bar y 
            &\ \ \ \ \ \ \ \ \ \ \left. + \bar h \cara_{\mathcal{O}}- \frac{1}{\mu_1} \bar p^1 \cara_{\mathcal{O}_1} -\frac{1}{\mu_2} \bar p^2 \cara_{\mathcal{O}_2} \right\|^{2}_{L^2(\rho_2^2,Q)} \\
			& = \int_Q \rho_2^2 \left| \frac{1}{\lambda}\left(F(y+\lambda\bar y,C\beta(x)(y_x + \lambda\bar y_x))-F(y,C\beta(x)y_x)\right) - \nabla F(y,C\beta(x)y_x)(\bar y,C\beta(x)\bar y_x) \right|^2 = J_1.
			%& + \int_Q \rho_2^2 \left| l(\int_0^1 (y+\lambda\bar y)) - l(\int_0^1 y) \right|^2 |(a(x)\bar y_x)_x|^2 = J_1 + J_2,
		\end{split}
	\end{equation*}
    
    By the mean value theorem and using that $F$ is of class $C^2$,  
    %there exists $c \in [a,b]$, where $a=\min \{y, y +\lambda \bar y \}$ and $b(t) = \max \{ y , y +\lambda \bar y \}$, such that
    for $\tilde\lambda = \tilde\lambda(x,t) \in (0,\lambda)$, we have 
	\begin{eqnarray*}
		J_1
        %&=&\int_Q \rho_2^2 \left| \frac{1}{\lambda}(F(y+\lambda\bar y)-F(y)) - F'(y)\bar y \right|^2\\
		&=&\int_Q \rho_2^2 \left| \nabla F(y+\tilde\lambda \bar y,C\beta(x)(y_x+\tilde\lambda \bar y_x)) - \nabla F(y,C\beta(x)y_x) \right|^2 (|\bar y|^2 + |C\beta(x)\bar y_x|^2)\\
		&=& C_1 \lambda \int_Q \rho_2^2 \left( |\bar y|^4 + |C\beta(x) \bar y_x|^4 \right)
	\end{eqnarray*}
	We get that $J_1$ converges to zero as $\lambda \to 0$ using \eqref{eq:estimativas_total}.
	
	On the other hand, for $i=1,2$ we get, analogously,
	\begin{equation}
		\begin{split}
			&\left\| \frac{1}{\lambda}\left[ \mathcal{A}_i ((y,p^1,p^2,h)+\lambda(\bar y, \bar p^1,\bar p^2,\bar h)) - \mathcal{A}_i (y,p^1,p^2,h) \right] - D\mathcal{A}_i(\bar y, \bar p^1,\bar p^2,\bar h) \right\|^{2}_{L^2(\rho_2^2,Q)} \\
			%&=\left\|
			%\bar y_t - \frac{1}{\lambda} \left[ \left( a(x) (p^i_x + \lambda\bar p^i_x) - a(x) p^i_x  \right)_x + \left( F'(y+\lambda \bar y) (p^i +\lambda \bar p^i) - F'(y) p^i \right) \right] \right. \\
			%& \left.\left. + l'(\int_0^1 (y +\lambda \bar y))\left(\int_0^1 a(x)(y_x+\lambda\bar y_x)(p^i_x+\lambda \bar p^i_x) \right) - l'(\int_0^1 y)\left(\int_0^1 a(x) y_x p^i_x \right) -\alpha_i \lambda \bar y 1_{O_{i,d}} \right] \right.\\
			%&\left. -\alpha_i \lambda \bar y 1_{O_{i,d}} - D\mathcal{A}_i(\bar y, \bar p^1,\bar p^2,\bar h) \right\|^{2}_{L^2(\rho_2^2,Q)} \\
			%&= \left\| \frac{1}{\lambda} \left( F'(y+\lambda \bar y) (p^i +\lambda \bar p^i) - F'(y) p^i  \right) - F''(y)\bar y p^i - F'(y)\bar p^i \right\|^{2}_{L^2(\rho_2^2,Q)} := J_2
                &= \left\| \frac{F_{\lambda,1}-\bar F_1}{\lambda} p^i -(D_{11}F) \bar y p^i - (D_{12}F) C\beta(x)\bar y_x p^i \right\|^{2}_{L^2(\rho_2^2,Q)} \\
                &+ \left\| \frac{\left( (F_{\lambda,2}-\bar F_2)C \beta(x) p^i \right)_x}{\lambda}- \left((D_{21}F)\bar y p^i + (D_{22}F) C\beta(x)\bar y_x p^i \right)_x \right\|^{2}_{L^2(\rho_2^2,Q)} \\
                &+ \left\| (F_{\lambda,1}-\bar F_1)\bar p^i - (D_1 F) \bar p^i + \left((F_{\lambda,2}-\bar F_2)C \beta(x) \bar p^i - (D_2 F) C \beta(x)\bar p^i \right)_x  \right\|^{2}_{L^2(\rho_2^2,Q)} :=J_2.
		\end{split}
	\end{equation}

    Proceeding as above, by the mean value theorem, using that $F$ is of class $C^2$ and the estimates \eqref{eq:estimativas_total} we get that $J_2$ converges to zero as $\lambda \to 0$.
	
    This finish the proof that $\mathcal{A}$ is Gateaux differentiable, with a \textit{G-derivative} $\mathcal{A^{\prime}}(y,{p^{1}},p^{2},h)= D\mathcal{A}(y,{p^{1}},p^{2},h)$.
	
    Now take $(y,p^{1},p^{2},h)\in \mathcal{Y}$ and let $((y_{n},p^{1}_{n},p^{2}_{n},h_{n}))_{n=0}^{\infty}$ be a sequence which converges to  $(y,p^{1},p^{2},h)$ in $\mathcal{Y}$. For each $(\bar{y},\bar{p}^{1},\bar{p}^{2},\bar{h})\in B_{r}(0)$. 
    %We already proved that 
%	\begin{multline*}
%			D\mathcal{A}_0(y,p^{1},p^{2},h)(\bar y, \bar p^1,\bar p^2,\bar h)=\bar y_t - b(t)(a(x) \bar y_x)_x -B\sqrt{a} \bar y_x+ F'(y)\bar y
%			- \bar h \cara_{\mathcal{O}} + \frac{1}{\mu_1} \bar p^1 \cara_{\mathcal{O}_{1}} + \frac{1}{\mu_2} \bar p^2 \cara_{\mathcal{O}_{2}}
%	\end{multline*}
%	and
%$$
%D\mathcal{A}_i(y,p^{1},p^{2},h)(\bar y, \bar p^1,\bar p^2,\bar h) =- \bar p_t^i - (a(x) \bar p_x^i)_x +\left(B\sqrt{a} \bar y\right)_x+ F''(y)\bar y p^i + F'(y)\bar p^i - \alpha_i \bar y \cara_{\mathcal{O}_{id}}
%$$
%	Then, 
    From the expression of the formal derivative of $\mathcal{A}$,  \eqref{eq:der_map_A}, we have 
	\begin{eqnarray*}
		&&(D\mathcal{A}_0(y_{n},p^{1}_{n},p^{2}_{n},h_{n})-D\mathcal{A}_0(y,p^{1},p^{2},h))(\bar y, \bar p^1,\bar p^2,\bar h) \\
        &&= \left(D_1 F(y_n,C,\beta(x)y_{n,x})-D_1 F(y,C\beta(x)y_x)\right)\bar y + \left(D_2 F(y_n,C,\beta(x)y_{n,x})-D_2 F(y,C\beta(x)y_x)\right)C\beta(x)\bar y\\
        &&= X_1^1+X_2^1.
	\end{eqnarray*}
	For $X^{1}_{1}$ we have, 
	\begin{eqnarray*}
		\int_{0}^{T}\int_{0}^{1}\rho^{2}_{2}|X_{1}^{1}|^{2} dxdt&\leq& C \int_{0}^{T}\int_{0}^{1}\rho_{2}^{2} |D_1F(y_n,C\beta(x)y_{n,x})-D_1F(y,C\beta(x)y_x)|^2 |\bar y|^{2}dxdt\\
		&\leq&C\int_{0}^{T}\int_{0}^{1}\rho_{2}^{2} M^2 |y_n-y|^2 |\bar y|^{2}dxdt
	\end{eqnarray*}
	By \eqref{eq:compara_rhos} we have:  $\rho_2^2\leq C \rho_1^2 \rho^2_1\leq C \hat{\rho}^2 \rho_0^2$, 
		\begin{eqnarray*}
		\int_{0}^{T}\int_{0}^{1}\rho^{2}_{2}|X_{1}^{1}|^{2} dxdt&\leq& C\int_{0}^{T}\int_{0}^{1}\hat{\rho}^2 \rho_0^2  |y_n-y|^2 |\bar y|^{2}dxdt\\
		&\leq& C\sup_{[0,T]}\left(\hat{\rho}^2\| \bar{y}\|^2_{L^2(\Omega)}\right)\int_{0}^{T}\int_{0}^{1} \rho_0^2  |y_n-y|^2 dxdt\\
			&\leq&C\|(\bar y,\bar p^{1},\bar p^{2},\bar h\|_{\mathcal{Y}} \cdot \|\rho_0 (y_n-y)\|_{L^2(Q)}\\
		&\leq&C\|(y_{n}-y),(p^{1}_{n}-p^{1}),(p^{2}_{n}-p^{2}),(h_{n}-h)\|_{\mathcal{Y}}\rightarrow 0.
	\end{eqnarray*}
	
	An analogous estimate holds also for $X_2^1$.
	
	Now consider the term $D\mathcal{A}_{i}$, then 
	\begin{eqnarray*}
		&&(D\mathcal{A}_i(y_{n},p^{1}_{n},p^{2}_{n},h_{n})- D\mathcal{A}_i(y,p^{1},p^{2},h))(\bar y, \bar p^1,\bar p^2,\bar h) \\ %&=&\textcolor{red}{F''(y_n)\bar y p_n^i + F'(y_n)\bar p^i - F''(y)\bar y p^i - F'(y)\bar p^i}\\
        &&=\left( D_1 F(y_n,C\beta(x)y_{n,x})\bar p^i - D_1 F(y,C\beta(x)y_{x})\bar p^i \right)\\
        &&\qquad - \left( \left(D_2 F(y_n,C\beta(x)y_{n,x}) \bar p^i \right)_x - \left(D_2 F(y,C\beta(x)y_x) \bar p^i \right)_x \right) \\
        &&\qquad + (D_{11}^2 F(y_n,C\beta(x)y_{n,x})\bar y p_n^i - D_{11}^2 F(y,C\beta(x)y_x)\bar y p^i) \\
        &&\qquad + (D_{12}^2(y_n,C\beta(x)y_{n,x})C\beta(x)\bar y_x p_n^i - D_{12}^2(y,C\beta(x)y_x)C\beta(x)\bar y_x p^i ) \\
        &&\qquad -(D_{21}^2 F(y_n,C\beta(x)y_{n,x})\bar y p_n^i)_x
 + (D_{21}^2 F(y,C\beta(x)y_{x})\bar y p^i)_x \\        
        &&\qquad - (D_{22}^2 F(y_n,C\beta(x)y_{n,x})C \beta(x) \bar y_x p_n^i)_x +  (D_{22}^2 F(y,C\beta(x)y_x)C \beta(x) \bar y_x p^i)_x \\
        &&= X^2_1 + X^2_2+X^2_3+X^2_4+X^2_5+X^2_6.
	\end{eqnarray*}
	
For example, for $X^{2}_{1}$, 
\begin{eqnarray*}
	\int_{0}^{T}\int_{0}^{1}\rho^{2}_{2}|X_{1}^{2}|^{2} dxdt&=&
    \int_{0}^{T}\int_{0}^{1}\rho_{2}^{2} \left| D_1 F(y_n,C\beta(x)y_{n,x})\bar p^i - D_1 F(y,C\beta(x)y_{x})\bar p^i \right|^2 \\
    &\leq& \int_{0}^{T}\int_{0}^{1}\rho_{2}^{2} \left| D_1 F(y_n,C\beta(x)y_{n,x})\bar p^i - D_1 F(y,C\beta(x)y_{n,x})\bar p^i \right|^2 \\
    && + \int_{0}^{T}\int_{0}^{1}\rho_{2}^{2} \left| D_1 F(y,C\beta(x)y_{n,x})\bar p^i - D_1 F(y,C\beta(x)y_{x})\bar p^i \right|^2 \\
%    &\leq& \int_{0}^{T}\int_{0}^{1}\rho^{2}_{2}|X_{1}^{2}|^{2} dxdt&=&
%    \int_{0}^{T}\int_{0}^{1}\rho_{2}^{2} (D_{11}^2 F(y_n,C\beta(x)y_{n,x})\bar y p_n^i - D_{11}^2 F(y,C\beta(x)y_x)\bar y p^i)\\
    &\leq& C \int_{0}^{T}\int_{0}^{1}\rho_{2}^{2} \left( |y_n - y|^2 |\bar p^i|^2 + C^2\beta(x)^2|y_{n,x}-y_x|^2 |\bar p^i|^2 \right) %\\
%    &\leq&
%    C \int_{0}^{T}\int_{0}^{1}\rho_{2}^{2} |F''(y_n)\bar y p_n^i + F'(y_n)\bar p^i- F''(y)\bar y p^i - F'(y)\bar p^i|^2 dxdt\\
%	&\leq&C\int_{0}^{T}\int_{0}^{1}\rho_{2}^{2} |F''(y_n)\bar y p_n^i - F''(y)\bar y p^i |^2 +\int_{0}^{T}\int_{0}^{1}\rho_{2}^{2} | F'(y_n)\bar p^i - F'(y)\bar p^i|^2 \\
%	&\leq &C\int_{0}^{T}\int_{0}^{1}\rho_{2}^{2} |F''(y_n) (p_n^i-p^i)+(F''(y_n) - F''(y)) p^i |^2 |\bar y|^2\\
%	&& +\int_{0}^{T}\int_{0}^{1}\rho_{2}^{2} | F'(y_n) - F'(y)|^2 |\bar p^i|^2\\
%	&\leq &C\int_{0}^{T}\int_{0}^{1}\rho_{2}^{2} |F''(y_n) (p_n^i-p^i) |^2 |\bar y|^2+\int_{0}^{T}\int_{0}^{1}\rho_{2}^{2} |(F''(y_n) - F''(y)) p^i |^2 |\bar y|^2\\
%	&& +\int_{0}^{T}\int_{0}^{1}\rho_{2}^{2} | F'(y_n) - F'(y)|^2 |\bar p^i|^2\\
%		&\leq &C\int_{0}^{T}\int_{0}^{1}\rho_{2}^{2} |p_n^i-p^i |^2 |\bar y|^2+\int_{0}^{T}\int_{0}^{1}\rho_{2}^{2} |y_n - y|^2 | p^i |^2 |\bar y|^2+\int_{0}^{T}\int_{0}^{1}\rho_{2}^{2} |y_n - y|^2 |\bar p^i|^2.
\end{eqnarray*}
By \eqref{eq:compara_rhos} we have:  $\rho_2^2\leq C \rho_1^2 \rho^2_1\leq C \hat{\rho}^2 \rho_0^2\leq C \hat{\rho}^4 \rho_0^2 $
\begin{eqnarray*}
	\int_{0}^{T}\int_{0}^{1}\rho^{2}_{2}|X_{1}^{2}|^{2} dxdt
    %&\leq& C\int_{0}^{T}\int_{0}^{1}\rho_{2}^{2} |p_n^i-p^i |^2 |\bar y|^2+\int_{0}^{T}\int_{0}^{1}\rho_{2}^{2} |y_n - y|^2 | p^i |^2 |\bar y|^2 \\
    %&&+ \int_{0}^{T}\int_{0}^{1}\rho_{2}^{2} |y_n - y|^2 |\bar p^i|^2\\
%	&\leq &C\int_{0}^{T}\int_{0}^{1}\hat{\rho}^2|\bar y|^2  \rho_0^2|p_n^i-p^i |^2 +\int_{0}^{T}\int_{0}^{1}\hat{\rho}^2| p^i |^2 \hat{\rho}^2|\bar y|^2 \rho_0^2 |y_n - y|^2 \\
%	&& +\int_{0}^{T}\int_{0}^{1}\hat{\rho}^2|\bar p^i|^2  \rho_0^2 |y_n - y|^2 \\
	&\leq&  C\sup_{[0,T]}\left(\hat{\rho}^2\| \bar{p}^i\|^2_{L^2(\Omega)}\right)\int_{0}^{T}\int_{0}^{1} \rho_0^2  \left( |y_n-y|^2 +  |\beta(x)(y_{n,x}-y_x)|^2  \right)\\
    &\leq&  C\sup_{[0,T]}\left(\hat{\rho}^2\| \bar{p}^i\|^2_{L^2(\Omega)}\right)\int_{0}^{T}\int_{0}^{1} \rho_0^2  \left( |y_n-y|^2 + |\sqrt{a}(y_{n,x}-y_x)|^2  \right)\\    
%	&&+C\sup_{[0,T]}\left(\hat{\rho}^2\| \bar{p}^i\|^2_{L^2(\Omega)}\right)\sup_{[0,T]}\left(\hat{\rho}^2\| \bar{y}\|^2_{L^2(\Omega)}\right)\int_{0}^{T}\int_{0}^{1} \rho_0^2  |y_n-y|^2 \\
	&\leq & C \|(\bar y,\bar p^{1},\bar p^{2},\bar h\|_{\mathcal{Y}} \|(y_{n}-y),(p^{1}_{n}-p^{1}),(p^{2}_{n}-p^{2}),(h_{n}-h)\|_{\mathcal{Y}}\rightarrow 0,  
%    \cdot \|\rho_0 (p^i_n-p^i)\|_{L^2(Q)}+\|(\bar y,\bar p^{1},\bar p^{2},\bar h\|_{\mathcal{Y}} \cdot \|\rho_0 (y_n-y)\|_{L^2(Q)}\\
%	&&+\|(\bar y,\bar p^{1},\bar p^{2},\bar h\|_{\mathcal{Y}}\cdot \|( y, p^{1}, p^{2}, h\|_{\mathcal{Y}}\cdot \|\rho_0 (y_n-y)\|_{L^2(Q)}\\
%		&\leq&C\|(y_{n}-y),(p^{1}_{n}-p^{1}),(p^{2}_{n}-p^{2}),(h_{n}-h)\|_{\mathcal{Y}}\rightarrow 0.
\end{eqnarray*}
as $n \to \infty$. The others terms $X^2_k$, $k=2,...,6$ are estimated similarly.

	Therefore, $(y,p^{1},p^{2},h)\longmapsto\mathcal{A}^{\prime}(y,p^{1},p^{2},h)$ is continuous from $\mathcal{Y}$ into $\mathcal{L}(\mathcal{Y},\mathcal{Z})$ and as consequently, in view of classical results, we will
	have that $\mathcal{A}$ is Fr\'echet-differentiable and $\mathcal{C}^{1}$.
\end{proof}

%\end{comment}
%----------------------- END COMMENT ------------------

\begin{lema}\label{Mapa sobrejetivo}
Let $\mathcal{A}$ be the mapping in \eqref{aplicação A}. Then, $\mathcal{A}^{\prime}(0,0,0,0)$ is onto.
\end{lema}

\begin{proof}
Let $(H, H_{1}, H_{2}, y_{0})\in \mathcal{Z}$. From Theorem \ref{theorem case linear} we know there exists $y, p^{1}, p^{2}$ satisfying \eqref{eq:linearized_system} and \eqref{estimate for solution}. Furthermore, we know that $y, p^{1}, p^{2}\in C^{0}([0,T];L^{2}(0,1))\cap L^{2}(0,T;H^{1}_{a})$. Consequently, $(y,p^{1},p^{2},h)\in \mathcal{Y}$ and $$\mathcal{A}^{\prime}(0,0,0,0)(y,p^{1},p^{2},h)=(H,H_{1}, H_{2},y_{0}).$$ 
\hfill

This ends the proof.
\end{proof}

\noindent\textbf{Proof of Theorem \ref{thm:local_null_controllability}.}

\noindent According to Lemmas \ref{A bem definido}-\ref{Mapa sobrejetivo} we can apply the Inverse Mapping Theorem (Theorem \ref{Liusternik}) and consequently there exists $\delta > 0$ and a mapping $W:B_{\delta}(0)\subset {\mathcal{Z}}\rightarrow {\mathcal{Y}}$ such that
\begin{equation*}
    W(z)\in B_{r}(0)\,\,\, \text{and}\,\,\, {\mathcal{A}}(W(z))=z, \,\,\, \forall z\in B_{\delta}(0).
\end{equation*}
Taking $(0,0,0,y_{0})\in B_{\delta}(0)$ and $(y,p^{1},p^{2},h)=W(0,0,0,y_{0})\in {\mathcal{Y}}$, we have
\begin{equation*}
    {\mathcal{A}}(y,p^{1},p^{2},h)=(0,0,0,y_{0}).
\end{equation*}
Thus, we conclude that \eqref{eq:optimality_system} is locally null controllable at time $T > 0$.

\section{Final remarks and some problems}
\label{sec:final_remarks}

Here we present some problems in the context addressed in this paper that, as far as we know, are open.
\begin{itemize}
%\item An interesting problem is a semilinear degenerate parabolic equation with gradient term in a non-cylindrical domain
%$$	\begin{cases}
%    u_t-({a}(x')u_{x'})_{x'} + F(u,u_{x'})=\widehat{h}\cara_{{\widehat{\mathcal{O}}}} + \widehat{v}^1\cara_{{\widehat{\mathcal{O}}_1}}+\widehat{v}^2\cara_{{\widehat{\mathcal{O}}_2}}, & \ \ \ \text{in} \ \ \ \widehat{Q}\\
%    u(0,t)=u(\ell(t),t)=0, & \ \ \ \text{in} \ \ \ \widehat{\Sigma}\\
%    u(0)=u_0(x'), & \ \ \ \text{in} \ \ \ \Omega_0
%\end{cases}$$
\item An interesting problem is the local null controllability of a quasilinear degenerate parabolic equation in a non-cylindrical domain, i.e., using the notations of our paper, 
		$$	\begin{cases}
		u_t-(b(u){a}(x')u_{x'})_{x'}+F(u,\beta(x) u_{x'})=\widehat{h}\cara_{{\widehat{\mathcal{O}}}}+\widehat{v}^1\cara_{{\widehat{\mathcal{O}}_1}}+\widehat{v}^2\cara_{{\widehat{\mathcal{O}}_2}}, & \ \ \ \text{in} \ \ \ \widehat{Q},\\
		u(0,t)=u(\ell(t),t)=0, & \ \ \ \text{on} \ \ \ \widehat{\Sigma},\\
		u(0)=u_0(x'), & \ \ \ \text{in} \ \ \ \Omega_0.
	\end{cases}$$
	Here $b = b(r)$ is a real function of class $C^3$ such that, for any $r\in \mathbb{R}$,
	$$0< b_0\leq b(r)\leq b_1, \ \ \ \ \ \ \ \ |b'(r)|+|b''(r)|+|b'''(r)|\leq M.$$
    A contribution to the case of quasilinear equations and hierarchical control was made by Huaman \cite{frances24} and Nu\~nez-Chavez and Limaco \cite{miguel}, who studied the problem of exact hierarchical controllability along trajectories in cylindrical domains. We intend to extend this work to moving or non-cylindrical domains.
    
    \item Another problem is the controllability of the following degenerate,
non-local, semi-linear parabolic equation in a non-cylindrical domain:
		$$	\begin{cases}
	u_t-\left({a}(x')l\left(\int_{\Omega_t}u dx'\right)u_{x'}\right)_{x'}+F(u,\beta(x)u_{x'})=\widehat{h}\cara_{{\widehat{\mathcal{O}}}}+\widehat{v}^1\cara_{{\widehat{\mathcal{O}}_1}}+\widehat{v}^2\cara_{{\widehat{\mathcal{O}}_2}}, & \ \ \ \text{in} \ \ \ \widehat{Q},\\
	u(0,t)=u(\ell(t),t)=0, & \ \ \ \text{on} \ \ \ \widehat{\Sigma},\\
	u(0)=u_0(x'), & \ \ \ \text{in} \ \ \ \Omega_0.
\end{cases}$$
The first contribution to this kind of problem was provided by Limaco et al. \cite{Joao_Juan_Suerlan-25}, who established controllability results for non-local degenerate equations in cylindrical domains. However, the situation becomes more intricate when the spatial domain evolves with time, as additional difficulties arise from the moving boundary. A natural continuation to \cite{Joao_Juan_Suerlan-25} is to investigate the extent to which controllability can be ensured in such non-cylindrical settings. Furthermore, it is of interest to study more general nonlinear models by incorporating gradient-dependent terms. This extension would not only broaden the theoretical framework of controllability for degenerate equations but also enhance its applicability to physical and biological systems where non-local effects, nonlinear interactions, and evolving domains play a crucial role.

\item Another related problem is a semilinear coupled degenerate parabolic system in a non-cylindrical domain:
$$\begin{cases}
	u_{1t}-({a}(x')u_{1x'})_{x'}+F_1(u_1,u_2)=\widehat{h}\cara_{_{\widehat{\mathcal{O}}}}+\widehat{v}^1\cara_{_{\widehat{\mathcal{O}}_1}}+\widehat{v}^2\cara_{_{\widehat{\mathcal{O}}_2}}, & \ \ \ \text{in} \ \ \ \widehat{Q}, \\
	u_{2t}-({a}(x')u_{2x'})_{x'}+F_2(u_1,u_2)=0, & \ \ \ \text{in} \ \ \ \widehat{Q},\\
	u_1(0,t)=u_1(\ell(t),t)=u_2(0,t)=u_2(\ell(t),t)=0, & \ \ \ \text{on} \ \ \ \widehat{\Sigma},\\
	u_1(0)=u_1^0(x'), \ u_2(0)=u_2^0(x') & \ \ \ \text{in} \ \ \ \Omega_0.
\end{cases}$$
The first work in this direction was carried out by Djomegne et al. \cite{diomedes}, who studied semilinear and degenerate coupled parabolic systems in fixed domains. The complexity of the semilinear terms is given by nonlinear functions $F(u_1)$ and $F_2(u_2)$. A natural extension would be to generalize these results to moving domains. This is an ongoing work of the authors \cite{GYL-sistema-2025}.

\item An interesting problem is a  fourth order  degenerate parabolic equation in a non-cylindrical domain:
$$	\begin{cases}
	u_t-\left({a}(x')u_{x'x'}\right)_{x'x'}=\widehat{h}\cara_{{\widehat{\mathcal{O}}}}+\widehat{v}^1\cara_{{\widehat{\mathcal{O}}_1}}+\widehat{v}^2\cara_{{\widehat{\mathcal{O}}_2}}, & \ \ \ \text{in} \ \ \ \widehat{Q},\\
	u(0,t)=u(\ell(t),t)=u_{x'x'}(0,t)=u_{x'x'}(\ell(t),t)=0, & \ \ \ \text{on} \ \ \ \widehat{\Sigma},\\
	u(0)=u_0(x'), & \ \ \ \text{in} \ \ \ \Omega_0.
\end{cases}$$
While classical results exist for second-order degenerate equations in fixed or moving domains , the extension to higher-order equations presents new challenges. In particular, the presence of fourth-order spatial derivatives requires the formulation of appropriate boundary conditions and the development of adapted Carleman estimates \cite{cuarto}. Furthermore, moving boundaries introduce additional technical difficulties complicating both the analysis of the system and the design of controls. Addressing this problem would significantly advance the theory of controllability for higher-order degenerate parabolic equations and could have applications in elasticity, thin film flows, and other physical systems described by fourth-order dynamics.''

%\newpage
\noindent {\bf Acknowledgments}

\noindent This study was financed in part by the Coordenação de Aperfeiçoamento de Pessoal de Nível Superior - Brasil (CAPES) - Finance Code 001.
%The authors thank Prof. Enrique Fernández-Cara for useful conversations on the Stefan problem.
J. L. was partially supported by CNPq-Brazil.  A.G. was partially supported by Capes-Brazil.

\appendix
\section{Well-Posedness of system (15) } \label{appendix A}

Consider the optimality system (\ref{eq:optimality_system}) and using the change of variable $\varphi^{i}(t)=p^{i}(T-t)$, for $i=1,2.$, we obtain 
\begin{equation*}
	\begin{cases}
		y_t-\frac{1}{\ell(t)^2}\left(a(x)y_x\right)_x-\frac{\ell'(t)}{\ell(t)}xy_x+F(y,C(t)\beta(x) y_x) = {h}\cara_{_{{\mathcal{O}}}}-\frac{1}{\mu_1}\varphi^1\cara_{_{\mathcal{O}_1}}-\frac{1}{\mu_2}\varphi^2\cara_{_{\mathcal{O}_2}}, & \ \ \ \text{in} \ \ \ {Q},\\
		\varphi^i_t-\frac{1}{\ell(t)^2}\left(a(x)\varphi^i_x\right)_x+\frac{\ell'(t)}{\ell(t)}\left(x\varphi^i\right)_x + D_1 F\left(y,C(t) \beta(x) y_x \right)\varphi^i \\
        \qquad - C(t) \left(D_2 F\left(y,C(t) \beta(x) y_x  \right)\beta(x) \varphi^i \right)_x  %F'(y)\varphi^i
        =\alpha_i(y-y_{id})\cara_{_{{\mathcal{O}}_{id}}}, & \ \ \ \text{in} \ \ \ {Q},\\
		y=0, \ \ \ \ \ \varphi^1=0, \ \ \ \ \ \varphi^2=0& \ \ \ \text{on} \ \ \ \Sigma,\\
		y(0)=y_0, \ \ \ \ \ \varphi^1(0)=\varphi_0^1, \ \ \ \ \ \ \varphi^2(0)=\varphi_0^2& \ \ \ \text{in} \ \ \ \Omega.
	\end{cases}
\end{equation*}

In fact, we prove the well-posedness of the more general system:
\begin{equation*}
	\left\{\begin{aligned}
		&y_t - b(t) \left(a(x) y_x\right)_x + c_1(x,t) \sqrt{a} y_x + F(y,c(t) \sqrt{a}y_x) = h 1_{O} - \sum_{i=1}^{2}\frac{1}{\mu_i} \varphi^i 1_{O_i} &&\text{in}&& Q, \\
		&\varphi_t^i - b(t) \left( a(x) \varphi_x^i \right)_x + c_2(x,y,c(t)\sqrt{a} y_x,t) \sqrt{a} \varphi^i_x + g(x,y,c(t) \sqrt{a} y_x,t) \varphi^i  \\
        &\qquad = \alpha_i (y-y_{i,d}) 1_{O_{i,d}}  &&\text{in}&& Q,\\
		&y(0,t)=y(1,t)=0, \ \varphi^i(0,t) = \varphi^i(1,t) = 0 &&\text{on}&& (0,T), \\
		& \varphi^i(\cdot,0) = \varphi^{i}_{0} &&\text{in }&& \Omega, \\
		&y(\cdot,0) = y_0 &&\text{in}&& \Omega,
	\end{aligned}
	\right.
\end{equation*}
where $c_1$, $c_2$, $g$ are bounded functions on $Q$, $0 < b_0 \leq b(t)$, $c(t)$ bounded   
%such that $|\frac{b_t}{b}|<C$, 
and $F$ is globally Lipschitz in both variables.
%has bounded derivatives up to order 2. 
In our case, $b(t) = \frac{1}{\ell(t)^2}$, $c_1 = -\frac{\ell(t)'}{\ell(t)} \frac{x}{\sqrt{a}}$, $c_2 = \frac{\ell(t)'}{\ell(t)} \frac{x}{\sqrt{a}} - c(t) D_2 F\left(y,c(t)\sqrt{a} y_x \right)$ and $g = \frac{\ell'(t)}{\ell(t)} + D_1 F \left(y,c(t) \sqrt{a} y_x \right) - c(t) \sqrt{a} \left( D_2 F\left(y, c(t) \sqrt{a} y_x \right) \right)_x$ and $\frac{x}{\sqrt{a}}$ is bounded.

Let $(w_{i})_{i}^{\infty}$ be an orthonormal basis of $H^{1}_{a}(0,1)$ such that 
%\begin{equation*}
	$-b(t)(a(x)w_{i,x})_{x}=\lambda_{i}w_{i}$.    
%\end{equation*}
Fix $m\in\mathbb{N}^{*}$. Due the Caratheodory's theorem, there exist absolutely continuous functions $g_{im}=g_{im}(t)$ and $h_{im}=h_{im}(t)$ with $i\in\{1,2,...,m\}$ such that, for $t \in [0,T]$,
\begin{eqnarray*}
	t\mapsto y_{m}(t)=\sum_{i=1}^{m}g_{im}(t)w_{i} \in H_{a}^{1}(0,1) \qquad\text{and}\qquad t\mapsto \varphi_{m}(t)=\sum_{i=1}^{m}h_{im}(t)w_{i} \in H_{a}^{1}(0,1), 
\end{eqnarray*}
%and 
%\begin{eqnarray*}
%	t\in[0,T]\mapsto \varphi_{m}(t)=\sum_{i=1}^{m}h_{im}(t)w_{i} \in H_{a}^{1}(0,1)
%\end{eqnarray*}
satisfy
\begin{equation}
	\label{eq:galerkin_system}
	\left\{\begin{aligned}
		&(y_{m,t},w) - b(t)((a(x) y_{m,x})_x,w) + (c_1 \sqrt{a}  y_{m,x},w) + (F(y_m, c(t) \sqrt{a} y_{m,x}),w) \\
        &\qquad = (h1_{O},w) - \sum_{i=1}^{2}\frac{1}{\mu_i} (\varphi^{i}_{m} 1_{O_i},w) &&\text{in}&& Q, \\
		&(\varphi_{m,t}^i,\hat{w}) - b(t)((a(x) \varphi_{m,x}^i)_x,\hat{w}) + (c_2 \sqrt{a}  \varphi^i_{m,x},\hat w) + (g \varphi^i_m, \hat w)  = (\alpha_i (y_m -y_{i,d})1_{O_{i,d}},\hat{w})   &&\text{in}&& Q,\\
		&y_{m}(0,t)=y_{m}(1,t)=0, \ \varphi_{m}^i(0,t) = \varphi_{m}^i(1,t) = 0 &&\text{on}&& (0,T), \\
		& \varphi_{m}^i(\cdot,0)\to \varphi^{i}_{0} &&\text{in }&& \Omega, \\
		&y_{m}(\cdot,0)\to y_{0} &&\text{in}&& \Omega.
	\end{aligned}
	\right.
\end{equation}

For any $w,\hat{w}\in [w_{1},w_{2},...,w_{m}]$  and $(\cdot,\cdot)=(\cdot,\cdot)_{L^{2}}$. Taking $w=y_{m}$ and $\hat{w}=\varphi^{i}_{m}$, then 
\begin{equation*}
	\begin{split}
		&\frac{1}{2}\frac{d}{dt}\left(\|y_{m}\|^{2}+\sum_{i=1}^{2}\|\varphi_{m}^{i}\|^{2}\right) + b(t)\left(\|\sqrt{a}y_{m,x}\|^{2}+\sum_{i=1}^{2}\|\sqrt{a}\varphi^{i}_{m,x}\|\right) \\
		&+ (c_1 \sqrt{a} y_{m,x},y_m) + \sum_{i=1}^{2}(c_2 \sqrt{a}\varphi^i_{m,x},\varphi^i_m) +(F(y_m,c(t) \sqrt{a} y_{m,x}),y_m) + \sum_{i=1}^{2}(g \varphi^i_m,\varphi^i_m) \\
		&=(h,y_{m})- \sum_{i=1}^{2}\frac{1}{\mu_{i}}(y_{m},\varphi^{i}_{m})+\sum_{i=1}^{2}(y_{m},\varphi^{i}_{m})-\sum_{i=1}^{2}(y_{i,d},\varphi^{i}_{m}).
	\end{split}
\end{equation*}

Using that $F$ is Lipschitz in both variables, that $b(t)$ is bounded, that $c(t)$ is bounded and that $c_i$ and $g$ are bounded functions on $Q$, we have, for some constant $C_* > 0$, 
%Integrating \eqref{systemwellposedness} from  to $0$ from $t$ and using Gronwall's inequality we deduce that, 
\begin{equation}
	\label{energyestimatewellpo1}
	\begin{split}
		&\|y_{m}(t)\|^{2}+\sum_{i=1}^{2}\|\varphi_{m}^{i}(t)\|^{2} + \int_{0}^{t}\|\sqrt{a}y_{m,x}\|^{2}ds + \int_{0}^{t}\sum_{i=1}^{2}\|\sqrt{a}\varphi_{m,x}^{i}\|^{2}ds\\ 
		&\leq e^{C_{*} T} \left(\|h\|^{2}_{L^{2}((0,T),L^2(O))}+\sum_{i=1}^{2}\|y_{i,d}\|^{2}_{L^{2}((0,T),L^2(O_{i}))} + \|y_{0}\|_{H^{1}_{a}(0,1)} +\sum_{i=1}^{2}\|\varphi_{0}^{i}\|^{2}_{L^{2}(0,1)}\right)
	\end{split}
\end{equation}
Since \eqref{energyestimatewellpo1} holds for any $t\in[0,T]$, 
\begin{equation}\label{energyestimate1}
	\begin{split}
		&\|y_m\|_{L^{\infty}(0,T,L^{2}(0,1))}^{2}+\sum_{i=1}^{2}\|\varphi_{m}^{i}\|_{L^{\infty}(0,T,L^{2}(0,1))}^{2}+\|\sqrt{a}y_{m,x}\|_{L^{2}(0,T,L^{2}(0,1))}^{2}+\sum_{i=1}^{2}\|\sqrt{a}\varphi_{m,x}^{i}\|_{L^{2}(0,T,L^{2}(0,1))}^{2}\\ 
		&\leq e^{C_{*} T} \left(\|h\|^{2}_{L^{2}((0,T),L^2(O))}+\sum_{i=1}^{2}\|y_{i,d}\|^{2}_{L^{2}((0,T),L^2(O_{i}))}+\|y_{0}\|_{H^{1}_{a}(0,1)}+\sum_{i=1}^{2}\|\varphi_{0}^{i}\|^{2}_{L^{2}(0,1)}\right) =: \mathcal{K}_1.
	\end{split}
\end{equation}

Estimate II: Taking $w=y_{m,t}$ and $\hat{w}=\varphi^{i}_{m,t}$ in \eqref{eq:galerkin_system}  we get
\begin{equation*}
	\begin{split}
		&\|y_{m,t}\|^2 + \sum_{i=1}^2 \|\varphi^i_{m,t}\|^2 + b(t) \frac{1}{2}\frac{d}{dt} \left( \|\sqrt{a}y_{m,x}\|^2 + \sum_{i=1}^2 \|\sqrt{a}y_{m,x}\|^2 \right) \\
		& + (c_1 \sqrt{a} y_{m,x},y_{m,t}) + \sum_{i=1}^2 (c_2 \sqrt{a}\varphi^i_{m,x},\varphi^i_{m,t}) +(F(y_m,c(t)\sqrt{a}y_{m,x}),y_{m,t}) + \sum_{i=1}^2 (g \varphi^i_m,\varphi^i_{m,t}) \\
		&\leq \|h\|^2 + \frac{1}{4}\|y_{m,t}\|^2 + C\sum_{i=1}^2 \|\varphi^i_m\|^2 + \frac{1}{4}\|y_{m,t}\|^2 + 2 C \|y_m\|^2 + \frac{1}{4}\sum_{i=1}^2  \|\varphi^i_{m,t}\|^2 \\
        &+ C \sum_{i=1}^2  \|y_{i,d}\|^2 + \frac{1}{4}\sum_{i=1}^2  \|\varphi^i_{m,t}\|^2.
	\end{split}
\end{equation*}

Thus, the lower-boundedness if $b(t)$, Cauchy-Schwarz inequality, the fact that $F$ is Lipschitz in both variables, that $c(t)$ is bounded, that $c_i$ and $g$ are bounded functions on $Q$ and Young's inequality,  imply that there exists a constants $D>0$ such that
%and $\epsilon>0$ such that 
%Integrating in $t$ on $[0,t]$  and applying Gronwall's inequality  we get
\begin{equation*}
	\begin{split}
		&\|y_{m,t}\|^2_{L^2(0,T,L^2(0,1))} + \sum_{i=1}^2  \|\varphi^i_{m,t}\|^2_{L^2(0,T,L^2(0,1))} + \left( \|\sqrt{a}y_{m,x}(t)\|^2 + \sum_{i=1}^2  \|\sqrt{a}\varphi^i_{m,x}(t)\|^2 \right) \\
		&\leq e^{DT} \left[ \|\sqrt{a}y_{m,x}(0)\|^2 + \sum_{i=1}^2  \|\sqrt{a}\varphi^i_{m,x}(0)\|^2 +\|h\|^2_{L^2(0,T,L^2(O))} + \sum_{i=1}^2  \|y_{i,d}\|^2_{L^2(0,T,L^2(O_i))} \right. \\
        &\left.+ \|y_m\|^2_{L^\infty(0,T,L^2(0,1))}T + \sum_{i=1}^2  \|\varphi^i_m\|^2_{L^\infty(0,T,L^2(0,1))}T   \right].
		%\\
		%&\leq e^{DT} \left[ \|\sqrt{a}y_{m,x}(0)\|^2 + \sum_{i=1}^2  \|\sqrt{a}\varphi^i_{m,x}(0)\|^2 +\|h\|^2_{L^2(0,T,L^2(O))} + \sum_{i=1}^2  \|y_{i,d}\|^2_{L^2(0,T,L^2(O_i))} + \|y_m\|^2_{L^\infty(0,T,L^2(0,1))}T \right. \\
		%&\left. + \sum_{i=1}^2  \|\varphi^i_m\|^2_{L^\infty(0,T,L^2(0,1))}T \right] e^{D \int_0^t \left(  \|\sqrt{a}y_{m,x}\|^2 + \sum_{i=1}^2  \|\sqrt{a}\varphi^i_{m,x}\|^2 \right) ds}
	\end{split}
\end{equation*}

Finally, using estimate I,
\begin{equation*}
	\begin{split}
		&\|y_{m,t}\|^2_{L^2(0,T,L^2(0,1))} + \sum_{i=1}^2  \|\varphi^i_{m,t}\|^2_{L^2(0,T,L^2(0,1))} + \|\sqrt{a}y_{m,x}\|^2_{L^\infty(0,T,L^2(0,1))} + \sum_{i=1}^2  \|\sqrt{a}\varphi^i_{m,x}\|^2_{L^\infty(0,T,L^2(0,1))} \\
		&\leq e^{DT} \left[ \|\sqrt{a}y_{m,x}(0)\|^2_{L^2(0,1)} + \sum_{i=1}^2  \|\sqrt{a}\varphi^i_{m,x}(0)\|^2_{L^2(0,1)} +\|h\|^2_{L^2(0,T,L^2(O))}  
		\right. \\
		&\left. 
		+ \sum_{i=1}^2  \|y_{i,d}\|^2_{L^2(0,T,L^2(O_i))} + 3\mathcal{K}_1 T  \right] =:\mathcal{K}_2.
	\end{split}
\end{equation*}

Estimate III: Taking $w=-(a(x)y_{m,x})_x$ and $\hat{w}=-(a(x) \varphi^{i}_{m,x})_x$ in \eqref{eq:galerkin_system}, and proceeding as in the above estimates, we get
\begin{equation*}
	\begin{split}
		&-(y_{m,t}, (a y_{m,x})_x) - \sum_{i=1}^2  (\varphi^i_{m,t}, (a \varphi^i_{m,x})_x) + b(t) \left(\| (a y_{m,x})_x \|^2 + \sum_{i=1}^2  \| (a \varphi^i_{m,x})_x \|^2 \right) \\
		&- (c_1 \sqrt{a} y_{m,x},(a(x)y_{m,x})_x) - \sum_{i=1}^2 (c_2 \sqrt{a}\varphi^i_{m,x},(a(x) \varphi^{i}_{m,x})_x) \\
		& - (F(y_m, c(t)\sqrt{a}y_{m,x}),(a(x)y_{m,x})_x) - \sum_{i=1}^2 (g \varphi^i_m,(a(x) \varphi^{i}_{m,x})_x)  \leq C_\epsilon \|h\|^2 + \epsilon\|(a y_{m,x})_x\|^2 + C_\epsilon \sum_{i=1}^2  \|\varphi^i_m\|^2 \\
		& + \epsilon\|(a y_{m,x})_x\|^2 + C_\epsilon 2 \|y_m\|^2 + \epsilon\sum_{i=1}^2  \|(a\varphi^i_{m,x})_x\|^2 + C_\epsilon\sum_{i=1}^2  \|y_{i,d}\|^2 + \epsilon\sum_{i=1}^2  \|(a\varphi^i_{m,x})_x\|^2.
	\end{split}
\end{equation*}

Thus rearranging, using the lower-boundedness if $b(t)$, Cauchy-Schwarz inequality, the fact that $F$ is Lipschitz in both variables, that $c(t)$ is bounded, that $c_i$ and $g$ are bounded functions on $Q$ and Young's inequality, and proceeding as in estimate II we get that there exists a constants $D>0$ such that 
\begin{equation*}
	\begin{split}
		\|\sqrt{a} y_{m,x}\|^2_{L^\infty(0,T,L^2(0,1))}  + \sum_{i=1}^2  \|\sqrt{a} \varphi^i_{m,x}\|^2_{L^\infty(0,T,L^2(0,1))} +  \|(a y_{m,x})_x\|^2_{L^2(0,T,L^2(0,1))} + \sum_{i=1}^2  \|(a \varphi^i_{m,x})_x\|^2_{L^2(0,T,L^2(0,1))} \\
		\leq e^{DT} \left[ \|\sqrt{a} y_{m,x}(0) \|^2 + \sum_{i=1}^2  \|\sqrt{a} \varphi^i_{m,x}(0)\|^2 + \|h\|^2_{L^2(0,t,L^2(O))} + \sum_{i=1}^2  \|y_{i,d}\|^2_{L^2(0,t,L^2(O_i))} + 6 \mathcal{K}_1 \right] =: \mathcal{K}_3.
	\end{split}
\end{equation*}
Since $\mathcal{K}_1$, $\mathcal{K}_2$ and $\mathcal{K}_3$ do not depend on $m$, the three estimates above imply that the sequences $(y_m)$, $(\varphi^1_m)$ and $(\varphi^1_m)$ are bounded in 
$$
L^2(0,T,L^2(0,1)) \cap H^1(0,T,H^2_a(0,1)).
$$
Therefore, there exist subsequences $(y_{m_j})$, $(\varphi^i_{m_j})$, $i=1,2$, such that
$$
y_{m_j} \rightharpoonup y,\qquad \varphi^i_{m_j} \rightharpoonup \varphi^i, \qquad \text{as } j\to \infty, 
$$
weakly in $L^2(0,T,L^2(0,1)) \cap H^1(0,T,H^2_a(0,1))$.
In fact, since the immersion $H^2_a(0,1)$ in $H^1_a(0,1)$ is compact, by the theorem of Aubin-Lions we get that
$$
y_{m_j} \to y,\qquad \varphi^i_{m_j} \to \varphi^i, \qquad \text{as } j\to \infty, 
$$
strongly in $H^1_a(0,1)$. Using the continuity of $F$ and $F'$,
at least for a subsequence, we can pass to the limit in $m$ in  all the terms of the the approximate system \eqref{eq:galerkin_system}.

The uniqueness is proved by standard methods for nonlinear systems, taking into consideration the bounds established in estimates I to III for $y$, $y_x$ and $\varphi^i$ and $\varphi^i_x$, for $i=1,2$.\\

\end{itemize}
\bibliographystyle{abbrv}
\bibliography{referencias}

\begin{comment}

\end{comment}

\end{document}